\documentclass[opre,nonblindrev]{informs3opt} 
\usepackage{comment}

\usepackage{endnotes}
\let\footnote=\endnote

\usepackage{natbib}
\usepackage{natbib}
\bibpunct[, ]{(}{)}{,}{a}{}{,}%
\def\bibfont{\small}%
\usepackage{titlesec,soul}
\usepackage[titletoc,toc,title]{appendix}
\usepackage{graphicx}
\usepackage{subfigure}
\usepackage{multirow}

\usepackage[hidelinks,bookmarksnumbered=true]{hyperref}
\usepackage{todonotes,bm}
\usepackage{tikz}
\usepackage{url}
\usepackage{amssymb}
\usepackage{psfrag}
\usepackage{amsmath}
\usepackage{setspace}
\newcommand{\exclude}[1]{}
\usepackage{bm}
\usepackage[flushleft]{threeparttable}
\usepackage{algorithm}
\usepackage{algpseudocode}
\algdef{SE}[DOWHILE]{Do}{doWhile}{\algorithmicdo}[1]{\algorithmicwhile\ #1}%
\algnewcommand{\Or}{\textbf{or}}
\algnewcommand{\And}{\textbf{and}}

\usepackage{tikz}
\usetikzlibrary{shapes,decorations,arrows,calc,arrows.meta,fit,positioning}
\tikzset{
	-Latex,auto,node distance =1 cm and 1 cm,semithick,
	state/.style ={ellipse, draw, minimum width = 0.7 cm},
	point/.style = {circle, draw, inner sep=0.04cm,fill,node contents={}},
	bidirected/.style={Latex-Latex,dashed},
	el/.style = {inner sep=2pt, align=left, sloped}
}

\usepackage{thmtools} 
\usepackage{thm-restate}

\usepackage{cleveref}

\declaretheorem[name=Theorem]{theorem}
\declaretheorem[name=Proposition]{proposition}
\declaretheorem[name=Lemma]{lemma}
\declaretheorem[name=Claim]{claim}
\declaretheorem[name=Corollary]{corollary}

\declaretheorem[name=Definition]{definition}
\declaretheorem[name=Example]{example}

\def\Re{\mathbb{R}}

\def\hat{\widehat}

\def \V{\mathbf{ V}}
\def \G{\mathcal{G}}

\def \F{\mathcal{F}}

\def \R{\mathcal{R}}

\def\N{{\mathcal N}}
\def\F{{\mathcal F}}
\def\H{{\mathcal H}}

\def\X{{\mathcal X}}

\def\C{{\mathcal{C}}}

\def\Re{{\mathbb R}}

\def\Q{{\mathcal Q}}

\def\ri{{\mathbf{ri}}}

\DeclareMathOperator{\clconv}{\overline{conv}}

\DeclareMathOperator{\reccone}{rec}

\DeclareMathOperator{\Proj}{proj}

\DeclareMathOperator{\aff}{aff}
\DeclareMathOperator{\spa}{span}

\DeclareMathOperator{\supp}{supp}

\DeclareMathOperator{\Diag}{Diag}
\DeclareMathOperator{\conv}{conv}
\DeclareMathOperator{\tr}{tr}

\DeclareMathOperator{\rank}{rank}

\DeclareMathOperator{\ext}{ext}
\DeclareMathOperator{\rel}{rel}
\DeclareMathOperator{\opt}{opt}

\renewcommand{\S}{\mathcal{S}}
\renewcommand*{\qed}{\hfill\ensuremath{\square}}
\newcommand*{\qedA}{\hfill\ensuremath{\diamond}}

\usepackage{enumerate}

\allowdisplaybreaks

\begin{document}
	
	\RUNAUTHOR{Yongchun Li and Weijun Xie}
	
	\RUNTITLE{On the Exactness of Dantzig-Wolfe Relaxation for Rank-Constrained Optimization Problems}
	
	\TITLE{On the Exactness of Dantzig-Wolfe Relaxation for Rank-Constrained Optimization Problems}
	
	\ARTICLEAUTHORS{%
		\AUTHOR{Yongchun Li}
		\AFF{H. Milton Stewart School of Industrial and Systems Engineering, Georgia Institute of Technology, Atlanta, GA, USA, \EMAIL{ycli@gatech.edu}} 
	\AUTHOR{Weijun Xie}
	\AFF{H. Milton Stewart School of Industrial and Systems Engineering, Georgia Institute of Technology, Atlanta, GA, USA, \EMAIL{wxie@gatech.edu}}
} 

	\ABSTRACT{
In the rank-constrained optimization problem (RCOP), it minimizes a linear objective function over a  prespecified closed rank-constrained domain set and $m$ generic two-sided linear matrix inequalities. Motivated by the Dantzig-Wolfe (DW) decomposition, a popular approach of solving many nonconvex optimization problems, we investigate the strength of DW relaxation (DWR) of the RCOP, which admits the same formulation as RCOP except replacing the domain set by its closed convex hull.  Notably, our goal is to characterize conditions under which the DWR matches RCOP for any $m$ two-sided linear matrix inequalities. From the primal perspective, we develop the first-known simultaneously necessary and sufficient conditions that achieve: (i) \textit{extreme point exactness}--all the extreme points of the DWR feasible set  belong to that of the RCOP;  (ii) \textit{convex hull exactness}-- the DWR feasible set is identical to the closed convex hull of RCOP feasible set; and (iii) \textit{objective exactness}--the optimal values of the DWR and RCOP coincide. The proposed conditions unify, refine, and extend the existing exactness results in the quadratically constrained quadratic program (QCQP) and fair unsupervised learning. 
}
	\maketitle
\newpage

\section{Introduction}
This paper studies the  Rank-Constrained Optimization Problem (RCOP) of the form:
\begin{align} \label{eq_rank}
	\text{(RCOP)} \quad	\V_{\opt} :=\min_{{\bm X \in \mathcal{X}}}\left\{\langle\bm A_0, \bm X\rangle: b_i^l
	\le \langle \bm A_i, \bm X\rangle \le  b_i^u, \forall i \in [m] \right\}, 
\end{align}
where $\langle\cdot, \cdot \rangle$ denotes the inner product of two matrices, the rank-constrained domain set $\X$ is closed and finite-dimensional, technology matrices $\bm A_0$ and $\{\bm A_i\}_{i\in [m]}$  can be  non-symmetric, and  for each $i\in [m]$,  
the lower or upper bounds of the $i$th two-sided Linear Matrix Inequality (LMI) can be negative infinite or positive infinite, respectively (i.e., 
$-\infty\leq b_i^l\leq b_i^u\leq +\infty$).  
We let  $\tilde{m}$ denote  the  dimension of technology matrices $\{\bm A_i\}_{i\in [m]}$ in RCOP \eqref{eq_rank}, i.e., the number of linearly independent matrices in the set $\{\bm A_i\}_{i\in [m]}$, and we must have $\tilde{m}\le m$. 
We  use the LMI to  denote the two-sided LMI for notational convenience.
	
	We let  the domain set $\X$  be
	\begin{align}\label{eq_set}
		\mathcal{X}:= \{\bm X\in \Q:   \rank(\bm X)\le k, F_j(\bm X)  \le 0, \forall j\in [t]\} ,
	\end{align}
	where the matrix space $\Q$ can denote positive semidefinite matrix space $\S_+^n$, symmetric matrix space $\S^n$, or non-symmetric matrix space $\Re^{n \times p}$ with  $k\le n\le p$ being positive integers, and for each $j\in [t]$, function $F_j(\cdot):\Q \to \Re$ is continuous but can be possible nonconvex. Throughout the paper, we assume that the closed convex hull $\clconv(\X)$ has no line, which is satisfied by all the examples in this paper.
	This generic domain set $\X$ allows the proposed RCOP framework \eqref{eq_rank} to deliver significant modeling flexibility.
	For example, when domain set $\X:= \{ \bm X\in \S_{+}^{n+1}: \rank(\bm X)\le 1\}$, RCOP reduces to quadratically constrained quadratic program (QCQP) of matrix form. 
	Subsection \ref{subsec:scope} reveals several 
	interesting  machine learning and optimization examples that fall into RCOP \eqref{eq_rank}. 
	
	Albeit versatile, the underlying rank-$k$ constraint dramatically complicates RCOP \eqref{eq_rank}, which often turns out to be a nonconvex bilinear program. 
	In this work, we leverage the closed convex hull  of domain set $\X$, denoted by $\clconv(\X)$, to  obtain a stronger convex relaxation of the RCOP  \eqref{eq_rank}, which refers to the  Dantzig-Wolfe  Relaxation (DWR) in literature (see, e.g., \cite{conforti2014integer}). Thus, we consider the following relaxation problem for RCOP \eqref{eq_rank}:
	\begin{align} \label{eq_rel_rank}
		\text{(DWR)} \quad  \V_{\rel}:=\min_{{\bm X \in \clconv(\mathcal{X})}}\left\{\langle\bm A_0, \bm X \rangle: 
		b_i^l \le  \langle \bm A_i, \bm X\rangle \le b_i^u, \forall i \in [m] \right\} \le \V_{\opt}. 
	\end{align}
	Note that 
	for a rank-constrained domain set $\X$, different techniques have been investigated to derive its closed convex hull  $\clconv(\X)$, such as the perspective  technique \citep{bertsimas2021new,de2022explicit, wei2022ideal} and majorization technique \citep{kim2021convexification}.
	Building on these exciting results, in this work, we assume that the closed convex hull  $\clconv(\X)$ is given. 
	It follows that the DWR \eqref{eq_rel_rank} can be solved  by the off-the-shelf solvers such as Gurobi and Mosek or the Dantzig-Wolfe decomposition algorithm as long as the separation problem over the domain set $\X$ can be done effectively.
	Based on that, this paper studies the exactness of DWR \eqref{eq_rel_rank}, that is, 
	\begin{quote}	\it under what conditions the DWR  \eqref{eq_rel_rank} matches
		the RCOP \eqref{eq_rank} when intersecting the domain set $\X$ with any $m$ LMIs of dimension $\tilde m$. 
	\end{quote}
	
	For the sake of notational convenience, throughout this paper, let us denote the feasible sets  of RCOP  \eqref{eq_rank} and DWR \eqref{eq_rel_rank} as
	\begin{equation} \label{eq_2sets}
		{\C}:= \left\{\bm X\in \X:  b_i^l
		\le \langle \bm A_i, \bm X\rangle \le b_i^u, \forall i \in [m]  \right\},    \ \ 
		{\C_{\rel}}:=\left\{\bm X\in \clconv(\X):  b_i^l
		\le \langle \bm A_i, \bm X\rangle \le b_i^u, \forall i \in [m]  \right\},
	\end{equation}
	where $\clconv(\C) \subseteq \C_{\rel}$ always holds. In this way, RCOP  \eqref{eq_rank} and DWR \eqref{eq_rel_rank} can be equivalently recast as $\min_{\bm X\in \C} \langle \bm A_0, \bm X\rangle$ and $\min_{\bm X\in \C_{\rel}} \langle \bm A_0, \bm X\rangle$, respectively.  


	\vspace{-1em}
	\subsection{Three Notions of DWR Exactness: Extreme Point, Convex Hull, and Objective Value}	
	To study the strength of DWR \eqref{eq_rel_rank}, we define three notions of DWR exactness: 
	extreme point exactness,  convex hull exactness, and objective exactness, where the first two concepts center on the feasible set and the last one highlights the optimal value. 
	We propose simultaneously necessary and sufficient conditions under which DWR \eqref{eq_rel_rank} achieves the three notions of exactness for any $m$ LMIs of dimension $\tilde m$, respectively. 
	To be specific,  given the domain set $\X$, \textit{extreme point exactness}  represents that all the extreme points of the feasible set of DWR \eqref{eq_rel_rank} are contained in the feasible set of RCOP \eqref{eq_rank}, i.e., $\ext(\C_{\rel}) \subseteq \C$ for any  $\tilde m$-dimensional LMIs; \textit{convex hull exactness} 
	implies that the DWR feasible set coincides with the closed convex hull of RCOP feasible set, i.e., $\C_{\rel} =\clconv(\C)$ for any  $\tilde m$-dimensional LMIs; and
	\textit{objective exactness} is another commonly-used concept in literature, meaning that DWR \eqref{eq_rel_rank} yields the same objective value as the  original RCOP \eqref{eq_rank}, i.e., $\V_{\opt} = \V_{\rel}$ for any  $\tilde m$-dimensional LMIs  or some special RCOP families.
	

	The connections among these exactness notions are illustrated in  \autoref{fig:relation}. The convex hull exactness is the strongest notion and  implies the other two. The convex hull exactness reduces to the extreme point exactness for a bounded set $\C_{\rel}$.
	We show that  if  DWR \eqref{eq_rel_rank} yields a finite objective value, i.e.,  $\V_{\rel}>-\infty$, then the extreme point exactness is equivalent to the objective exactness for any linear objective function. If the objective exactness holds for any linear objective function, then the convex hull exactness naturally follows. 
	Additional objective exactness results are derived when we focus on two special yet  intriguing RCOP families as detailed in Section \ref{sec:obj}. 
	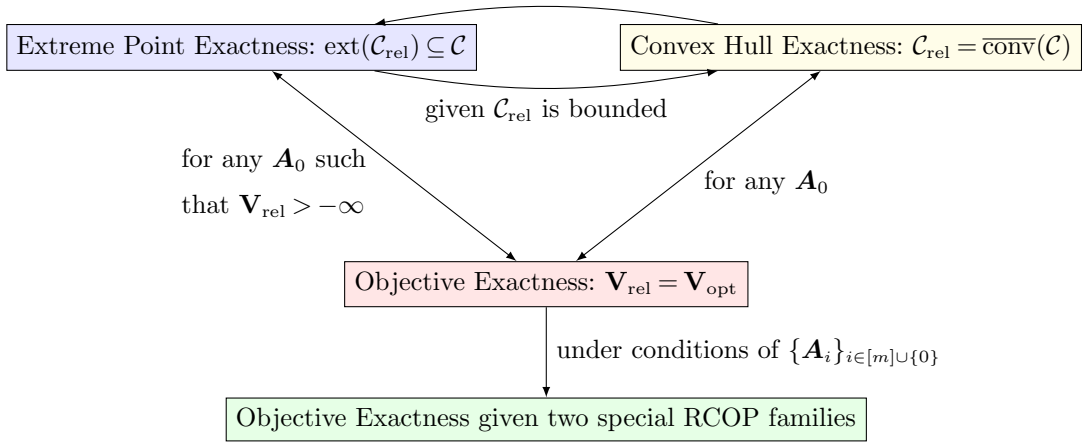
\begin{figure}[h]
		\centering
		\begin{tikzpicture}[scale=0.9, font=\small]
			\node[state,  fill=red!10, rectangle, align=center] (A) at (6, -1.5) {Objective Exactness: $\V_{\rel} =  \V_{\opt}$};
			\node[state,  fill=green!10, rectangle, align=center] (D) at (6, -3.5) {Objective Exactness  given two special RCOP families};
			\node[state,  fill=blue!10, rectangle, align=center] (B) at (1.5, 2) {Extreme Point Exactness: $\ext(\C_{\rel}) \subseteq \C$};
			\node[state,  fill=yellow!10, rectangle, align=center] (C) at (10.5, 2) {Convex Hull  Exactness: $\C_{\rel} = \clconv (\C)$};
			
			\draw[latex-latex] (B) to node[below,rotate=60] {} (A);
			\draw[latex-latex] (A) to node[right] {} (C);
			\draw[-latex] (A) to node[right] {
				under conditions of $\{\bm A_i\}_{i\in [m]\cup\{0\}}$
			} (D);
			\node[text width=2.5cm] at (2, 0) {for any $\bm A_0$ such that $\V_{\rel}>-\infty$};
			
			\draw[-latex] (B) to[bend right=10] node[below] {given $\C_{\rel}$ is bounded} (C);
			\draw[-latex] (C) to[bend right=10] node[above] {} (B);
			
			\node[text width=3cm] at (10, 0) {for any $\bm A_0$};
		\end{tikzpicture}
		\caption{The relations among three DWR exactness notions.}
		\label{fig:relation}
	\end{figure}

	\subsection{Scope and Flexibility of Our RCOP Framework  \eqref{eq_rank}} \label{subsec:scope}
	This subsection presents RCOP examples from  optimization, statistics, and machine
	learning fields. 
	
	\noindent\textbf{Quadratically Constrained Quadratic Program (QCQP) with $k=1$. }   The QCQP has been widely used in many application areas, including optimal power flow  \citep{josz2016ac,low2013convex}, sensor network problems \citep{bertrand2011consensus, khobahi2019optimized}, signal processing \citep{huang2014randomized,gharanjik2016iterative}, among others.
	The QCQP problem can be formulated in the following form:
	\begin{align}\label{qcqp}  
		\text{(QCQP)} \quad	\min_{\bm x \in \Re^n} \left\{ \bm x^{\top} \bm Q_0 \bm x + \bm q_0^{\top} \bm x:  b_i^l \le \bm x^{\top} \bm Q_i \bm x + \bm q_i^{\top} \bm x  \le b_i^u, \forall i\in [m]\right\},
	\end{align}
	where matrices $\bm Q_0, \bm Q_1, \cdots, 
	\bm  Q_{m} $ are symmetric but  may not be positive semidefinite.
	Introducing the matrix variable 
	$\bm X \in \S_{+}^{n+1} :=  \begin{pmatrix}
		1 & \bm x^{\top}\\
		\bm x & \bm x \bm x^{\top}
	\end{pmatrix}$, the resulting equivalent formulation of QCQP in matrix form can be viewed a special case of our RCOP  \eqref{eq_rank} {with $(m+1)$ LMIs} as shown below
	\begin{align} \label{eq_qcqp}
		\text{(QCQP)} 	\quad \min_{\bm X \in \X} \left\{\langle\bm A_0, \bm X\rangle: 
		b_i^l \le	\langle \bm A_i, \bm X\rangle \le  b_i^u, \forall i \in [m], X_{11}=1 \right\},    \X:= \{\bm X \in \S_+^{n+1}: \rank(\bm X)\le 1\},
	\end{align}
	where $ \bm A_0 = \begin{pmatrix}
		0& { \bm q_0^{\top}}/{2}\\
		{\bm q_0}/{2}  & \bm Q_0\\
	\end{pmatrix}$ and  ${\bm A}_{i} = \begin{pmatrix}
		0& { \bm q_i^{\top}}/{2}\\
		{\bm q_i}/{2}  & \bm Q_i\\
	\end{pmatrix}$ for each $i\in [m]$.  
	We notice that as $\clconv(\X) = \S_+^{n+1}$, the corresponding DWR of the QCQP \eqref{eq_qcqp} reduces to the well-known Semidefinite Programming (SDP) relaxation in literature, i.e,
	\begin{align*}
		\text{(DWR of QCQP \eqref{eq_qcqp})} \quad 	\min_{\bm X \in \clconv(\X)} \left\{\langle\bm A_0, \bm X\rangle: 
		b_i^l \le	\langle \bm A_i, \bm X\rangle \le  b_i^u, \forall i \in [m], X_{11}=1 \right\},  \clconv(\X):=\S_+^{n+1}.
	\end{align*}
	Note that we can strengthen the SDP relaxation by incorporating more 
	constraints into the domain set $\X$, which will be illustrated in Subsection \ref{sec:dual}.

	\noindent\textbf{Fair Unsupervised Learning with $k\ge 1$.} 
	Conventional unsupervised learning approaches (e.g., PCA) may  produce biased learning results against sensitive  attributes, such as gender, race, or education level. Fairness  has recently been introduced to these problems. For example,  Fair PCA (FPCA) was studied in \cite{samadi2018price,tantipongpipat2019multi} and Fair SVD (FVSD) was proposed in \cite{buet2022towards}. Formally, FPCA in \cite{tantipongpipat2019multi} is defined as
	\begin{align*}
		&\text{(FPCA)} \quad  \max_{(z, \bm X) \in \Re\times\X} \left\{z: z\le \langle  \bm A_i, \bm X \rangle, \forall i \in [m]\right\}, \ \ {\X:= \{\bm X \in \S_+^{n}: \rank(\bm X)\le k, ||\bm X||_2 \le 1 \}},
	\end{align*}
	where $||\cdot||_2$ denotes the spectral norm (i.e., the largest singular value) of a matrix and  matrices $\bm A_1, \cdots \bm A_m \in \S_+^n$ denote the sample covariance matrices from $m$ different groups. Note that FSVD has a similar formulation as FPCA except that in FSVD, $\bm A_1, \cdots \bm A_m \in \Re^{n\times p}$ denote non-symmetric data matrices and the corresponding domain set  $\X\subseteq \Re^{n\times p}$. 
	Simple calculations show that the closed convex hull of domain set $\X$ admits a closed form and thus, its DWR can be written as
	\begin{align*}
		\text{(DWR)}   \max_{(z, \bm X) \in \Re\times\clconv(\X)} \left\{z: z\le \langle  \bm A_i, \bm X \rangle, \forall i \in [m]\right\},  \clconv(\X):=\{ \bm X\in \S_{+}^n: \tr(\bm X)\le k, ||\bm X||_2 \le 1 \}.
	\end{align*}
	In a similar vein, the DWR of FSVD can be obtained. For  FPCA and FSVD, their DWR exactness results are delegated to Section \ref{sec:obj}.

	\subsection{Review of Relevant Work}
	As far as we are concerned, existing works on the DWR exactness in literature mainly study special cases of our RCOP \eqref{eq_rank}, i.e., QCQP with a rank-$k=1$ constraint and Fair PCA with a rank-$k\ge 1$ constraint. 
	
	\noindent\textbf{QCQP with $k=1$.} For the DWR exactness of QCQP \eqref{qcqp}, extensive research has focused on deriving sufficient conditions  given that data $\{(b_i^l, \bm A_i, b_i^u)\}_{i\in [m]}$ in the $m$ LMIs are specified beforehand. 
	Furthermore, analyzing the DWR exactness of QCQP from a dual perspective comes into the focus in literature, which often requires the Slater condition or more strict assumptions.
	Please see the excellent survey  by \cite{kilincc2021exactness} and the references therein. 
	From a primal and geometrically interpretable angle, this paper first develops conditions that are simultaneously necessary and sufficient to guarantee the DWR exactness of RCOP \eqref{eq_rank}  for any $m$ LMIs.

	\textit{QCQP with $m\le 2$ constraints.} Early works  have used the S-lemma to explore specific problems of QCQP \eqref{qcqp} that admit the DWR exactness, which can date back to \cite{yakubovich1971s}. 
	It is known that QCQP \eqref{qcqp} with one or two quadratic constraints can yield the DWR exactness under some mild assumptions. 
	For example, the DWR achieves the objective exactness for the Trusted Region Subproblem (TRS), Generalized TRS (GTRS), and two-sided GTRS under Slater condition \citep{yakubovich1971s,polik2007survey,wang2015strong}, a class of QCQP with $m=1$ quadratic constraint. Beyond the objective exactness, it is proven by \cite{ho2017second,kilincc2021exactness}  that convex hull exactness holds for TRS and GTRS under Slater condition.  When the quadratic coefficient matrix $\bm Q_1$ in QCQP \eqref{qcqp} is nonzero,  the convex hull exactness also holds for the two-sided GTRS
	without assuming the Slater condition \citep{joyce2021convex}.
	In general, QCQP \eqref{qcqp} with $m=2$ quadratic constraints may not  have DWR exactness. 
	Existing works have attempted to investigate sufficient conditions under which the objective exactness holds under this setting (see, e.g., the celebrated papers \cite{ye2003new, ben2014hidden}). Another interesting result in \cite{ santana2020convex} shows that the convex hull of a set that consists of a quadratic equality constraint and a bounded polyhedron is second-order cone representable.
	Our proposed conditions can reprove and connect these exactness results discussed above in a unified way
	and more importantly, successfully get rid of the Slater condition. 
	
	%

	\textit{QCQP with $m$ constraints.} A recent  thread of  work on  QCQP \eqref{qcqp} aims to develop sufficient conditions for its DWR exactness  given $m$ quadratic constraints in contrast to the previously discussed  ones, which address $m\le 2$ constraints. 
	When the QCQP \eqref{qcqp} admits  bipartite graph structures, several studies have proven the DWR objective exactness (see, e.g., \cite{azuma2022exact,sojoudi2014exactness,kim2003exact}). Besides,  for the diagonal QCQP in which matrices $\bm Q_0$ and $\{\bm Q_i\}_{i\in [m]}$ in \eqref{qcqp}  are diagonal,
	\cite{burer2020exact,locatelli2020kkt} proposed sufficient conditions of DWR objective exactness. 
	Particularly, \cite{burer2020exact} extended the results to the general QCQP \eqref{qcqp}, providing the first-known sufficient condition of DWR objective exactness.
	Recently, a seminal study on QCQP  proposed   sufficient conditions and necessary conditions for the objective exactness and convex hull exactness under the assumption that the Lagrangian dual set of QCQP is strictly feasible and polyhedral \cite{wang2022tightness}. 
	Their follow-up work relaxed the polyhedral assumption \citep{wang2020geometric} and proved that the proposed sufficient conditions are also necessary for the convex hull exactness whenever the polar cone of the Lagrange dual set is facially exposed. 
	
	Another separate yet related line of work on QCQP provides the (usually non-simultaneously)  necessary conditions and sufficient conditions for the rank-one generated (ROG) property of a convex positive semidefinite cone  studied in \cite{hildebrand2016spectrahedral}.  Notably, a convex positive semidefinite cone satisfies ROG property if it can be written as the conic hull of all its rank-one elements (see, e.g., \cite{argue2022necessary, kilincc2021exactness}). For a QCQP, 
	in \cite{argue2022necessary}, the authors showed that the ROG property implies the convex hull exactness but not vice versa. Since our set $\C_{\rel}$ defined in \eqref{eq_2sets} is beyond the conic set and $k\ge 1$, the DWR exactness notions do not necessarily encompass this property, and this paper provides a simultaneously necessary and sufficient condition for  convex hull exactness for a QCQP (i.e., RCOP with $k=1$). 
	Recently, the work in \cite{dey2019study} proved that under some conditions of QCQP input data,  the convex hull of QCQP feasible set could
	be polyhedral or second-order cone representable. 
	Very recently, in \cite{dey2022obtaining, blekherman2022aggregations}, the authors provided the convex hull of a QCQP with $m=3$ quadratic constraints via aggregations under mild conditions on matrices $\{\bm Q_i\}_{i\in [m]}$.
	Overall, most studies reviewed here have investigated the DWR exactness  for QCQP from its dual perspective and relied on Slater conditions and hence on strong duality \citep{azuma2022exact,burer2020exact,locatelli2020kkt,wang2022tightness,wang2020geometric}. 
	To the best of our knowledge, all the existing results for QCQP can neither be directly applied nor be simply extended to our RCOP \eqref{eq_rank}.

	\noindent\textbf{FPCA with $k\ge 1$.} The DWR exactness of the FPCA has been studied by \cite{tantipongpipat2019multi}, where the authors proved the extreme point  exactness for any $m=2$ different groups of covariance matrices. Our proposed conditions successfully extend this result to the convex hull exactness for any $\tilde m=2$ linearly dependent groups of covariance matrices. We go beyond the positive semidefinite set to study FSVD problem, and we 
	prove that for FSVD, the convex hull exactness holds for any $\tilde m=2$ linearly dependent groups of data matrices.

	\vspace{-0.5em}
	\subsection{Summary of Main Contributions and Organization}
	This paper studies three notions of DWR exactness on the general RCOP \eqref{eq_rank} for any $\tilde m$-dimensional LMIs and derives simultaneously necessary and sufficient conditions for each notion, where our conditions are primal-oriented,  geometrically interpretable, 
	domain dependent,  and Lagrangian dual free. The main contributions and an outline of the remaining of the paper  are summarized below:
	\begin{enumerate}[(i)]
		\item Section \ref{sec:face} investigates the extreme point exactness and convex hull exactness.
		\begin{itemize}
			\item  
			In Subsection \ref{sec:loc} and Subsection \ref{sec:ext}, the locations of extreme points in set $\C_{\rel}$ on  $\clconv(\X)$ motivate us to derive a simultaneously necessary and sufficient condition for extreme point exactness. 
			\item For the extreme rays of the recession cone in set $\C_{\rel}$, we give in Subsection \ref{sec:ch} their 
			precise locations on the recession cone of $\clconv(\X)$, which results in a sufficient condition for convex hull exactness. 
			Besides,  this condition  becomes necessary and sufficient when the domain set $\X$ is conic.
			\item Subsection
			\ref{sec:app} presents how our proposed conditions refine and extend the existing exactness results for some special cases of the QCQP as detailed in the first six problems in \autoref{table:egs}. This contributes to the literature of QCQP in that our exactness results for these problems get rid of Slater condition
		\end{itemize}

		\item Section \ref{sec:obj} 
		investigates four different classes of
		objective exactness. 
		\begin{itemize}
			\item As illustrated in \autoref{fig:relation},  the objective exactness for any $\bm A_0$ and  for any $\bm A_0$ satisfying $\V_{\rel}>-\infty$ reduce to the extreme point exactness and convex hull exactness, respectively, which allows us to directly extend the proposed conditions in Section \ref{sec:face} to objective exactness under these two settings in Subsection \ref{sec:any}.
			\item Subsections \ref{sec:bind} and \ref{sec:dual}  further relax the necessary and sufficient conditions of objective exactness under two special classes of the  linear objective function, which satisfy a fixed number of binding constraints and a fixed number of nonzero optimal Lagrangian multipliers, respectively.
			\item The proposed conditions for DWR objective exactness can be applied to the remaining problems in \autoref{table:egs} in the literature, where we recover the objective exactness for QCQP with multiple constraints,  prove the convex hull exactness of FSVD, and generalize the others by either providing stronger exactness results or using less strict assumptions.
		\end{itemize}
		\item  Section	\ref{sec:con} concludes this paper.	
	\end{enumerate}

	\begin{table}[h] 
		\centering
		\caption{Example Applications of  Our Proposed Conditions}
		\label{table:egs}
		\begin{threeparttable}
			\setlength{\tabcolsep}{2pt}\renewcommand{\arraystretch}{1.2}
			\begin{tabular}{c| c| c |c| c}
				\hline  
				\multicolumn{1}{c|}{Application}&	\multicolumn{1}{c|}{Problem} &	\multicolumn{1}{c|}{ Setting} & 	\multicolumn{1}{c|}{DWR Exactness} & Assumption   \\
				\hline
				\multirow{8}{2cm}{\centering QCQP \eqref{qcqp} with $m=1$ quadratic constraint ($k=1$)}	&  \multirow{2}{*}{QCQP-1} &  \multirow{2}{3.6cm}{single  constraint} &  extreme point  & \multirow{2}{*}{--\tnote{i}} \\
				& &   &  (\autoref{cor:qp1})  & \\
				
				\cline{2-5}
				& \multirow{2}{*}{TRS} &  \multirow{2}{3.6cm}{single ball constraint} & convex hull & \multirow{2}{*}{--}\\
				& &  & (\autoref{cor:trs})  & \\
				\cline{2-5}
				&\multirow{2}{*}{GTRS} & \multirow{2}{3.6cm}{single inequality constraint}  & convex hull &\multirow{2}{*}{--} \\
				& &   &  (\autoref{cor:gtrs})  & \\
				\cline{2-5}
				&\multirow{2}{*}{Two-sided GTRS} & \multirow{2}{3.6cm}{single two-sided quadratic constraint} & convex hull &\multirow{2}{4.1cm}{$\bm Q_1\neq \bm 0$; \\ $-\infty  < b_1^l \le b_1^u < +\infty$}  \\
				& &   &  (\autoref{cor:tgtrs})  & \\
				\hline
				\hline
				\multirow{8}{2cm}{\centering QCQP \eqref{qcqp} with $m=2$  quadratic constraints ($k=1$)}&\multirow{2}{*}{HQP-2} & \multirow{2}{3.6cm}{homogeneous QCQP} &  \multirow{2}{2.7cm}{\centering extreme point\\ (\autoref{cor:hqp})}&\multirow{2}{*}{--}\\
				& &   &    & \\
				\cline{2-5}
				&\multirow{2}{*}{--\tnote{i}}& \multirow{2}{3.7cm}{one constraint is not binding } & \multirow{2}{2.7cm}{\centering objective (\autoref{cor:bind})} &\multirow{2}{4.1cm}{$\V_{\rel}>-\infty$;\\ bounded optimal set\tnote{ii}}\\
				& &   &    & \\
				\cline{2-5}
				&\multirow{3}{*}{--} & \multirow{3}{3.6cm}{one optimal dual variable is zero} & \multirow{3}{2.7cm}{\centering objective (\autoref{cor:dual})} &\multirow{3}{4.1cm}{$\V_{\rel}>-\infty$;\\
					bounded optimal set;\\
					relaxed Slater condition} \\
				& &   &    & \\
				& &   &    & \\
				\hline
				\hline
				\multirow{6}{2.5cm}{\centering QCQP \eqref{qcqp} with $m$ inequality quadratic constraints ($k=1$)}&\multirow{3}{*}{--} & \multirow{3}{3.6cm}{all off-diagonal elements are sign-definite } &  \multirow{3}{2.5cm}{\centering objective \\(\autoref{cor:diag})}& \multirow{3}{4.1cm}{cyclic structures} \\
				& &   &    & \\
								& &   &    & \\
				\cline{2-5}	
				&\multirow{3}{*}{--} & \multirow{3}{3.7cm}{diagonal QCQP with sign-definite linear terms } &  \multirow{3}{2.5cm}{\centering objective\\ (\autoref{cor:diag})}&\multirow{3}{*}{--} \\
				& &   &    & \\
								& &   &    & \\
				\hline
				\hline
				\multirow{4}{2cm}{\centering Fair Unsupervised Learning ($k\ge 1$)} & \multirow{2}{*}{Fair PCA} &  \multirow{2}{3.6cm}{$m=2$ groups} &  convex hull  & \multirow{2}{*}{--}  \\
				& &   &  (\autoref{cor:fpca})  & \\
				\cline{2-5}
				& \multirow{2}{*}{Fair SVD} &  \multirow{2}{3.6cm}{$m=2$ groups} & convex hull & \multirow{2}{*}{--}  \\
				& &  &  (\autoref{cor:fsvd})  & \\
				\hline
			\end{tabular}%
			\vspace{-.1em}
			\begin{tablenotes}
				\item[i] ``--" denotes either an empty assumption or an unspecified name of the problem
				\item[ii] The set of all optimal solutions to DWR is bounded
			\end{tablenotes}  
		\end{threeparttable}
		\vspace{-.3em}
	\end{table}
	
%
	
	\noindent	\textit{Notation:} The following notation is used throughout the paper.
	For a closed convex set $D$, let $ \aff(D) $ denote the affine hull of set $D$, let $ \dim(D) $ denote the dimension of set $D$, let $\reccone(D)$ denote the recession cone of set $D$ when it is unbounded, and let $\ri(D)$ denote the relative interior of set $D$. Given $m$ matrices $\{\bm A_i\}_{i\in [m]}$, their linear span is defined by $\spa(\{\bm A_i\}_{i\in [m]}):=\{\sum_{i\in [m]}\alpha_i \bm A_i: \bm \alpha \in \Re^m \}$.
	For a matrix $\bm X$, let $||\bm X||_2$ denote its spectral norm (i.e., the largest singular value), let $||\bm X||_*$ denote its nuclear norm (i.e., the sum of singular values) and $||\bm X||_*=\tr(\bm X)$ when $\bm X\in \S_+^n$. Additional notation will be introduced later as needed.

	\section{A Geometric View  of DWR Exactness: Simultaneously Necessary and Sufficient Conditions for  Extreme Point Exactness and Convex Hull Exactness} \label{sec:face}
	This section investigates simultaneously necessary and sufficient conditions on
	the DWR  exactness from  a  geometric view of its feasible set: extreme points and convex hull. 
	Specifically, extreme point exactness guarantees all the extreme points in the DWR set $\C_{\rel}$  to belong to the original feasible set $\C$ of RCOP \eqref{eq_rank}, i.e., $\ext(\C_{\rel})\subseteq  \C$, and \textit{convex hull exactness} requires $\C_{\rel}$ to  be identical to the closed convex hull of set $\C$, i.e., $\C_{\rel} = \clconv(\C)$.
	It is worth mentioning that when the convex hull of set $\C$ is closed, we have $\clconv(\C)=\conv(\C)$. 
	
	\subsection{Where  Are Extreme Points of Set $\C_{\rel}$ Located at  Set $\clconv(\X)$?}\label{sec:loc}
	
	In this subsection, we study the locations
	of extreme points in set $\C_{\rel}$, an important insight for the extreme point exactness. 
	Observe that given a domain set $\X$, set $\C_{\rel}$ in \eqref{eq_2sets} is constructed by intersecting its closed convex hull (i.e., $\clconv(\X)$) with $m$ LMIs; therefore, we are interested in identifying which faces of the set $\clconv(\X)$ contain all extreme points in the intersection set $\C_{\rel}$. It is desired to make the result hold for any $m$ LMIs.

To begin with, let us define the  face and its dimension:

\begin{definition}[Face, Proper Face, Exposed Face, \& Dimension] \label{def:face}
	A convex subset $F$ of a closed convex set $D$ is called  a face of $D$ if for any line segment $[a,b]\subseteq D$ such that $F \cap (a,b)\neq \emptyset$, we have $[a,b]\subseteq F$. A nonempty face $F$ of $D$  is called a proper face if $F \subsetneq D$.
	If a face $F$ can be represented as the intersection of set $D$ with its supporting hyperplane,  it is called an exposed face. 	The dimension of a face  is equal to the dimension of its affine hull.
\end{definition}

Throughout the paper, for any closed convex set $D$ and a positive integer $d$, we let $\F^{d}(D)$ denote the collection of faces in this set with dimension no larger than  $d$.
Some faces with different dimensions are illustrated in \autoref{fig:face}. Furthermore, these faces are proper and exposed.

\begin{figure}[ht]
	\vspace{-1.5em}
	\centering
	\subfigure[Zero-dimensional face]{
		{\includegraphics[width=0.18\columnwidth]{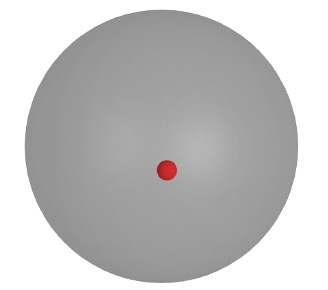}}\label{egd0}}
	~~~~~~~~~~~~
	\subfigure[One-dimensional face]{
		{\includegraphics[width=0.165\columnwidth]{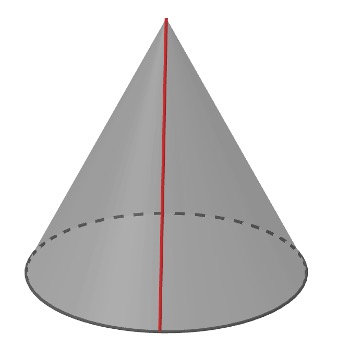}}	\label{egd1}}
	~~~~~~~~~~~~
	\subfigure[Two-dimensional face]{
		{\includegraphics[width=0.165\columnwidth]{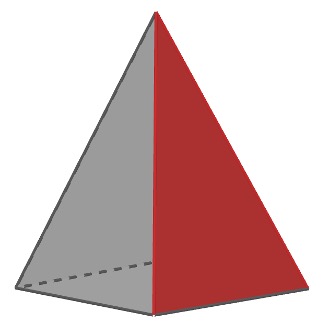}}	\label{egd2}}
	\caption{Examples of faces with different dimensions highlighted by red color}
	\label{fig:face}
	\vspace{-1em}
\end{figure}

The example below illustrates the basic idea of identifying where extreme points of set $\C_{\rel}$ are.
\begin{example}\label{eg1}\rm 
	Suppose domain set $\X:= \{\bm X\in \S_+^2: \rank(\bm X) \le 1, X_{12}=0, ||\bm X||_2\le 1 \}$ and $m=1$ LMI: $X_{11} \le 0.5$. Then we have $\clconv(\X) = \conv(\X) = \{\bm X\in \S_+^2: X_{12}=0, ||\bm X||_2\le 1 \}$ and the two sets defined in \eqref{eq_2sets} become
	\[\C :=  \{\bm  X \in \X: X_{11} \le 0.5\}, \ \ 
	\C_{\rel} :=  \{\bm  X \in \clconv(\X): X_{11} \le 0.5\}.\]
	
	For this example, we see that  the matrix in domain set $\X$ is of size two by two, positive semidefinite, and diagonal. Thus, introducing vector variable $\bm x:=(X_{11}, X_{22})$, we can equivalently recast $\X$ as a two-dimensional set, i.e.,  $\X:=\{\bm  x:=(X_{11}, X_{22})\in \Re_+^{2}: ||\bm x||_{0} \le 1, ||\bm x||_{\infty} \le 1\}$,
	as shown in Figure \ref{eg1x}, where the horizontal and vertical axes represent $X_{11}$ and $X_{22}$, respectively. Analogously, we can project the closed convex hull of the domain set $\X$ to the vector space as illustrated in Figure \ref{eg1conv}. 
	The solid red 
	line and the red shadow area in Figures \ref{eg1c} and \ref{eg1c1} represent sets $\C$ and $\C_{\rel}$ that are obtained by intersecting sets $\X$, $\clconv(\X)$ with a half-space $\{\bm x:=(X_{11}, X_{22}) \in \Re_+^2: X_{11}\le 0.5\}$, respectively. We observe in \autoref{fig:eg1} that given $m=1$ LMI, each  extreme point of  $\C_{\rel}$ is located at  a point or an edge in set $\clconv(\X)$ which is exactly a zero-- or one-dimensional face by \autoref{def:face}. Specifically, for the set $\C_{\rel}$ in Figure \ref{eg1c1}, extreme points $a_1$ and $a_2$ belong to  zero-dimensional faces of $\clconv(\X)$, and extreme points $a_3$ and $a_4$ belong to one-dimensional faces of $\clconv(\X)$.  Motivated by this observation in \autoref{eg1}, we will show the identity  between the dimension of $m$  LMIs (i.e., $\tilde{m}\le m$) and the largest dimension of faces in set $\clconv(\X)$ among which contain all extreme points in set  $\C_{\rel}$.
	\qedA
\end{example}

\begin{figure}[ht]
	\centering
	\vskip -0.15in
	\subfigure[$\X$]{
		{\includegraphics[width=0.18\columnwidth]{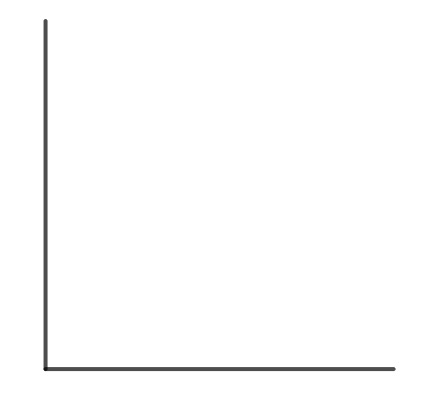}}	\label{eg1x}}
	~~~~~~
	\subfigure[$\clconv(\X)$]{
		{\includegraphics[width=0.18\columnwidth]{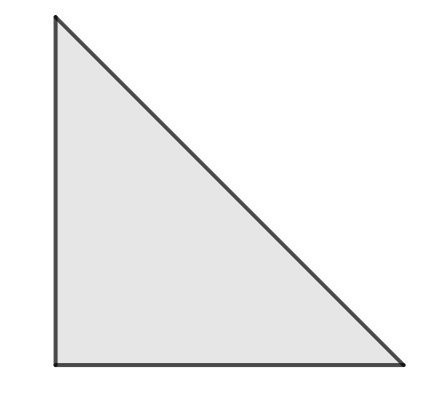}}\label{eg1conv}}
	~~~~~~
	\subfigure[$\C$]{
		{\includegraphics[width=0.17\columnwidth]{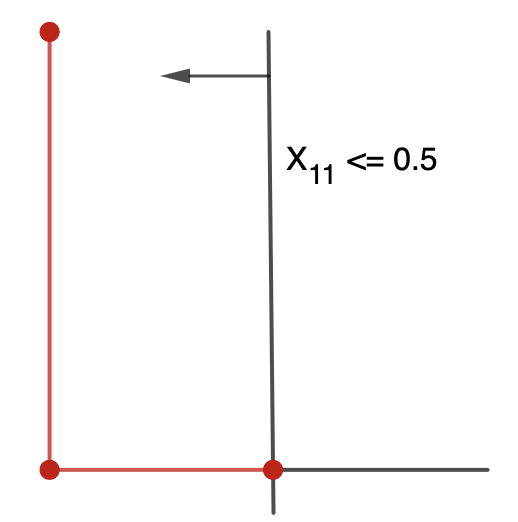}}	\label{eg1c}}
	~~~~~~
	\subfigure[$\C_{\rel}$]{
		{\includegraphics[width=0.18\columnwidth]{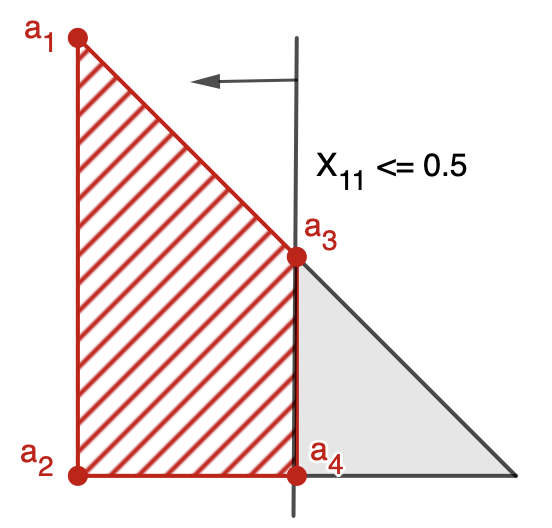}}	\label{eg1c1}}
	\caption{Illustration of sets in \autoref{eg1} and $\ext(\C_{\rel}) \nsubseteq \C$.}
	\label{fig:eg1}
	\vspace{-1.5em}
\end{figure}

\begin{restatable}{lemma}{lemext} \label{lem:genext}
	Given a  closed  convex set $D\subseteq\Re^{n}$,  for any $m$ linear constraints $\{\bm x \in \Re^{n}: b_i^l  \le \langle \bm a_i, \bm x \rangle \le b_i^u, \forall i\in [m]\}$,  each extreme point of the intersection  set $\H$ is contained in a face  of set $D$ with dimension at most $\tilde{m}$, where $\H:=D \cap \{\bm x \in \Re^{n}: b_i^l  \le \langle \bm a_i, \bm x \rangle \le b_i^u, \forall i\in [m]\}$ and $\tilde{m}\le m$ denotes the number of linearly independent vectors   $\{\bm a_i\}_{i\in [m]}$.
\end{restatable}

\begin{proof}
	We use induction to prove this result. 
	Suppose that set $D$ lies in a dimension-$d$ affine space.	When $d\le \tilde{m}$, it is trivial to verify the statement  since set $D$ itself is a $d$-dimensional face and $ \H \subseteq D$.
	Suppose  that the result holds for any $d\in[\bar{d}-1]$ with $\bar{d} \ge \tilde{m}+1$. Then we will  show that the result can be extended to the case when $d:=\bar{d}$ by contradiction. Let $\hat{\bm x}$  be an extreme point in set $ \H$. Suppose that the face  of the smallest dimension in set $D$ containing $\hat{\bm x}$ has a dimension $\hat{d}$  greater than $\tilde m$, denoted by $F^{\hat{d}}\subseteq D$ and $  \tilde{m}+1 \le \hat{d} \le \bar{d}$. Then there are two cases to be discussed depending on whether $\hat{d}=\bar{d}$ or not.
	\begin{enumerate}[(i)]
		\item Suppose $ \hat{d}< \bar{d}$. Since $F^{\hat{d}}$ is a $\hat{d}\le \bar{d}-1$ dimensional closed convex  set and $\hat{\bm x}$ is also an extreme point of the intersection set $F^{\hat{d}} \cap \{\bm x \in \Re^n: b_i^l  \le \langle \bm a_i, \bm x \rangle \le b_i^u, \forall i\in [m] \}$, by induction, $\hat{\bm x}$ belongs to a face of dimension up to $\tilde{m}$ in set $F^{\hat{d}}$. It is known that a face of a face of a closed convex set is also a face of this set (see section 18 in \cite{rockafellar2015convex}). Thus, the face in set $F^{\hat{d}}$ including the extreme point $\hat{\bm x}$  is also a face of set $D$ and is at most $\tilde{m}$-dimensional, which contradicts the fact that $\hat{d}\ge \tilde m +1$ is the smallest dimension of faces in set $D$  including  $\hat{\bm x}$.
		
		\item Suppose $\hat{d}=\bar{d}$. Since set $D$ itself is the one and only one $\bar{d}$-dimensional face, i.e., $F^{\hat{d}}=F^{\bar{d}}=D$, then $\hat{\bm x} $ does not belong to any proper face of $D$. According to proposition 3.1.5 in \cite{hiriart2004fundamentals}, the relative boundary of a closed convex set is equal to the union of all the exposed proper faces of this set. Therefore, $\hat{\bm x} $ must be in the relative interior of set $D$, and there exists a scalar $\alpha >0$ such that 
		\begin{align*}
			B(\hat{\bm x}, {\alpha})\cap \aff(D) \subseteq D,
		\end{align*}
		where $B(\hat{\bm x}, {\alpha}):=\{\bm x\in \Re^n: ||\hat{\bm x}-\bm x||_2 \le \alpha\}$.
		
		
		Given $n \ge \bar{d} \ge \tilde{m}+1$, for any $m$ vectors $\{\bm a_i\}_{i\in [m]}$ of dimension  $\tilde m$, there exists a nonzero vector $\bm y\in \Re^n$ satisfying $\langle\bm a_i, \bm y\rangle = 0$ for all $i\in [m]$. 	In addition, there exists a small scalar $0< \epsilon \le \alpha$ such that two vectors $\hat{\bm x} \pm \epsilon \bm y/||\bm y||_2 $  belong to set $ B(\hat{\bm x}, \alpha) \cap \aff(D)\subseteq D$. 
		Hence, both points $\hat{\bm x} \pm \epsilon \bm y/||\bm y||_2  $  belong to the intersection set $\H$  and we have
		$\hat{\bm x}  = \frac{1}{2}(\hat{\bm x} + \epsilon \bm y/||\bm y||_2) + \frac{1}{2} (\hat{\bm x} - \epsilon \bm y/||\bm y||_2),$
		which contradicts the fact that $\hat{\bm x}$  is an extreme point of set $\H$. 	This completes the proof. \hfill\qed
	\end{enumerate}
	%
\end{proof}

For 
ease of analysis,  \autoref{lem:genext} considers the intersection set in a vector space, which in fact, can cover any matrix-based set by reshaping a matrix into a long vector.
Since the result in \autoref{lem:genext} is independent of the vector length $n$ and holds for any vectors $\{\bm a_i\}_{i\in [m]}$, a natural generalization follows. 
\begin{corollary} \label{cor_lem:genext}
	Given a  closed convex set $D$ of matrix space $\Q$, for any $m$ LMIs $\{\bm X \in \Q: b_i^l  \le \langle \bm A_i, \bm X \rangle \le b_i^u, \forall i\in [m]\}$, suppose $\H:=D \cap \{\bm X \in \Q: b_i^l  \le \langle \bm A_i, \bm X \rangle \le b_i^u, \forall i\in [m]\}$. Then  each extreme point of the intersection set $\H$ is contained in a face  of set $D$ with dimension at most $\tilde{m}$, where $\tilde{m}$ denotes the number of linearly independent  matrices $\{\bm A_i\}_{i\in [m]}$.
\end{corollary}
The result in \autoref{cor_lem:genext} can be applied to the extreme point characterization of set  $\C_{\rel}$ by letting $D:=\clconv(\X)$ and $\H:=\C_{\rel}$. Given a domain set $\X$, when intersecting its closed convex hull with any $\tilde m$-dimensional LMIs,  \autoref{cor_lem:genext} implies that only those no larger than $\tilde m$-dimensional	faces of set $\clconv(\X)$, i.e.,  $\F^{\tilde m}(\clconv(\X))$, play a critical role in generating the 
extreme points in the DWR set $\C_{\rel}$. This motivates us to explore necessary and  sufficient conditions for 
the DWR exactness based on $\F^{\tilde m}(\clconv(\X))$. It should be emphasized that the results in \autoref{lem:genext} and \autoref{cor_lem:genext}  only require a closed convex set and thus can be applied to the closed convex hull of any domain set $\X$. 

\subsection{A Simultaneously Necessary and Sufficient Condition for Extreme Point Exactness} \label{sec:ext}
By exploring \autoref{cor_lem:genext}, this subsection presents a simultaneously necessary and sufficient condition under which the DWR problem \eqref{eq_rel_rank} achieves the extreme point exactness, i.e., $\ext(\C_{\rel})\subseteq \C$. 

We use \autoref{eg1} to reveal our main idea. 
We observe in \autoref{fig:eg1} that the extreme points $a_1, a_2, a_4$ of  set $\C_{\rel}$ belong to set $\C$ as the points and edges (i.e., zero and one-dimensional faces) where they locate in $\clconv(\X)$  belong to domain set $\X$. By contrast, the one-dimensional face in set $\clconv(\X)$ including the extreme point $a_3$ of  set $\C_{\rel}$ is not contained in domain set $\X$, and  $a_3$ does not belong to set $\C$, either. Motivated by this example, whether an extreme point of set $\C_{\rel}$ belongs to set $\C$ highly depends on whether its related face in set $\clconv(\X)$ is a subset of domain set $\X$, which offers a simultaneously necessary and sufficient condition  as below.

\begin{restatable}{theorem}{themext} \label{them:ext}
	Given a nonempty closed  domain set $\X$ with its closed convex hull $\clconv(\X)$ being line-free, the followings are equivalent:
	\begin{enumerate}[(a)]
		\item \textbf{Inclusive Face:} Any no larger than $\tilde{m}$-dimensional face of set $\clconv(\X)$ is contained in the domain set $\X$, i.e., $F \subseteq \X$ for all $F\in \F^{\tilde{m}}(\clconv(\X))$;
		\item \textbf{Extreme Point Exactness:} All the extreme points of set $\C_{\rel}$ belong to set $\C$ for any $m$ LMIs of dimension $\tilde{m}$ in RCOP \eqref{eq_rank}, i.e., $\ext(\C_{\rel}) \subseteq \C$.
	\end{enumerate}
\end{restatable}
\begin{proof}
	Let us prove the two directions of the equivalence, respectively.
	\begin{enumerate}[(i)]
		\item  $F \subseteq \X$ for all $F\in \F^{\tilde{m}}(\clconv(\X))$ $ \Longrightarrow \ext(\C_{\rel})\subseteq \C $.
		
		Let $\hat{\bm X}$ be an extreme point in set $\C_{\rel}$. According to \autoref{cor_lem:genext}, there exists a face $F$ in $\F^{\tilde{m}}(\clconv(\X))$ satisfying  $\hat{\bm X} \in F$. Since $F \subseteq \X$, it follows that $\hat{\bm X}\in \X\cap \C_{\rel} \subseteq \C$. Therefore,  $\ext(\C_{\rel})\subseteq \C $.
		
		\item  $\ext(\C_{\rel})\subseteq \C \Longrightarrow$  $F \subseteq \X$ for all $F\in \F^{\tilde{m}}(\clconv(\X))$.
		
		Suppose that $F \in \F^{\tilde{m}}(\clconv(\X))$ denotes a face of dimension at most $\tilde m$ in $\clconv(\X)$  not contained in set $\X$. Let $\hat{\bm X}$ be a point satisfying $\hat{\bm X} \in F \setminus \X $.
		Since face $F$ is no larger than $\tilde{m}$-dimension, we can construct  an $\tilde m$-dimensional LMI system  $\H:=\{\bm X\in \Q: \langle{\bm A}_i, \bm X\rangle = \langle{\bm A}_i, \hat{\bm X}\rangle, \forall i \in [m] \}$ such that
		the intersection $\aff(F)\cap \H =\{\hat{\bm X}\}=F\cap \H$ is zero-dimensional and thus a singleton.
		Then given the LMI system $\H$, the DWR feasible set defined in \eqref{eq_2sets} reduces to
		\[ \hat{\C}_{\rel}:=\{\bm X\in \clconv(\X): \langle{\bm A}_i, \bm X\rangle = \langle{\bm A}_i, \hat{\bm X}\rangle, \forall i \in [m] \} = \clconv(\X)\cap \H. \]
		We claim that $\hat{\bm X}\in \ext(\hat{\C}_{\rel})$.  If not,  then $\hat{\bm X}$ can be written as the convex combination below 
		$$\hat{\bm X}= \alpha \bm X_1 + (1-\alpha)\bm X_2, \ \ 0< \alpha <1,$$
		where $\bm X_1, \bm X_2 \in \hat{\C}_{\rel}$ are two distinct points. Since $\bm X_1, \bm X_2 \in \hat{\C}_{\rel} \subseteq \clconv(\X)$ and $\hat{\bm X}\in F$, according to \autoref{def:face} of a face, $\bm X_1$ and $\bm X_2$ must also belong to  the face $F$. In addition,  we have $\bm X_1, \bm X_2 \in \hat{\C}_{\rel} \subseteq \H$.  These results indicate that $\bm X_1, \bm X_2 \in F \cap \H$, contradicting  that  the intersection set $F \cap \H$ is a singleton.

		Since $\hat{\bm X}$ is an extreme point in set $\hat{\C}_{\rel}$, according to the presumption that $\ext(\hat{\C}_{\rel})\subseteq\hat{\C}$, where 
		\[ \hat{\C}:=\{\bm X\in \X: \langle{\bm A}_i, \bm X\rangle = \langle{\bm A}_i, \hat{\bm X}\rangle, \forall i \in [m] \} = \X \cap \H, \]
		we must have $\hat{\bm X}\in \hat{\C}\subseteq \X$, a contradiction.
		%
		Therefore,  $F$ must be contained in set $\X$. 
		\hfill\qed
	\end{enumerate}	
\end{proof}

We remark that (i) set $ \F^{\tilde{m}}(\clconv(\X))$ refers to the collection of all faces in set $\clconv(\X)$ up to dimension $\tilde{m}$, and any $\ge \tilde m$-dimensional face $F$ in set $\clconv(\X)$ is equal to $\clconv(\X)$ itself when the dimension of $\clconv(\X)$ is less than $\tilde{m}$; and (ii) Part (a) in \autoref{them:ext} provides a simultaneously necessary and sufficient condition of the extreme point exactness.

As aforementioned in \autoref{fig:relation}, the extreme point exactness sheds light on the convex hull exactness, where  a compact set $\C_{\rel}$ exactly equals the convex combination  of its extreme points. The relation between them motivates us to further investigate the  convex hull exactness by leveraging the  faces in set $\clconv(\X)$.


\subsection{One Sufficient and Two Simultaneously Necessary and Sufficient Conditions for Convex Hull Exactness}\label{sec:ch}
In this subsection, we study under which conditions the DWR \eqref{eq_rel_rank} attains convex hull exactness for any $\tilde m$-dimensional LMIs. 
As illustrated below, more than the condition in \autoref{them:ext}   may be needed 
to guarantee the  DWR  convex hull exactness.
\begin{example}\label{eg2}\rm
	Suppose domain set $\X:=\{\bm X \in \S_+^2: \rank(\bm X)\le 1\}$. Then $\conv(\X) = \clconv(\X) :=\S_+^2$. Let us construct the following intersection sets with $m=2$ LMIs.
	\begin{align*}
		&\C:=\{\bm X \in \S_+^2:  X_{12}= 0, X_{11}=1,\rank(\bm X) \le 1 \},  \ \ \C_{\rel}:=\{\bm X \in \S_+^n:  X_{12} =  0, X_{11} = 1 \}.
	\end{align*}
	In this example, the domain set is equivalent to $\X=\{\bm X \in \Re^{2\times 2}: X_{11} X_{22} = X_{12}^2, X_{11} \ge 0, X_{22}\ge 0 \}$. 
	We see that both sets $\X$ and $\clconv(\X)$ are unbounded, as shown in Figure \ref{eg2x} and Figure \ref{eg2convx}, respectively. We also see that 
	\begin{enumerate}[(i)]
		\item The domain set $\X$ is a three-dimensional unbounded surface, i.e., the boundary of its convex hull;
		\item Any zero or one-dimensional face in $\clconv(\X)$ is contained in the domain set $\X$ and $\clconv(\X)$ does not have any two-dimensional face, which is also trivially contained in set $\X$; and
		\item  The only three-dimensional face in $\clconv(\X)$ is itself and does not belong to the domain set $\X$.
	\end{enumerate}	
	Hence, the domain set $\X$ contains  any face of $\clconv(\X)$ with dimension no larger than two, i.e.,  $F \subseteq \X$ for all $F\in \F^{2}(\clconv(\X))$. 
	When intersecting set $\X$ with the following $m=2$ LMIs:
	$X_{12}= 0, X_{11}=1$, we have that $m=\tilde{m}=2$. 
	The resulting set $\C$ is a singleton, marked as a red solid point in Figure \ref{eg2c}, while the DWR set $\C_{\rel}$ illustrated in Figure \ref{eg2crel}  is a ray. 
	
	Note
	that the extreme point exactness holds, i.e., $\ext(\C_{\rel}) \subseteq C$,  as indicated in \autoref{them:ext}. However, 
	the closed convex hull of set $\C$ is itself and is not identical to set $\C_{\rel}$, i.e., the convex hull exactness fails, $\clconv(\C) \neq \C_{\rel}$.
	\qedA
\end{example}

\begin{figure}[ht]
	\centering
	\subfigure[$\X$]{
		{\includegraphics[width=0.15\columnwidth]{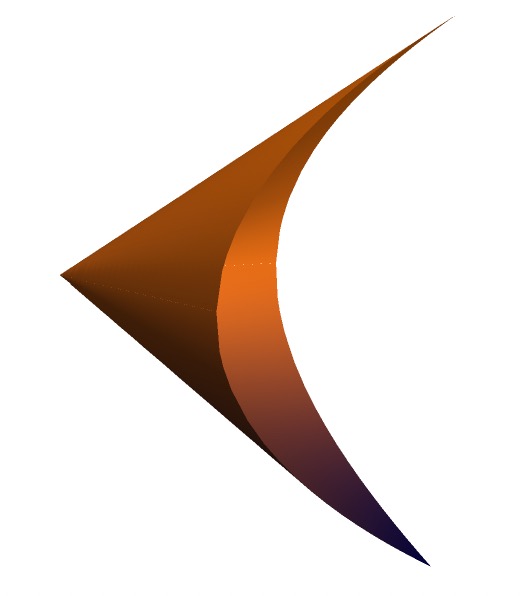}}	\label{eg2x}}
	~~~~~~
	\subfigure[$\clconv(\X)$]{
		{\includegraphics[width=0.15\columnwidth]{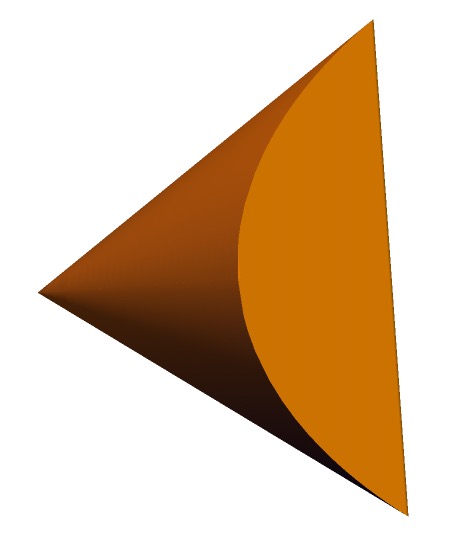}}	\label{eg2convx}}
	~~~~~~
	\subfigure[$\C$ ]{
		{\includegraphics[width=0.22\columnwidth]{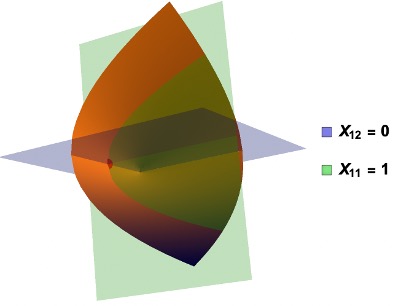}}	\label{eg2c}}
	~~~~~~
	\subfigure[$\C_{\rel}$ ]{
		{\includegraphics[width=0.22\columnwidth]{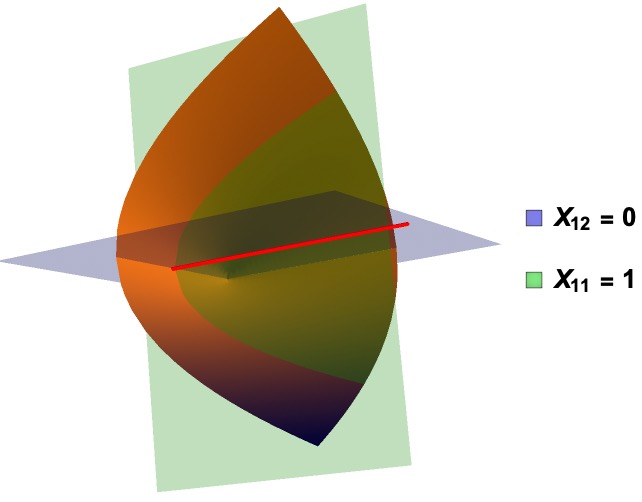}}	\label{eg2crel}}
	\caption{Illustration of Sets  in \autoref{eg2} with $\X=\{\bm X\in \S_+^2: \rank(\bm X) \le 1\}$ and $\ext(\C_{\rel}) \subseteq \C$,  $\clconv(\C) \neq \C_{\rel}$.  }
	\label{eq2:unbd}
\end{figure}

Inspired by \autoref{eg2}, the convex hull exactness requires stronger conditions than that of extreme point exactness in \autoref{them:ext}. 
One exemplary condition is that 
set $\C_{\rel}$ is bounded, which enables us to derive a simultaneously necessary and sufficient condition as below. Note that $\clconv(\C)=\conv(\C)$ when set $\C$ is compact.
\begin{restatable}{theorem}{themequalset}\label{them:exactch} 
	Given a nonempty closed domain set $\X$, 
	the followings  are equivalent:
	\begin{enumerate}[(a)]
		\item \textbf{Inclusive Face:} Any no larger than $\tilde{m}$-dimensional face of  set $\clconv(\X)$ is contained in the domain set $\X$, i.e.,  $F \subseteq \X$ for all $F\in \F^{\tilde{m}}(\clconv(\X))$;
		\item \textbf{Extreme Point Exactness:} All the extreme points of set $\C_{\rel}$ belong to set $\C$  for any $m$ LMIs of dimension $\tilde{m}$ in RCOP \eqref{eq_rank} such that set $\C_{\rel}$ is bounded, i.e., $\ext(\C_{\rel}) \subseteq \C$;
		\item \textbf{Convex Hull Exactness:} The feasible set $\C_{\rel}$ is equal to the  convex hull of set $\C$  for any $m$ LMIs of dimension $\tilde{m}$ in RCOP \eqref{eq_rank} such that set $\C_{\rel}$ is bounded, i.e., $\C_{\rel}=\conv(\C)$.
	\end{enumerate}
\end{restatable}
\begin{proof} 
	According to \autoref{them:ext}, we have  $(a) \iff (b)$. It remains to prove $(b) \iff (c)$.
	
	First, we observe that set $\C \subseteq \C_{\rel}$ is also compact and thus $\conv(\C)$ is compact. Besides, since the compact convex set $\C_{\rel}$ matches the convex hull of its extreme points, i.e., $\C_{\rel} = \conv(\ext(\C_{\rel}))$, it is evident that $\C_{\rel}=\conv(\C) \iff \ext(\C_{\rel}) \subseteq \C$. This completes the proof. \qed
\end{proof}

A considerable amount of literature has investigated sufficient conditions for DWR exactness (see, e.g., \cite{azuma2022exact, kilincc2021exactness,polik2007survey}); however, such  studies can only deal with QCQP, a particular case of our RCOP \eqref{eq_rank}. Compared to them, the result in \autoref{them:exactch} has the following favorable aspects: 
\begin{enumerate}[(i)]
	\item 
	When set $\C_{\rel}$ is bounded, Part (a) in \autoref{them:exactch} gives the first-known simultaneously necessary and sufficient condition of all the three notions of DWR exactness for any $\tilde m$-dimensional LMIs, 
	which is beyond the scope of QCQP and can be useful to refine many existing results in the following subsection;
	\item From a novel yet geometrically interpretable perspective, we show that the DWR exactness only depends on some crucial faces in set $\X$, i.e., those containing extreme points of set $\C_{\rel}$;
	\item Most exactness conditions proposed in the literature rely on some restricted assumptions, e.g., the Lagrangian dual set of DWR \eqref{eq_rel_rank} needs to be  polyhedral in \cite{kilincc2021exactness}, which prevents the results from being applied to general QCQP.
	The assumption adopted in  our \autoref{them:exactch} is quite mild, namely the compactness of set $\C_{\rel}$; and
	\item Finally, the result in \autoref{them:exactch} is quite general and can be applied to any closed domain set $\X$.
\end{enumerate}

When dealing with the  convex hull exactness in \autoref{them:exactch}, we assume that set $\C_{\rel}$  is bounded. 
In \autoref{eg2}, it implies that the necessary and sufficient condition in \autoref{them:exactch} may be insufficient to guarantee the convex hull exactness of an unbounded set $\C_{\rel}$. 
To be specific, the unbounded set $\C_{\rel}$ is a half-line while  set $\C$ is a singleton in \autoref{eg2}.
The following theorem provides a sufficient condition for the convex hull exactness under the unbounded setting.
	The extreme ray defined below is the key ingredient that distinguishes an unbounded closed convex set from a compact one. 
	%
	
	\begin{definition}[Recession Cone, Extreme Ray, \& Extreme Direction]  \label{def:ext}
		For  a closed line-free convex set $D$, the recession cone of $D$ is a closed convex cone containing all the directions in set $D$, denoted by $\reccone(D)$. 
		The extreme ray of the recession cone $\reccone(D)$ is a half-line face emanating from the origin. 
		In addition, an extreme direction in set $D$ is the direction of an extreme ray of its recession cone $\reccone(D)$. 
	\end{definition}
	
	Using the representation theorem 18.5 in \cite{rockafellar2015convex}, set $\C_{\rel}$ is equal to the closed convex hull of all the extreme points and extreme directions of $\C_{\rel}$. 
	Therefore, we next explore the properties of  extreme rays of  recession cone of set $\C_{\rel}$. Similar to \autoref{lem:genext}, it is intuitive to  study where the extreme rays  of  $\reccone(\C_{\rel})$ are located at the recession cone of $\clconv(\X)$, which is illustrated in the example below.
	
	
	
	\begin{example}\label{eg:ray} \rm
		Suppose  domain set $\X:= \{\bm X\in \S_+^2: \rank(\bm X) \le 1, X_{12}=0\}$ and there is $m=1$ LMI $X_{11} \le X_{22}$. Then we have $\clconv(\X) =\conv(\X)=\{\bm X\in \S_+^2:  X_{12}=0\}$ and two feasible sets defined in \eqref{eq_2sets} become
		\[\C :=  \{\bm  X \in \X: X_{11} \le X_{22}\}, \ \ 
		\C_{\rel} :=  \{\bm  X \in \clconv(\X): X_{11} \le X_{22}\}.\]
		Following \autoref{eg1}, the domain set $\X$ here  can be recast into the two-dimensional vector space, i.e., $\X:=\{(X_{11}, X_{22}) \in \Re^2_+: X_{11}X_{22}=0 \}$ and the LMI is reduced to $X_{11}\le X_{22}$.
		Therefore, we have $\clconv(\X)=\conv(\X):=
		\Re^2_+
		$. As shown in Figure \ref{egrayx} and Figure \ref{egrayconv}, sets $\X$ and $\conv(\X)$ in this example are unbounded.
		The corresponding feasible  set $\C$ is the red vertical line in Figure \ref{egrayc}, and the red shadow area in Figure \ref{egrayc1} corresponds to set $\C_{\rel}$. 
		
		It is seen that (i) any zero-- or one-dimensional face of $\conv(\X)$ belongs to $\X$, i.e., $F \subseteq \X$ for all $F\in \F^{1}(\conv(\X))$; (ii)  for the conic line-free set $\C_{\rel}$ in Figure \ref{egrayc1}, the recession cone of set $\C_{\rel}$  is itself, and its extreme rays belong to no larger than two-dimensional faces of $\clconv(\X)$. Particularly, the extreme ray $\{(X_{11}, X_{22}) \in \Re^2_+: X_{11}=X_{22}\}$ of set $\C_{\rel}$ is contained in a two-dimensional face of set $\clconv(\X)$; and
		(iii) in this example, set $\C_{\rel}$ achieves the extreme point exactness, whereas the convex hull exactness fails. \qedA
	\end{example}
	\begin{figure}[ht]
		\centering
		\vskip -0.05in
		\subfigure[$\X$]{
			{\includegraphics[width=0.14\columnwidth]{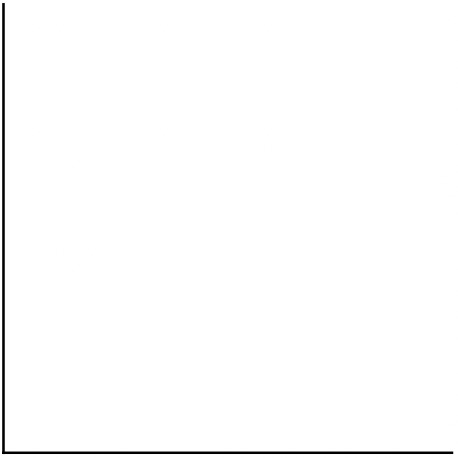}}	\label{egrayx}}
		~~~~~~~~~
		\subfigure[$\clconv(\X)$]{
			{\includegraphics[width=0.14\columnwidth]{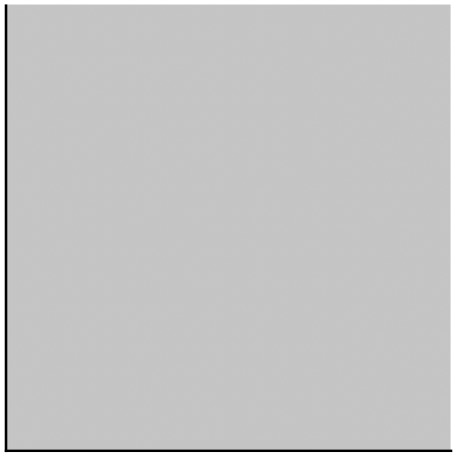}}\label{egrayconv}}
		~~~~~~~~~
		\subfigure[$\C$]{
			{\includegraphics[width=0.205\columnwidth]{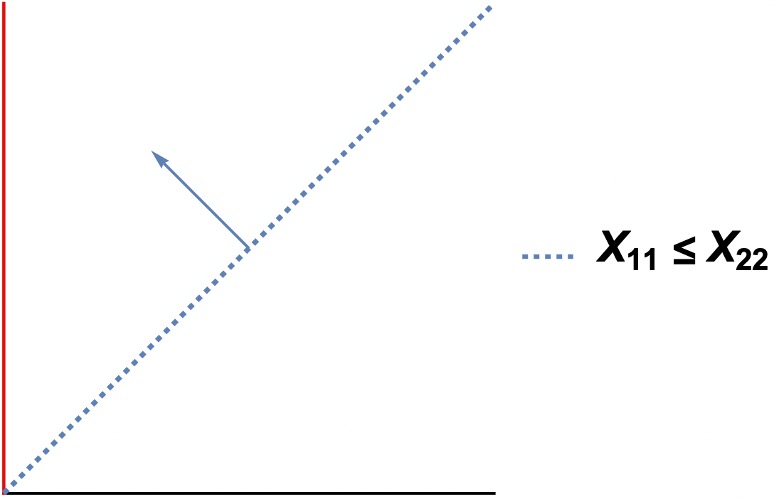}}	\label{egrayc}}
		~~~~~~~~~
		\subfigure[$\C_{\rel}$]{
			{\includegraphics[width=0.205\columnwidth]{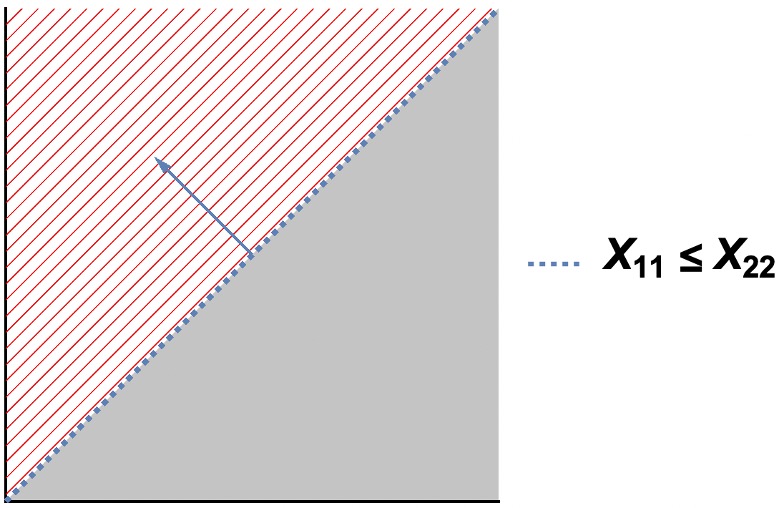}}	\label{egrayc1}}
		\caption{Illustration of Sets in \autoref{eg:ray} and $\ext(\C_{\rel}) \subseteq \C$, $\C_{\rel} = \clconv(\C')$ }
		\label{fig:egray}
		\vskip -0.1in
	\end{figure}

	%

	\autoref{eg:ray} motivates us the following technical result.
	\begin{restatable}{lemma}{lemext} \label{lem:genray}
		Given a  closed line-free convex set $D \subseteq \Re^n$, for any $m$ linear constraints $\{\bm x \in \Re^{n}: b_i^l  \le \langle \bm a_i, \bm x \rangle \le b_i^u, \forall i\in [m]\}$, suppose that  the intersection set $\H:=D \cap \{\bm x \in \Re^{n}: b_i^l  \le \langle \bm a_i, \bm x \rangle \le b_i^u, \forall i\in [m]\}$ is unbounded.
		Then each extreme ray  of the recession cone of set $\H$ is contained in a face of the recession cone of  set $D$ with dimension no larger than $\tilde{m}+1$, where $\tilde{m}\le m$ denotes the number of linearly independent vectors $\{\bm a_i\}_{i\in [m]}$.
	\end{restatable}
	\begin{proof}
		We will show that any extreme ray in the recession cone of  $\H$ is contained in $\F^{\tilde{m}+1}(\reccone(D))$.
		Similar to \autoref{lem:genext},  we  use induction to prove it. Let $d$ denote the dimension of the recession cone of set $D$. First, the result trivially holds if $d \le \tilde{m}+1$.
		
		Suppose that the result holds for any $d\in [\bar{d}-1]$ with $\bar{d}\ge \tilde{m}+2$.
		Then let us prove the case of $d:=\bar{d}$ by contradiction. 
	If there is an extreme ray $F$ in $\reccone(\H)$ that is not contained in  $\F^{\tilde{m}+1}(\reccone(D))$, then we let $\hat{d}$ denote the smallest dimension among all the faces in $\reccone(D)$ that contain $F$. Thus, we must have $ \tilde{m}+1 < \hat{d}  \le \bar{d}$. We split the remaining proof into two parts depending on whether $\hat{d}< \bar{d}$ or $\hat{d}= \bar{d}$. 
	\begin{enumerate}[(i)]
		\item If $\hat{d}< \bar{d}$, then following the induction and the similar analysis in \autoref{lem:genext}, the result holds.
		
		\item If $\hat{d}= \bar{d}$, then by the definition of $\hat{d}$, the face $F$ must intersect the relative interior of the recession cone of set $D$, i.e., $F\cap \ri(\reccone(D)) \neq \emptyset$.
		Suppose that $\bm x_1, \bm x_2$ are two distinct points in $F\cap \ri(\reccone(D))$ (since $F$ is a closed convex set and $\ri(\reccone(D))$ is an open convex set), i.e., $\bm x_1, \bm x_2 \in F\cap \ri(\reccone(D)),\bm x_1\neq \bm x_2 $.
		Given $n\ge \bar{d}  > \tilde{m}+1$, for any $m$ vectors $\{\bm a_i\}_{i\in [m]}$ of dimension $\tilde m$, there exists a nonzero vector $\bm y  \in \Re^n$ such that 
		\[\langle \bm a_i, \bm y\rangle = 0, \forall i\in [m],  \ \  \langle \bm x_1-\bm x_2, \bm y\rangle = 0,\]
		where the last equality implies that $\bm y$ is orthogonal to the extreme ray  $F$, i.e., $\bm y \perp F$. Since the point $\bm x_1 \in F\cap \ri(\reccone(D))$, there exists a small scalar $\epsilon >0$ such that $\bm x_1\pm \epsilon \bm y\in \reccone(D)$. 
		

		Besides, it is known (see, e.g., \cite{luc1990recession}) that the recession cone of the intersection set  $\H$ is equal to intersecting recession cones of set $D$ and the linear system $\{\bm x \in \Re^{n}: b_i^l  \le \langle \bm a_i, \bm x \rangle \le b_i^u, \forall i\in [m]\}$.
		We have $ \langle \bm a_i , \bm x_1 \pm \epsilon \bm y \rangle = \langle \bm a_i,  \bm x_1 \rangle$ for all $i\in [m]$ and $\bm x_1 \in \F$,  implying that points $\bm x_1 \pm \epsilon \bm y$ also belong to the recession cone of the linear system $\{\bm x \in \Re^{n}: b_i^l  \le \langle \bm a_i, \bm x \rangle \le b_i^u, \forall i\in [m]\}$.
		It follows that $\bm x_1 \pm \epsilon \bm y \in \reccone(\H)$.
		According to \autoref{def:ext}, any extreme ray in the recession cone $\reccone(\H)$ is exactly a  half-line face in this cone, i.e., $F$ is a one-dimensional face. Thus,
		$\bm x_1 = \frac{1}{2} (\bm x_1 + \epsilon \bm y)+\frac{1}{2} (\bm x_1 - \epsilon \bm y)$, implying that $\bm x_1\pm \epsilon \bm y \in F$ according to \autoref{def:face} of a face. This contradicts the fact $\bm y \perp F$.
	\end{enumerate}
	
	This completes the proof. \qed
\end{proof}

Similar to \autoref{lem:genext}, the result in \autoref{lem:genray} can be generalized to any matrix-based set. 
\begin{corollary} \label{cor_lem:genray}
	Given a closed line-free convex set $D$ of matrix space $\Q$,  for any $m$ LMIs $\{\bm X \in \Q: b_i^l  \le \langle \bm A_i, \bm X \rangle \le b_i^u, \forall i\in [m]\}$, suppose that  the intersection set $\H:=D \cap \{\bm X \in \Q: b_i^l  \le \langle \bm A_i, \bm X \rangle \le b_i^u, \forall i\in [m]\}$ is unbounded.
	Then each extreme ray of the recession cone of set $\H$  is contained in a face  of  the recession cone of set $D$ with dimension no larger than $\tilde{m}+1$,  where $\tilde{m}\le m$ denotes the dimension of matrices $\{\bm A_i\}_{i\in [m]}$.
\end{corollary}

When intersecting the set $\clconv(\X)$ with $\tilde m$-dimensional LMIs, \autoref{lem:genext} together with \autoref{def:ext} indicates that some special faces in the recession cone of set  $\clconv(\X)$ play an important role in determining the extreme rays of the intersection set $\C_{\rel}$. Next, we show a sufficient condition under which the convex hull exactness holds.
%
\begin{restatable}{theorem}{themext} \label{them:ubdconv}
	Given a nonempty closed  domain set $\X$ with its closed convex hull $\clconv(\X)$ being line-free, the following statement (a) implies statement  (b):
	\begin{enumerate}[(a)]
		\item \textbf{Inclusive Face:}  The (Minkowski) sum of any no larger than $\tilde{m}$-dimensional face of set $\clconv(\X)$ and any no larger than $(\tilde{m}+1)$-dimensional face of recession cone $\reccone(\clconv(\X))$  is contained in the domain set $\X$, i.e.,  $F+ \hat{F} \subseteq \X$ for all $F\in \F^{\tilde{m}}(\clconv(\X))$ and $\hat{F}\in \F^{\tilde{m}+1}(\reccone(\clconv(\X)))$;
		\item \textbf{Convex Hull Exactness:} The feasible set $\C_{\rel}$ is equal to the closed convex hull of set $\C$ for any $m$ LMIs of dimension $\tilde{m}$  in RCOP \eqref{eq_rank}, i.e., $\C_{\rel}=\clconv(\C)$.
	\end{enumerate}
\end{restatable}
\begin{proof}
	According to Part (a) and \autoref{them:ext},  we have that each extreme point of set $\C_{\rel}$ belongs to  set $\C$. 
	
	Given a line-free set $\clconv(\X)$, the intersection set $\C_{\rel}$ is closed, convex, and line-free. For any extreme direction $\bm D$ in the line-free set $\C_{\rel}$, according to \autoref{def:ext}, there is an extreme ray $\hat{S}$ in the recession cone of set $\C_{\rel}$, i.e.,   $\hat{S}:=\{\alpha \bm D: \alpha \ge 0\}\subseteq \reccone(\C_{\rel})$.
	We show in  \autoref{lem:genray} and its corollary that the extreme ray   $\hat{S}$ must  belong to a face $\hat{F}$ in
	$\F^{\tilde{m}+1}(\reccone(\clconv(\X)))$, i.e., $\hat{S}\subseteq \hat{F}$. Given the presumption on domain set $\X$ in Part (a), 
	there is an extreme point $\hat{\bm X}$ in set $\C$ such that  $\hat{\bm X}+ \hat{S} :=\{ \hat{\bm X}+\alpha \bm D: \alpha \ge 0\} \subseteq \hat{\bm X}+ \hat{F}  \subseteq \X$ holds.
	
	Besides, it is known (see, e.g., \cite{luc1990recession}) that the recession cone of the intersection of closed convex sets is equal to intersecting recession cones.
	Since the extreme direction $\bm D$ also lies in the recession cone of the $m$ LMIs, we conclude that $\{ \hat{\bm X}+\alpha \bm D: \alpha \ge 0\}$ belongs to set $\C$. Therefore, $\bm D$ is also an extreme direction in set $\clconv(\C)$ as $\clconv(\C) \subseteq \C_{\rel}$ always holds.
	Using the representation theorem 18.5 in \cite{rockafellar2015convex}, the set $\C_{\rel}$ is equal to the closed convex hull of all extreme points and extreme directions of $\C_{\rel}$. Therefore, we have that $\C_{\rel}=\clconv(\C)$.\qed
\end{proof}

The following example shows that the sufficient condition  in \autoref{them:ubdconv}, unfortunately, is not necessary for the convex hull exactness.
\begin{example}\label{eq:counter} \rm
	Suppose that the domain set $\X:=\{ \bm x\in \R_{+}^2: \rank(\bm x)\le 1, (x_{1}-1)^2+ (x_{2}-1)^2 \ge 0.5^2\}$. Then the domain set is equivalent to 
	$\X:= \Re_+^2 \setminus \{\bm x \in \Re_+^2:  (x_1-1)^2+(x_2-1)^2 <0.5^2 \}$. That is, the domain set $\X$ is defined as removing an open ball from the interior of two-dimensional nonnegative orthant. Hence, $\clconv(\X):= \Re_+^2$.  
	Since set $\clconv(\X)$ itself is a  two-dimensional face and is equal to the recession cone, 
	our sufficient condition in \autoref{them:ubdconv} becomes that $F \subseteq \X$ for all $F\in \F^{\tilde{m}+1}(\clconv(\X))$.  
	We see that the domain set $\X$ does not contain the two-dimensional face in $\clconv(\X)$. 
	That is, the condition  fails even when $\tilde{m}=1$. However, the convex hull exactness always holds when intersecting sets $\X$ and $\clconv(\X)$ with any $m=\tilde m=1$ LMI, respectively. This shows that the sufficient condition  in \autoref{them:ubdconv} may not be necessary.
	\qedA
\end{example}

\begin{figure}[htbp]
	\centering
	\vskip -0.2in
	\subfigure[$\X$ ]{
		{\includegraphics[width=0.17\columnwidth]{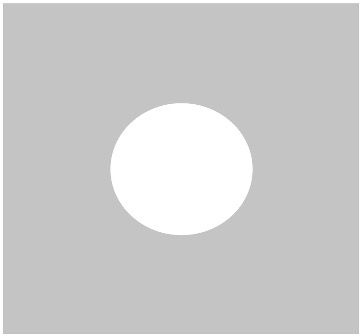}}	\label{eg4c}}
	\hspace{3cm}
	\subfigure[$\clconv(\X)$ ]{
		{\includegraphics[width=0.17\columnwidth]{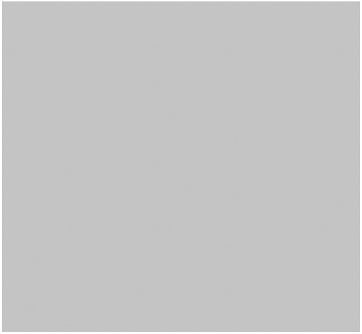}}	\label{eg4conv}}
	\caption{Illustration of Sets in \autoref{eq:counter}  and $\C_{\rel} = \clconv(\C)$ for any $\tilde m=1$-dimensional linear constraints.  }
	\label{fig:eg4}
	\vskip -0.1in
\end{figure}

Interestingly, we show that the sufficient condition in \autoref{them:ubdconv} becomes necessary when the domain set $\X$ is closed and conic. For example, the domain set $\X$ of QCQP only has a rank-1 constraint and is thus conic. 
In this case, we have	$\reccone(\clconv(\X)) = \clconv(\X)$,
which simplifies the sufficient condition in \autoref{them:ubdconv}  to be  that $F \subseteq \X$ for all $F\in \F^{\tilde{m}+1}(\clconv(\X))$.
\begin{theorem}\label{them:qcqpch}
	Given a nonempty closed conic domain set $\X$, i.e., for any $\alpha\geq 0$ and $\bm X\in \X$, we have $\alpha \bm X \in \X$, and its closed convex hull $\clconv(\X)$ being line-free, then the followings are equivalent:
	\begin{enumerate}[(a)]
		\item \textbf{Inclusive Face:} Any no larger than $(\tilde{m}+1)$-dimensional face of set $\clconv(\X)$ is contained in the domain set $\X$, i.e.,  $F \subseteq \X$ for all $F\in \F^{\tilde{m}+1}(\clconv(\X))$;
		\item \textbf{Convex Hull Exactness:} The feasible set $\C_{\rel}$ is equal to the closed convex hull of set $\C$ for any $m$ LMIs of dimension $\tilde{m}$ in RCOP \eqref{eq_rank}, i.e., $\C_{\rel}=\clconv(\C)$.
	\end{enumerate}
\end{theorem}
\begin{proof}	
	Since $\clconv(\X)$ is a closed convex cone, we have that $\F^{\tilde{m}}(\clconv(\X))\subseteq \F^{\tilde{m}+1}(\clconv(\X))$ and $\reccone(\clconv(\X)) = \clconv(\X)$. Thus, $\F^{\tilde{m}}(\clconv(\X))+\F^{\tilde{m}+1}(\clconv(\X))=\F^{\tilde{m}+1}(\clconv(\X))$. 
	Using the result in \autoref{them:ubdconv}, it remains to show the necessity of the condition in Part (a). Suppose that set $\C_{\rel}$ achieves the convex hull exactness for any $m$ LMIs of dimension $\tilde{m}$ in RCOP \eqref{eq_rank}. According to \autoref{them:ext}, we must have that $F \subseteq \X$ if $F\in \F^{\tilde{m}}(\clconv(\X))$. Next, we show that 
	the domain set $\X$ contains all $(\tilde{m}+1)$-dimensional faces in its closed convex hull by contradiction.
	
	Suppose that $F^*$  is an $(\tilde{m}+1)$-dimensional face  of $\clconv(\X)$ that is not contained in $\X$. Then there is a half-line $\{\alpha \bm D: \alpha \ge 0\}$ in $F^*$ with nonzero direction $\bm D$ only intersected with $\X$ at origin, provided that the domain set $\X$ is conic. 
	Since $\dim(F^*)=\tilde{m}+1$, 
	there exists an $\tilde m$-dimensional LMI system
	$\H:=\{\bm X\in \Q: \langle{\bm A}_i, \bm X\rangle = 0, \forall i \in [m] \}$ such that the intersection set $F\cap \H$ is one-dimensional and  equal to $\{\alpha \bm D: \alpha \ge 0\}$.
	Given the LMI system $\H$, the two feasible sets $\C, \C_{\rel}$ defined in \eqref{eq_2sets} become
	\[\hat{\C}=\H\cap \X,   \ \  \hat{\C}_{\rel}=\H\cap \clconv(\X). \]
	Following the analysis in \autoref{them:ext}, we can show that $\{\alpha \bm D: \alpha \ge 0\}$ is an extreme ray in the recession cone of  the set $\hat{\C}_{\rel}$. According to \autoref{def:ext},  $\bm D$ is naturally an extreme direction of set $\C_{\rel}$.
	However, $\bm D$ cannot be an extreme direction in set $\clconv(\hat{\C})$ since it does not belong to the domain set $\X$, contradicting that $\hat\C_{\rel} = \clconv(\hat\C)$ for any $m$-dimensional LMIs and completes the proof. 
	\qed
\end{proof}

The result in \autoref{them:qcqpch} can be used to show the convex hull exactness for the QCQP in which the domain set is $\X:=\{\bm   X\in \S_+^{n+1}: \rank(\bm X)\le 1 \}$ and  thus conic.  Particularly, we remark that 
\begin{enumerate}[(i)]
	\item For QCQP, when the corresponding DWR set $\C_{\rel}$ is conic, the convex hull exactness reduces to the Rank-One Generated (ROG) property. Thus, \autoref{them:qcqpch} gives a simultaneously necessary and sufficient condition for the ROG property of QCQP.
	\item 	In \autoref{eg2}, as a special case of the domain set $\X$ of the QCQP with $n=2$, the set $\X$ in \autoref{eg2} contains any no larger than two-dimensional face of $\clconv(\X)$. We see that the convex hull exactness does not hold  in \autoref{eg2} since there are $\tilde m=2$-dimensional LMIs, which demonstrates the correctness of  \autoref{them:qcqpch}; and
	\item In contrast,  we show  that if there is only one $m=\tilde m=1$ LMI in \autoref{eg2}, then the convex hull exactness must hold based on \autoref{them:qcqpch} as illustrated below.
\end{enumerate}


\begin{example}\rm
	\label{eg3}
	Let us consider the same domain set $\X:= \{\bm X \in \S_+^2: \rank(\bm X ) \le 1\}$ as \autoref{eg2}. Intersecting $\X$ and its convex hull with an $m=\tilde m=1$ LMI- $X_{12}\le 0$ yields sets
	\begin{align*}
		& \C:=\{\bm X \in \S_+^2:  X_{12}\le 0 ,\rank(\bm X)\le 1 \},  \ \  \C_{\rel}:=\{\bm X \in \S_+^2:  X_{12}\le 0 \},
	\end{align*}
	respectively. This domain set $\X$ contains all faces of $\clconv(\X)$ with dimension up to two as shown in \autoref{eg2}. 
	As $m=\tilde{m}=1$, according to \autoref{them:qcqpch}, we must have convex hull exactness, i.e., $\clconv(\C)=\C_{\rel}$.
	In fact, set $\C$ is precisely the lower surface of domain set $\X$, as marked {in red} in Figure \ref{egc1}. The red area in Figure \ref{egcrel} illustrates set $\C_{\rel}$. It is seen that (i) both $\C$ and $\C_{\rel}$  are unbounded; (ii) the convex hull of set $\C$ is half closed and half open, as $X_{12}$ cannot attain zero; and (iii) the closed convex hull of set $\C$ is equal to set $\C_{\rel}$.  \qedA
\end{example}

\begin{figure}[ht]
	\centering
	\vskip -0.2in
	\subfigure[$\C$ ]{
		{\includegraphics[width=0.2\columnwidth]{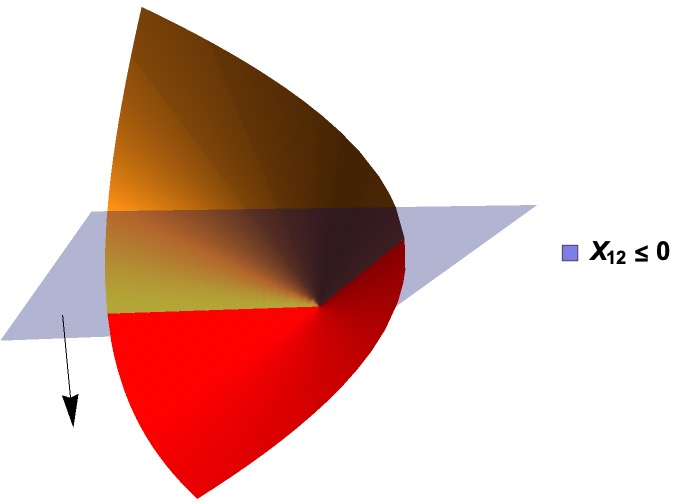}}	\label{egc1}}
	\hspace{4em}
	\subfigure[$\C_{\rel}$]{
		{\includegraphics[width=0.2\columnwidth]{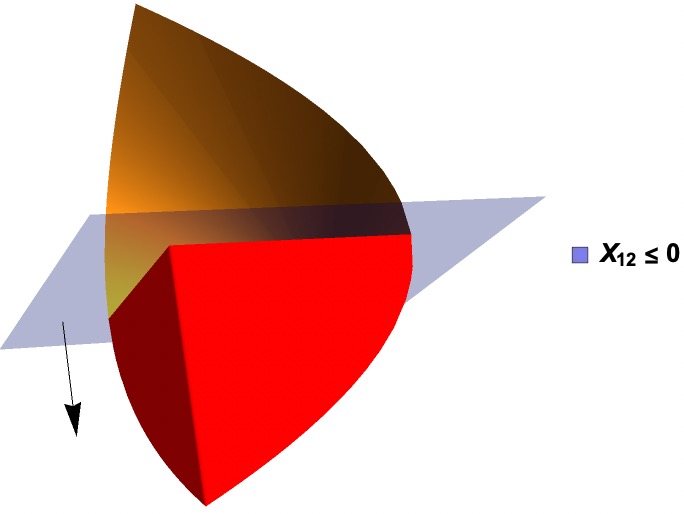}}	\label{egcrel}}
	\caption{Illustration of Sets in \autoref{eg3} with $\X := \{\bm X \in \S_{+}^2: \rank(\bm X)\le 1\}$ and we have $\clconv(\C) = \C_{\rel}$.}
	\label{fig:egubd}
\end{figure}

So far, the proposed necessary and sufficient conditions 
have revealed a significant connection between the DWR exactness and faces of set $\clconv(\X)$ that the domain set $\X$ contains.
The next subsection proves the  DWR exactness results using these conditions in some special cases of QCQP.

\subsection{Applying Our Proposed Necessary and Sufficient Conditions to QCQP}\label{sec:app}
As a special  yet important case of our RCOP \eqref{eq_rank},
this subsection investigates several QCQP problems to derive the extreme point exactness and convex hull  exactness for their corresponding DWRs. 
We first study what  faces of the closed convex hull (i.e., $\clconv(\X)$) are included in the domain set $\X$ of QCQP \eqref{qcqp}. 
Please note that for a particular QCQP example, we specify their related sets  $\X, \clconv(\X), \C, \C_{\rel}$.

As mentioned in Subsection \ref{subsec:scope}, the  domain set of general QCQP \eqref{eq_qcqp} only involves a rank-one constraint, i.e., $\X:=\{\bm X \in \S_+^{n+1}: \rank(\bm X)\le 1\}$, whose convex hull is closed, line-free, and equal to the positive semidefinite cone, i.e., $\clconv(\X)= \conv(\X) :=\S_+^{n+1}$. 
The domain set $\X$ in \autoref{eg2} is in fact a special case of QCQP with $n+1=2$.  Note that the domain set $\X$ in \autoref{eg2} contains all one- or two-dimensional faces of set $\clconv(\X)$ as shown in Figure \ref{eg2x} and Figure \ref{eg2convx}. We prove that this result can be extended to any solely rank-$k$ constrained positive semidefinite domain set as below. 
Different from \cite{pataki1998rank}, our proof idea focuses on the faces of set $\clconv(\X)$.
%


\begin{lemma}\label{lem:psd}
	Suppose that $\X=\{\bm X \in \S_+^{n+1}: \rank(\bm X)\le k\}$, i.e., $\Q=\S_+^{n+1}$ and $t= 0$ in the domain set \eqref{eq_set}.
	Then we have  $\clconv(\X) :=\S_+^n$, and any no larger than $\frac{k(k+3)}{2}$-dimensional face of $\clconv(\X) $ is contained in $\X$, i.e.,  $F \subseteq \X$ for all $F\in \F^{\frac{k(k+3)}{2}}(\clconv(\X))$.
\end{lemma}

\begin{proof}
	Given $\X=\{\bm X \in \S_+^{n+1}: \rank(\bm X)\le k\}$,	we must have $\clconv(\X)=\conv(\X) :=\S_+^n$ because any positive semidefinite matrix can be written as a convex combination of rank-one matrices. 
	
	Next, we  prove the facial inclusion result. 
	Observe that any extreme point of set $\clconv(\X)$ belongs to  $\X$. Let us denote by $F$ a face in $\clconv(\X)$ with dimension $d \le \frac{k(k+3)}{2}$. Suppose that $F$ is not contained in set $\X$, then there exists a matrix $\hat{\bm X}\in F$ with rank  $r> k$, i.e., $\hat{\bm X} \notin \X$. Let $\hat{\bm X} = \bm Q \bm \Lambda \bm Q^{\top}$ denote  the eigen-decomposition of matrix $\hat{\bm X}$ 
	where $\bm \Lambda \in \S_{++}^r$ consists of positive eigenvalues, and the eigenvector matrix $\bm Q \in \Re^{(n+1)\times r}$ has rank-$r$.
	
	Since face $F$ is of dimension $d$,  there are $d+1$ distinct and linearly independent points in $F$, denoted by $\{\bm X_i\}_{i\in [d+1]}\subseteq \S_{+}^n$ such that $$F\subseteq \aff(\bm X_1, \bm X_2, \cdots, \bm X_{d+1}).$$
	In addition, given $r\ge k+1$, we have $\frac{r}{2}(r+1) \ge  \frac{k+1}{2}(k+2) > \frac{k}{2}(k+3) \ge d $; hence,  there exists a nonzero symmetric matrix $\bm \Delta \in S^r$ satisfying
	\[\langle \bm \Delta, \bm Q^{\top} (\bm X_1-\bm X_i)\bm Q \rangle = 0, \forall i\in [2, d+1],\]
	which means that $\bm Q \bm \Delta \bm Q^{\top}$ is a nonzero matrix orthogonal to the face $F$, 
	i.e., $\bm Q \bm \Delta \bm Q^{\top} \perp F$.

	%
	
	Then let us construct two matrices $\bm X^+$ and $\bm X^-$ as below
	\[\bm X^+:= \hat{\bm X}+ \delta \bm Q \bm \Delta \bm Q^{\top} ,  \ \ \bm X^-:=\hat{\bm X}- \delta \bm Q \bm \Delta \bm Q^{\top},\]
	where $\delta >0$. 
	It is clear that $\bm X^+$ and $\bm X^-$ have nonzero eigenvalues identical to  $\bm \Lambda + \delta \bm \Delta$ and $\bm \Lambda - \delta\bm \Delta$, respectively. 
	Since $\bm \Lambda \in \S_{++}^r$, we can always make $\delta$ small enough such that $\bm X^+, \bm X^-\in \clconv(\X)$. 
	
	According to \autoref{def:face} of the face $F$, we conclude that $\bm X^+, \bm X^-\in F$ since $\hat{\bm X}=1/2\bm X^+ + 1/2\bm X^-$, which contradicts $\bm Q \bm \Delta \bm Q^{\top} \perp F$. Therefore, any face of dimension at most $\frac{k}{2}(k+3)$ in $\clconv(\X)$ belongs to $\X$. 	\qed 
	%
\end{proof}

The facial inclusion result in \autoref{lem:psd} enables us to apply the proposed  exactness conditions 
to the DWR \eqref{eq_rel_rank}  of QCQP. As a side product, if no DWR exactness holds, we can derive an upper bound for the largest rank among all the extreme points in the feasible set $\C_{\rel}$.
The results are summarized below. 
\begin{theorem}\label{them:egqcqp}
	Suppose the domain set $\X:=\{\bm X \in \S_+^{n+1}: \rank(\bm X)\le 1\}$ in QCQP \eqref{eq_qcqp},
	then we have
	\begin{enumerate}[(i)]
		\item 	
		The DWR \eqref{eq_rel_rank} attains  extreme point exactness for any $\tilde m\le 2$-dimensional LMIs;
		\item 	The DWR \eqref{eq_rel_rank} attains  convex hull exactness for any $\tilde m\le 1$-dimensional LMIs;
		\item 
		Each extreme point in set $\C_{\rel}$ has a rank at most $r^*$, where $r^*$ is the smallest integer satisfying $r^*(r^*+1) \le 2\tilde{m}$ for some $\tilde{m}\ge 3$.
	\end{enumerate}
\end{theorem}
\begin{proof}
	The proof includes two parts.
	\begin{enumerate}[(i)]
		\item According to \autoref{lem:psd} with $k=1$ and \autoref{them:ext},  we have the extreme point exactness.
		\item According to \autoref{lem:psd} with $k=1$  and \autoref{them:qcqpch}, the  convex hull exactness holds when $\tilde{m}\le 1$.
		\item 
			We prove the rank bound by contradiction. Suppose there is an extreme point $\hat{\bm X}$ in set $\C_{\rel}$ such that $r^*(r^*+1)> 2\tilde{m}  \Longleftrightarrow \tilde{m} \le  \frac{r^*-1}{2}(r^*+2)$. According to \autoref{lem:genext}, the extreme point $\hat{\bm X}$ is contained in a face of set $\conv(\X)$ with dimension at most $\frac{r^*-1}{2}(r^*+2)$.
			Using in \autoref{lem:psd} with $k=r^*-1$, the extreme point $\hat{\bm X}$ should have rank $r^*-1$, a contradiction. \qed
		\end{enumerate}
		%
		%
	\end{proof}
	
	We remark that the exactness results in Part (i) and Part (ii) of \autoref{them:egqcqp} do not require any additional assumption, and they can be applied to general QCQP \eqref{eq_qcqp}.  
	Besides, the rank bound in Part (iii) of \autoref{them:egqcqp} recovers the classical result in QCQP, which has been independently proved by  \cite{barvinok1995problems,deza1997geometry,pataki1998rank}.
	Since set $\C_{\rel}\subseteq \S_{+}^n$ for QCQP  contains no line, the extreme point exactness  implies the objective exactness for  any linear objective function with $\V_{\rel} >-\infty$. 
	
	Many results have been developed for the DWR exactness of QCQP  with one or two quadratic constraints.
	Our \autoref{them:egqcqp} can immediately recover or generalize those results and our proof does not rely on strong duality or Slater condition. 
	In what follows, we show the DWR exactness for 
	QCQP with one quadratic constraint (QCQP-1) and homogeneous QCQP with two quadratic constraints (HQP-2).
	

	\noindent \textbf{QCQP with One Quadratic Constraint (QCQP-1).} Formally, QCQP-1 is defined as
	\[ \text{(QCQP-1) } \quad \min_{\bm x \in \Re^n} \left\{ \bm x^{\top} \bm Q_0 \bm x + \bm q_0^{\top} \bm x + b:   b_1^l\le   \bm x^{\top} \bm Q_1 \bm x + \bm q_1^{\top} \bm x \le b_1^u \right\}, \]
	which can be formulated as a special case of our RCOP \eqref{eq_rank} with  $m=\tilde{m}=2$ LMIs as below
	\begin{align*}
		\min_{\bm X \in \X} \left\{\langle\bm A_0, \bm X\rangle: 
		b_1^l\le \langle \bm A_1, \bm X\rangle \le b_1^u,  X_{11}=1 \right\}, \ \  \X:=\{\bm X\in \S_+^{n+1}: \rank(\bm X)\le 1\},
	\end{align*}
	where ${\bm A}_0 = \begin{pmatrix}
		b& { \bm q_0^{\top}}/{2}\\
		{\bm q_0}/{2}  & \bm Q_0\\
	\end{pmatrix}$ and ${\bm A}_{1} = \begin{pmatrix}
		0& { \bm q_1^{\top}}/{2}\\
		{\bm q_1}/{2}  & \bm Q_1\\
	\end{pmatrix}$. 
		Then Part (i) of  \autoref{them:egqcqp} implies the following conclusion.
		\begin{corollary}\label{cor:qp1}
			For QCQP-1, its corresponding DWR  admits the extreme point exactness.
		\end{corollary}
		QCQP-1 covers many important and challenging quadratic optimization problems
		which have attracted much attention in various applications from different domains such as robust optimization \citep{ben2009robust}, regularization problem (e.g., ridge regression) \citep{hoerl1970ridge,tikhonov1977solutions,xie2020scalable}, and subproblems in signal preprocessing \citep{huang2016consensus}. We show that some widely-studied special cases of QCQP-1 may possess the  convex hull exactness more than the  extreme point exactness in \autoref{cor:qp1}, as discussed below. 

		
		\noindent \textit{Trust Region Subproblem (TRS).} The classical TRS, a special case of QCQP-1, is  to minimize a quadratic objective over a ball ($\bm x ^{\top} \bm x\le 1$), arising from trust region methods for nonlinear programming \citep{conn2000trust}. 
		Albeit being nonconvex, the TRS problem is known to achieve DWR objective exactness and strong duality (see a survey \cite{polik2007survey}).   
		Recently, \cite{burer2015gentle} explicitly described  the convex hull of feasible set of the TRS based on a  second-order cone. Our result of QCQP-1 in \autoref{cor:qp1} implies that the DWR of TRS problem also achieves convex hull exactness.
		\begin{corollary}[TRS] \label{cor:trs}
			The  convex hull exactness holds for the TRS problem.
		\end{corollary}
		\begin{proof}
			The constraint $\bm x^{\top}\bm x\le 1$ in the TRS implies that
			the DWR problem of TRS  has a bounded feasible set $\C_{\rel}$, and thus, using \autoref{cor:qp1}, the convex hull exactness follows from the equivalence between Part (b) and Part (c) in \autoref{them:exactch}. \qed
		\end{proof}

		
		\noindent \textit{Generalized TRS (GTRS).} Replacing the ball constraint in TRS by an arbitrary one-sided  quadratic constraint (i.e., $\bm x^{\top} \bm Q_1 \bm x + \bm q_1^{\top} \bm x \le b_1$)  leads to the following GTRS problem:
		\begin{align}\label{eq:gtrs}
			\text{(GTRS)} \ \ \min_{\bm X \in \X} \left\{\langle\bm A, \bm X\rangle: 
			\langle \bm A_1, \bm X\rangle \le b_1,  X_{11}=1 \right\}, \ \  \X:=\{\bm X\in \S_+^{n+1}: \rank(\bm X)\le 1\},
		\end{align}
		which satisfies $\tilde{m}=2$ and thus admits the extreme point exactness as a special case of QCQP-1. Note that the feasible set $\C_{\rel}$ in the DWR of GTRS \eqref{eq:gtrs} can be unbounded, which often results in the failure of  convex hull exactness; see, e.g., \autoref{eg2}.  
		On the other hand, the special linear constraint $X_{11}=1$ inspires us 
		to prove the convex hull exactness in the corollary below.
		It is worth
		mentioning that our result strengthens the one  in \cite{kilincc2021exactness} that  relies on the assumption that the dual set of DWR of the GTRS  is strictly feasible.
		\begin{restatable}{corollary}{corgtrs}\textbf{\rm \textbf{(GTRS)}} \label{cor:gtrs}
			The convex hull exactness holds for the GTRS  problem.
		\end{restatable}
		\begin{proof}
			See Appendix \ref{proof:corgtrs}. \qed
		\end{proof}
		
		\noindent \textit{Two-sided GTRS.} As an extension of GTRS, the two-sided  GTRS problem has a two-sided quadratic constraint, i.e., $-\infty<b_1^l \le \bm x^{\top} \bm Q_1 \bm x + \bm q_1^{\top} \bm x \le b_1^u<+\infty$,  which has been successfully applied to signal processing (see \cite{huang2014randomized} and references therein). Using S-lemma, {the objective exactness} for the DWR of the two-sided GTRS   has been established under Slater assumption  (see survey by \cite{wang2015strong}  and references therein), which is equivalent to the extreme point exactness.
		According to \autoref{cor:qp1}, 
		we can readily derive  the first-known DWR extreme point exactness of the two-sided  GTRS without Slater condition.  
		Recent work by \cite{joyce2021convex} showed that the two-sided GTRS has
		the convex hull exactness given that the data matrix $\bm Q_1$ above is nonzero, which can be recovered by our framework. However, the convex hull exactness may not  hold for the general two-sided GTRS (see \autoref{eg2} with $\bm Q_1=\bm 0$).
		
		\begin{corollary}[Two-Sided GTRS]\label{cor:tgtrs}
			For the two-sided GTRS  problem, we have 
			\begin{enumerate}[(i)]
				\item The extreme point exactness holds;
				\item The convex hull exactness holds when $\bm Q_1 \neq \bm 0$ and $-\infty < b_1^l \le b_1^u < +\infty$. 
			\end{enumerate}
		\end{corollary}
		\begin{proof}
			Part (i) can be obtained by simply following the proof of QCQP-1 in \autoref{cor:qp1}. 
			
			Next, let us prove Part (ii).  First, the corresponding sets $\C$ and $\C_{\rel}$  for two-sided GTRS are
			\begin{align*}
				&\C:=\left\{\bm X\in \S_+^{n+1}: \rank(\bm X)\le 1, X_{11}=1,  b_1^l \le \langle\bm A_1, \bm X\rangle \le b_1^u \right\}, \\ 
				&\C_{\rel}  :=\left\{\bm X\in \S_+^{n+1}:  X_{11}=1, b_1^l \le \langle\bm A_1, \bm X\rangle \le b_1^u \right\},
			\end{align*}
			where $\bm A_1 = \begin{pmatrix}
				0 & \bm q_1^{\top}/2\\
				\bm q_1/2 & \bm Q_1
			\end{pmatrix}$. 
			Following the analysis in \autoref{cor:gtrs}, the recession cone of set $\C_{\rel}$ is equal to
			\begin{align*}
				\reccone(\C_{\rel}) = \clconv\left( \left\{\bm X\in \S_+^{n+1}: \rank(\bm X)\le1, \langle \bm A_1, \bm X\rangle = 0,  X_{11}=0 \right\}\right).
			\end{align*}
			For any rank-one direction $\bm D := \begin{pmatrix}0 & \bm 0^{\top}\\\bm 0 & \bm y \bm y^{\top}\end{pmatrix}$ in $\reccone(\C_{\rel})$, according to proposition 3 in \cite{joyce2021convex}, $\bm{D}$ is also a direction in $\clconv(C)$ when $\bm Q_1\neq \bm 0$. Thus, $\reccone(\C_{\rel}) =	\reccone(\clconv(C))$, which, together with the extreme point exactness and the representation theorem in \cite{rockafellar2015convex}, leads to the desired conclusion. 
			\qed
		\end{proof}

		\noindent \textbf{Homogeneous QCQP with Two Independent Constraints (HQP-2).} Another  special case of the QCQP  is a homogeneous QCQP with  two independent constraints without linear terms, denoted as HQP-2, which has  witnessed applications in robust receive beamforming \citep{khabbazibasmenj2010robust, huang2010dual} and signal processing \citep{huang2014randomized}. The HQP-2 admits the following form
		\begin{equation*}
			\text{(HQP-2)} \quad 
			\min_{\bm x \in \Re^n} \left\{ \bm x^{\top} \bm Q_0 \bm x : b_1^l \le \bm x^{\top} \bm Q_1 \bm x  \le b_1^u,   b_2^l \le \bm x^{\top} \bm Q_2 \bm x \le b_2^u\right\}.
		\end{equation*}
		The HQP-2 is different from the general QCQP-1 since its equivalent rank-one constrained formulation builds on the size-$n$ positive semidefinite set instead of $n+1$ without the auxiliary constraint $X_{11}=1$
		\begin{align*}
			\min_{\bm X \in \X} \left\{\langle\bm A, \bm X\rangle:  b_1^l \le 	\langle \bm A_1, \bm X\rangle \le  b_1^u,   b_2^l \le 	\langle \bm A_2, \bm X\rangle  \le   b_2^u  \right\} ,\ \  \X:=\{\bm X\in \S_+^{n}: \rank(\bm X)\le 1\},
		\end{align*}
		where ${\bm A}= \bm Q_0$, ${\bm A}_1= \bm Q_1$, and ${\bm A}_2= \bm Q_2$.
		
		Part (i) in \autoref{them:egqcqp} directly implies the extreme point exactness of HQP-2. In fact, this result was first proved by \cite{polyak1998convexity}, and years later, \cite{ye2003new} used the matrix rank-one decomposition procedure to reprove it. It is worth mentioning that both proofs rely on the strong duality assumption. In contrast, our analysis manages to relax
		the strong duality assumption. Note that we may not derive the convex hull exactness of HQP-2 since set $\C_{\rel}$ can be unbounded, and its recession cone can be different from those in $\clconv(\C)$ (see \autoref{eg2} with $n=2$ for an illustration).
		
		\begin{corollary}[HQP-2]\label{cor:hqp}
			For HQP-2, its corresponding DWR  admits the extreme point exactness.
		\end{corollary}

					\section{An Optimality View of DWR Exactness: Simultaneously Necessary and Sufficient Conditions for Objective Exactness} \label{sec:obj}
					The objective  exactness is another common way to show whether DWR \eqref{eq_rel_rank} matches the original
					RCOP \eqref{eq_rank} regarding the optimal values (i.e., whether $\V_{\opt} = \V_{\rel}$). The main difference of
					objective exactness from the other two exactness notions is that it depends on a given linear objective function in RCOP \eqref{eq_rank}. To illustrate this difference, let us review \autoref{eg1}, where two optimal values $\V_{\opt} = \V_{\rel}=-0.5$ are equal using the linear objective function $-X_{11}$, while neither extreme point exactness nor convex hull exactness holds. 
					Therefore,  this section investigates the objective exactness of DWR \eqref{eq_rel_rank} for any $m$ LMIs of dimension $\tilde{m}$  concerning four favorable families of linear objective functions in RCOP \eqref{eq_rank} specified as follows and derives their simultaneously necessary and sufficient conditions.\par
					\noindent \textbf{Setting I.} Any linear objective function;\par
					\noindent \textbf{Setting II.} Any linear objective function such that $\V_{\opt} > -\infty$;\par
					\noindent \textbf{Setting III.} Any linear objective function such that $\V_{\opt} > -\infty$, the set  of optimal solutions is bounded, and the binding LMIs are of dimension $\underline{m}$; \par
					\noindent \textbf{Setting IV.} Any linear objective function such that the relaxed Slater condition  holds, $\V_{\opt} > -\infty$, the set  of optimal solutions is bounded, and there are $\underline{m}^* $  nonzero optimal Lagrangian multipliers corresponding to the optimal DWR.\par
					Note that  $\underline{m}$ and $\underline{m}^*$ are used throughout this section to indicate the number of linearly independent matrices from binding LMIs of DWR \eqref{eq_rel_rank} and the smallest number of nonzero  Lagrangian multipliers among all the optimal DWR dual solutions. 
					For comparison purposes, the four key notations of the LMIs in DWR \eqref{eq_rel_rank} are listed in  \autoref{tab:my_label}. Please note that we will show $\underline{m}^*\le \underline{m} \le \tilde{m} \le m$.
					
					\begin{table}[ht]
						\centering
						\setlength{\tabcolsep}{3pt}\renewcommand{\arraystretch}{1.2}
						\begin{tabular}{|c|l|}
							\hline
							Notation   & 	\multicolumn{1}{c|}{Definition} \\
							\hline
							$m$   & the number of  LMIs  \\
							\hline
							$\tilde m$ & dimension of technology matrices in all  $m$  LMIs \\
							\hline
							$\underline{m}$ & dimension of technology matrices in binding  LMIs \\
							\hline
							$\underline{m}^*$ & the smallest number of nonzero optimal Lagrangian multipliers \\
							\hline
						\end{tabular}
						\caption{Notations about LMIs in DWR  \eqref{eq_rel_rank}.}
						\label{tab:my_label}
						\vspace{-1em}
					\end{table}
					As illustrated in \autoref{fig:relation},   the objective exactness under settings (I) and (II) are equivalent to the convex hull  exactness and the extreme point exactness, respectively. Thus, the results in the previous section can be directly applied to the simultaneously necessary and sufficient conditions for objective exactness.
					The remaining two settings (III) and (IV) focus on two special yet intriguing RCOP families by analyzing primal and dual perspectives, respectively. 
					Please note that although our proposed conditions for objective exactness under settings (III) and (IV) require assumptions of technology matrices in the  LMIs of DWR \eqref{eq_rel_rank}, they can cover and extend the existing results in the celebrated papers \citep{ye2003new,ben2014hidden}.
					
					\subsection{Objective Exactness Under  Settings (I) and (II): Simultaneously Necessary and Sufficient Conditions}\label{sec:any}
					This subsection presents simultaneously  necessary and sufficient conditions for objective exactness of DWR \eqref{eq_rel_rank} under settings (I) and (II) using  their equivalence to convex hull exactness and extreme point exactness, respectively.
					
					We extend \autoref{them:ubdconv}  to provide a sufficient condition for the objective exactness under setting (I).
					\begin{theorem}\label{them:undobj}
						Given a nonempty closed domain set $\X$ with its closed convex hull $\clconv(\X)$ being line-free, 
						the following statement (a) implies statement (b):
						\begin{enumerate}[(a)]
							\item \textbf{Inclusive Face:} The (Minkowski) sum of any no larger than $\tilde{m}$-dimensional face of set $\clconv(\X)$ and any no larger than $(\tilde{m}+1)$-dimensional face of recession cone $\reccone(\clconv(\X))$  is contained in the domain set $\X$, i.e.,  $F+ \hat{F} \subseteq \X$ for all $F\in \F^{\tilde{m}}(\clconv(\X))$ and $\hat{F}\in \F^{\tilde{m}+1}(\reccone(\clconv(\X)))$;
							\item \textbf{Objective Exactness:} The DWR  \eqref{eq_rel_rank} has the same optimal value as RCOP \eqref{eq_rank} (i.e., $\V_{\opt}=\V_{\rel}$) for any linear objective function and any $m$ LMIs of dimension $\tilde{m}$.
						\end{enumerate}
					\end{theorem}
					\begin{proof}
						The objective exactness for any linear objective function is equivalent to the convex hull exactness, and thus, the proof follows from \autoref{them:ubdconv}. \qed
					\end{proof}
					
					Besides, following the analysis in \autoref{them:qcqpch}, the sufficient condition in \autoref{them:undobj} becomes necessary when the domain set $\X$ is 
					closed and conic. 
					\begin{theorem}\label{them:qcqpobj}
						Given a nonempty closed conic domain set $\X$, i.e., for any $\alpha\geq 0$ and $\bm X\in \X$, we have $\alpha \bm X \in \X$, with its closed convex hull $\clconv(\X)$ being line-free the followings are equivalent:
						\begin{enumerate}[(a)]
							\item \textbf{Inclusive Face:} Any no larger than $(\tilde{m}+1)$-dimensional face of set $\clconv(\X)$ is contained in the domain set $\X$, i.e.,  $F \subseteq \X$ for all $F\in \F^{\tilde{m}+1}(\clconv(\X))$; 
							\item \textbf{Objective Exactness:} The DWR  \eqref{eq_rel_rank} has the same optimal value as problem \eqref{eq_rank} (i.e., $\V_{\opt}=\V_{\rel}$) for any linear objective function and any $m$ LMIs of dimension $\tilde{m}$.
						\end{enumerate}
					\end{theorem}
					
					
					Let us now consider the objective exactness of the DWR
					with finite optimal value, i.e., $\V_{\rel} >-\infty$.
					In this situation, the objective exactness is equivalent to the extreme point exactness. Thus, we readily obtain the following simultaneously necessary and sufficient condition for objective exactness. 
					\begin{restatable}{theorem}{themexactobj}\label{them:exactobj} 
						Given a nonempty closed domain set $\X$ with its closed convex hull $\clconv(\X)$ being line-free, the followings  are equivalent.
						\begin{enumerate}[(a)]
							\item \textbf{Inclusive Face:} Any no larger than $\tilde{m}$-dimensional face of set $\clconv(\X)$ is contained in the domain set $\X$, i.e.,  $F \subseteq \X$ for all $F\in \F^{\tilde{m}}(\clconv(\X))$; 
							\item \textbf{Objective Exactness:} The DWR  \eqref{eq_rel_rank} has the same optimal value as problem \eqref{eq_rank} (i.e., $\V_{\opt}=\V_{\rel}$) for any linear objective function such that $\V_{\rel}>-\infty$ and any $m$ LMIs of dimension $\tilde{m}$.
						\end{enumerate}
					\end{restatable}
					\begin{proof}
						According to \autoref{them:ext}, we have that statement $(a) \iff$ extreme point exactness. Thus, it is equivalent to prove that extreme point exactness $ \iff $ objective exactness for any linear objective function such that $\V_{\rel}>-\infty$. 
						\begin{enumerate}[(i)]
							\item $\ext(\C_{\rel}) \subseteq \C \Longrightarrow \V_{\opt}=\V_{\rel}>-\infty$. For any linear objective function $ \langle\bm A_0, \bm X\rangle $, since $\V_{\opt}\geq \V_{\rel}>-\infty$, we have $ \langle\bm A_0, \bm D\rangle \geq 0$ for any direction $\bm D$ in set $\C_{\rel}$ and so does set $\C$. In addition, both sets $\C_{\rel}$  and $\clconv(\C)$ contain no line. 
							Note that according to the theorem 18.5 in \cite{rockafellar2015convex}, a closed convex line-free set can be represented as sum of a convex combination of extreme points and a conic combination of extreme directions. 
							Thus, it suffices to rewrite RCOP \eqref{eq_rank} and its DWR \eqref{eq_rel_rank} as $\V_{\opt}=\min_{\bm{X}}\{ \langle\bm A_0, \bm X\rangle :\bm{X}\in \ext(\clconv(\C))\}$ and $\V_{\rel}=\min_{\bm{X}}\{ \langle\bm A_0, \bm X\rangle :\bm{X}\in \ext(\C_{\rel})\}$. Given $\ext(\C_{\rel})\subseteq \C$,  we must have $\V_{\opt}=\V_{\rel}>-\infty$.
							
							
							\item $\V_{\opt}=\V_{\rel}>-\infty \Longrightarrow \ext(\C_{\rel}) \subseteq \C $. For any exposed extreme  point $\hat{\bm X} $ in set $\C_{\rel}$,  there exists a supporting hyperplane $\{\bm X \in \Q: \langle \hat{\bm A}_0, \bm X\rangle = \langle \hat{\bm A}_0, \hat{\bm X}\rangle= \V_{\rel} >-\infty\}$ of set $\C_{\rel}$ which only intersects set $\C_{\rel}$ at  $\hat{\bm X} $ and satisfies $\langle \hat{\bm A}_0, \bm X\rangle > \V_{\rel}$ for any  $\bm{X}\in \C_{\rel}$ and $\bm{X}\neq \hat{\bm X}$. Therefore, by setting the linear objective function $\langle \hat{\bm A}_0, \bm X\rangle$ in RCOP  \eqref{eq_rank}, $\hat{\bm X} $ is the unique optimal solution to DWR \eqref{eq_rel_rank}. Since $\V_{\opt}=\V_{\rel}>-\infty$ and $\C \subseteq \C_{\rel}$, we conclude that $\hat{\bm X} \in \C$.
							
							For a non-exposed extreme point $\hat{\bm X} $ in the closed convex set $\C_{\rel}$, according to Straszewicz's theorem in \cite{rockafellar2015convex}, there exists a sequence of exposed  points $\{\bm X_{\ell}\}_{\ell=1,2,\cdots, \infty}$ in set $\C_{\rel}$ such that $\lim\limits_{\ell \to \infty} \bm X_{\ell}= \hat{\bm X}$. Using the result above, we can show $\{\bm X_{\ell}\}_{\ell=1,2,\cdots, \infty} \subseteq \C$. Since set $\C$ is  closed,  we must have $\hat{\bm X}\in \C$ as it is the limit point of a sequence in set $\C$.  This proves $\ext(\C_{\rel}) \subseteq \C$.	\qed 
						\end{enumerate}
					\end{proof}

					Next we apply \autoref{them:exactobj} to show the objective exactness for a special QCQP problem, known as Simultaneously Diagonalizable QCQP (SD-QCQP), i.e., the matrices $\{\bm Q_0, \bm Q_1, \cdots,\bm Q_m \}$ in QCQP \eqref{qcqp} are simultaneously diagonalizable. \cite{ben2014hidden} first proved that the DWR of SD-QCQP with a one-sided quadratic constraint admitted the objective  exactness under the assumption that $\V_{\rel}>-\infty$, which is a special case of QCQP-1 in  \autoref{cor:qp1}. In fact, we can further generalize their result by showing that the objective exactness for SD-QCQP with  $\V_{\rel}>-\infty$ and a two-sided constraint  always holds.
					\begin{corollary}[SD-QCQP] \label{cor:sdqcqp}
						For the SD-QCQP with a two-sided quadratic constraint, its DWR  admits the objective exactness for any linear objective function such that $\V_{\rel}>-\infty$.
					\end{corollary}
					\begin{proof}
						As SD-QCQP is a special case of QCQP-1 and the corresponding set $\C_{\rel}$ is always line-free, the conclusion follows by \autoref{cor:qp1} and \autoref{them:exactobj}. \qed
						%
						%
					\end{proof}
					

						As mentioned before, the QCQP \eqref{eq_qcqp} has a domain set $\X=\{\bm X\in \S_+^n: \rank(\bm X)\le 1\}$ being  conic  closed, and its corresponding set $\clconv(\X)$ is always line-free; hence, our \autoref{them:qcqpobj} and \autoref{them:exactobj} can be directly applied, which provide, for the first time, the simultaneously necessary and sufficient conditions for the objective exactness under Setting I and Setting II, 
						respectively. 
						When applying to the QCQP, our proposed conditions only involve the domain set $\X$ and  are regardless of the linear objective function and LMIs. The objective exactness for the convex relaxations of a QCQP has been extensively studied in the literature (see, e.g., \cite{azuma2022exact, burer2020exact,kilincc2021exactness, sojoudi2014exactness}); however,  their conditions mainly focus on specific objective coefficients or technology matrices of a given QCQP.
						
						\subsection{Objective Exactness Under Setting (III): A Relaxed Simultaneously Necessary and Sufficient Condition based on Binding LMIs} \label{sec:bind}
						%
						Our proposed simultaneously necessary and sufficient condition in \autoref{them:exactobj} guarantees the objective exactness when DWR \eqref{eq_rel_rank} is equipped with any linear objective function such that $\V_{\rel}>-\infty$. Beyond that, when this proposed condition fails, the objective exactness  may still hold for the DWR with favorable objective functions
						(see, e.g., \autoref{eg:sdqcqp} below). This motivates us to study a relaxed necessary and sufficient condition for DWR objective  exactness  given $\underline m$-dimensional binding LMIs at optimality, which covers and extends the objective exactness of two major applications in  fair unsupervised learning: fair PCA and  fair SVD. 
						
						Throughout this subsection, for ease of the analysis, we  consider only one-sided LMIs for RCOP \eqref{eq_rank} in which, without loss of generality, we let $b_i^l=-\infty$ for each $i\in [m]$. In fact, any $i$-th LMI of RCOP \eqref{eq_rank} can be recast as two 
						one-sided  LMIs of  dimension one as shown below since matrices $\{\bm A_i, -\bm A_i\}$ have the dimension of one.
						\begin{align*}
							&b_i^l \le \langle \bm A_i, \bm X_i \rangle \le b_i^u  \ \ \Longleftrightarrow  \ \ \langle -\bm A_i, \bm X_i \rangle \le -b_i^l,  \langle \bm A_i,  \bm X_i \rangle \le b_i^u.
						\end{align*}

						We begin with an example illustrating why the  binding LMIs of the DWR is important for objective exactness in \autoref{them:exactobj}. 
						\begin{example} \label{eg:sdqcqp} \rm
							Using the same domain set $\X=\{\bm X\in \S_+^2: \rank(\bm X)\le 1, X_{12}=0\}$ and its closed convex hull $\clconv(\X)=\{\bm X\in \S_+^2: X_{12}=0\}$ as shown in \autoref{eg:ray}, we consider  $m=\tilde m=2$ LMIs $X_{11}\le X_{22}$ and $X_{22}\le1$  in RCOP \eqref{eq_rank}. Then the feasible sets defined in \eqref{eq_2sets} are
							\[ \C:=\{ \bm X\in \X:  X_{11}\le X_{22}, X_{22}\le 1  \},  \  \ \C_{\rel}:=\{ \bm X \in \conv(\X):  X_{11}\le X_{22}, X_{22}\le 1 \}.\]
							Note that sets $\X$ and $\clconv(\X)$ are plotted in Figure \ref{egrayx} and Figure \ref{egrayconv}, respectively, by projecting them onto a two-dimensional vector space over $(X_{11}, X_{22})$. 
							The corresponding sets $\clconv(\C)$ and $\C_{\rel}$ are presented in \autoref{fig:sdqcqp} below.
							
							Since only zero- and one-dimensional faces of $\conv(\X)$ belong to set $\X$ as mentioned in \autoref{eg:ray}, according to \autoref{them:exactobj},  the objective exactness may fail in this example since there are $\tilde{m}= 2$-dimensional LMIs. 
							Albeit powerful,  \autoref{them:exactobj} may not rule out the possibility of attaining the objective exactness in this example.  For instance, if we set the objective function to be $X_{22}$, then the  objective exactness holds since $\V_{\opt} = \V_{\rel}=0$ with the optimal solution at point $a_1$, where there is only one binding LMI at optimality. Thus, the binding LMIs at optimality may be sufficient for the objective exactness instead of using all the LMIs.
							\qedA
						\end{example}
						
						\begin{figure}[ht]
							\vskip -0.15in
							\centering
							\subfigure[$\clconv(\C)$ ]{
								{\includegraphics[width=0.35\columnwidth]{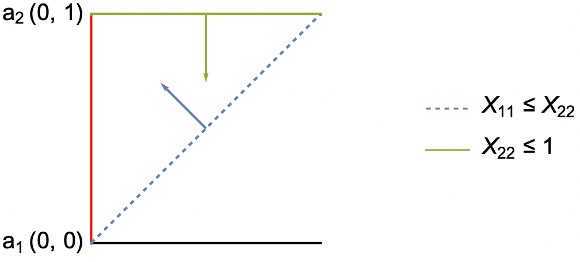}}	\label{sdc}}
							\hspace{6em}
							\subfigure[$\C_{\rel}$ ]{
								{\includegraphics[width=0.35\columnwidth]{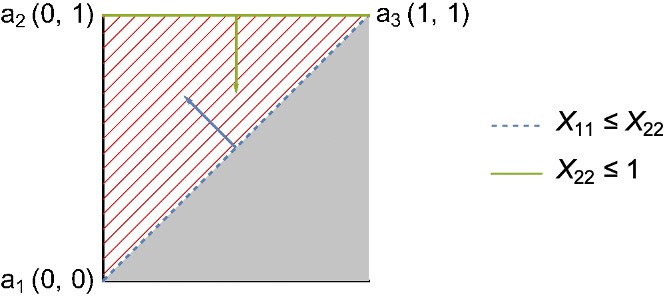}}	\label{sdcrel}}
							\caption{Illustration of Sets in \autoref{eg:sdqcqp}  and $\C_{\rel}\neq \clconv(\C)$.  }
							\label{fig:sdqcqp}
							\vskip -0.1in
						\end{figure}

						
						Motivated by \autoref{eg:sdqcqp}, we derive a relaxed necessary and sufficient condition for objective exactness  under the setting  that $\V_{\opt} >-\infty$ and there is an optimal solution to DWR  falling on the intersection of at most $\underline{m}$-dimensional binding LMIs. The relaxed condition is based on the vector projection, which is defined below.
						\begin{definition}[Vector Projection] \label{def:proj}
							For a nonzero vector $\bm x \in \Re^n$ and a set $ D \in \Re^n$, we let $\Proj_{D} (\bm x)$ denote the orthogonal projection of $\bm x$ onto $D$, called ``vector projection." Thus, $\Proj_{D} (\bm x)$ is parallel  to set $D$.
						\end{definition}
						The vector projection can be straightforwardly applied to our matrix set by vectorizing a matrix.
						
						\begin{theorem}\label{them:bind}
							Given a nonempty closed domain set $\X$ with its closed convex hull $\clconv(\X)$ being line-free and $b_i^{l}=-\infty$ for any $i\in [m]$ in RCOP \eqref{eq_rank}, the followings are equivalent:
							\begin{enumerate}[(a)]
								\item \textbf{Inclusive Face:} Any no larger than $\underline{m}$-dimensional face of set $\clconv(\X)$ is contained in the domain set $\X$, i.e.,  $F \subseteq \X$ for all $F\in \F^{\underline{m}}(\clconv(\X))$; 
								\item \textbf{Objective Exactness:} The DWR  \eqref{eq_rel_rank} has the same optimal value as RCOP \eqref{eq_rank} (i.e., $\V_{\opt}=\V_{\rel}$) for any linear objective function and any $m$ LMIs  such that (i)  $\V_{\rel}>-\infty$, (ii) the optimal set of the DWR  \eqref{eq_rel_rank} is  bounded, (iii) there are $\underline{m}$-dimensional  binding LMIs indexed by $T\subseteq [m]$ at optimality, and (iv) matrices $\{ \bm A_j - \Proj_{\H}(\bm A_j) \}_{j \in [m]\setminus T}$ are parallel with the same direction where $\H:= \spa( \{\bm A_i\}_{i\in T})$ denotes the linear combination of matrices $\{\bm A_i\}_{i\in T}$.
							\end{enumerate}
						\end{theorem}
						\begin{proof}We split the proof into two parts.\par
							\noindent \textbf{Part I.} When statement (b) holds, we can set all the non-binding LMIs to be trivial with all technology matrices and right-hand sides being zeros. Since there is no line in set $\C_{\rel}$,  following the similar argument to \autoref{them:exactobj} as well as  \autoref{them:ext}, we can derive statement (a).
							
							\noindent \textbf{Part II.} Let us next show that statement (a) implies statement (b). Without loss of generality, we assume that matrix $\bm A_i$ is nonzero for all $i\in [m]$. 
							As there are only $\underline{m}$-dimensional  binding LMIs indexed by $T$, the DWR  \eqref{eq_rel_rank} can equivalently reduce to the one with these binding LMIs. 
							We let $\hat{\C}_{\rel}$ denote the intersection set of these $\underline m$-dimensional binding LMIs with set $\clconv(\X)$, i.e., 
							$$\hat{\C}_{\rel}=\{\bm{X}\in \clconv(\X):	\langle\bm A_{i}, \bm{X} \rangle =b_i^u,\forall i\in T\}.$$
							According to our assumption, there exists an optimal DWR solution $\bm X^*\in \C_{\rel}$ satisfying $\bm X^* \in \hat{\C}_{\rel}$. There are two cases remaining to be discussed.
							\begin{enumerate}[{Case }1.]
								\item Suppose that $\bm X^*$ is an extreme point of set $\hat{\C}_{\rel}$. 
								Since there are $\underline{m}$-dimensional  LMIs in set $\hat{\C}_{\rel}$, \autoref{cor_lem:genext} shows that any extreme point of set $\hat{\C}_{\rel}$ belongs to a face in $\F^{\underline{m}}(\clconv(\X))$. Since any face $F\subseteq \X$ for all $F\in \F^{\underline{m}}(\clconv(\X))$ and $\bm X^*\in \C_{\rel}$, we have that $\bm X^*$ belongs to set $\C$, implying the objective exactness.
								
								\item Suppose that $\bm X^* $ is not an extreme point in  set $\hat{\C}_{\rel}$. Recall that $\H:= \spa( \{\bm A_i\}_{i\in T})$. Now we observe that
								\begin{claim} \label{claim}
									For any nonzero direction $\bm{D}$ in the LMI system $\{ \bm X \in \Q: \langle \bm A_i, \bm X\rangle =b_i, \forall i \in T \}$,  if there exists some $\ell \in [m]\setminus T$ such that $\bm A_{\ell} - \Proj_{\H}(\bm A_{\ell}) \neq \bm 0$ and  $\langle \bm A_{\ell}, \bm D \rangle \le 0$, then $\langle \bm A_j, \bm{D} \rangle \le 0$ holds for all $j\in [m]\setminus T$.
								\end{claim}
								\begin{proof}
									Since  $\langle \bm A_i, \bm{D} \rangle =0$ for all $i\in T$ (i.e., $\bm D \perp \H$), we have that for each $j\in [m]\setminus T$,
									\begin{align*}
										\langle\bm A_{j}, \bm{D}  \rangle &=    \langle \Proj_{\H}(\bm A_j) + \bm A_j - \Proj_{\H}(\bm A_j),  \bm{D}  \rangle =  \langle \bm A_j - \Proj_{\H}(\bm A_j), \bm{D}  \rangle \\
										&= c_{j\ell} \langle \bm A_{\ell} - \Proj_{\H}(\bm A_{\ell}), \bm D \rangle =  c_{j\ell}\langle \bm A_{\ell} , \bm{D}  \rangle  \le 0, \ \  \exists  c_{j\ell}\ge 0,
									\end{align*}
									where the second and last equations are from the fact that $\bm D \perp \H$ and $\Proj_{\H}(\bm A_{\ell})$ is parallel with $\H$  for all $j\in [m]\setminus T$ based on \autoref{def:proj}, and the third one is because the matrices $\{\bm A_j - \Proj_{\H}(\bm A_j) \}_{j \in [m]\setminus T}$ are parallel with the same direction.  
									Therefore, if $\bm{D}$ is also a direction in some non-binding LMI, it must be a direction for all non-binding LMIs.
									\qedA
								\end{proof}
								%
								
								Then, using the representation theorem 18.5 in \cite{rockafellar2015convex}, the optimal solution $\bm X^*$ in the closed convex line-free set $\hat{\C}_{\rel}$ can be represented as a finite convex combination of extreme points and a finite conic combination of extreme directions below.
								\[\bm  X^* = \sum_{l\in  [\tau]} \alpha_l \bm X_l + \sum_{l \in  [\hat{\tau}]} \beta_{l} \bm Y_l, \]
								where  $\{\bm X_1, \cdots, \bm X_\tau\}$ denote extreme points of set $\hat{\C}_{\rel}$ which belong to $\F^{\underline{m}}(\clconv(\X))$ by \autoref{cor_lem:genext},  $\{\bm  Y_1, \cdots, \bm Y_{\hat{\tau}}\}$
								denote extreme directions of set $\hat{\C}_{\rel}$, $\bm \alpha \in \Re_{++}^\tau$, $\sum_{l\in [\tau]}\alpha_l=1$, and $\bm \beta \in \Re_{++}^{\hat{\tau}}$.

								Since $\bm{X}^*\in\arg\min_{\bm{X}\in \hat{\C}_{
										\rel}}\langle \bm A_0, \bm X \rangle $, we must have $\langle \bm A_0, \bm X_l \rangle =\langle \bm A_0, \bm X^* \rangle $ for all $l\in [\tau]$ and $\langle \bm A_0, \bm Y_l \rangle=0$ for all $l\in [\hat{\tau}]$.
								According to \autoref{claim} and the presumption that the set of optimal solutions is bounded and optimal value is finite, for any extreme direction $\bm Y_l$ in set $\hat{C}_{\rel}$,  we have that  $\langle \bm A_j, \bm Y_l \rangle \geq0$ for all $j\in [m]\setminus T$. Hence, for any $j\in [m]\setminus T$, we must have $\langle \bm A_j, \bm X^* \rangle \geq  \sum_{l\in  [\tau]} \alpha_l\langle \bm A_j, \bm X_l \rangle $.
								It follows that there exists an extreme point $\bm X_{l^*}$ in set $\hat{C}_{\rel}$ such that $\bm X_{l^*} - \bm X^*$ is a direction for all non-binding LMIs according to \autoref{claim}. Therefore, we can conclude $\bm X_{l^*} \in  \C_{\rel}$. Along with the previous result that $\langle \bm A_0,  \bm X^*\rangle=\langle \bm A_0, \bm X_{l^*} \rangle$,  the extreme point $\bm X_{l^*}$ of set $\hat{C}_{\rel}$  is also optimal to DWR. According to \textbf{Case 1}, we must have $\bm X_{l^*}\in \C$, which proves the objective exactness. \qed
							\end{enumerate}
						\end{proof}
						
						In fact, the objective exactness in \autoref{eg:sdqcqp} satisfies the above condition
						since any zero- or one-dimensional face of $\clconv(\X)$ is contained in domain set $\X$, and there is only $\underline{m}=1$ binding LMI and one non-binding LMI at the unique optimal solution $a_2$ as shown in \autoref{fig:sdqcqp}. 
						Besides, we remark that
						\begin{enumerate}[(i)]
							\item Part (a) in \autoref{them:bind}  serves as a relaxed simultaneously necessary and sufficient condition of the DWR objective exactness with a finite optimal  value, bounded optimal set, and at most $\underline{{m}}$-dimensional binding LMIs at optimality, provided that the non-binding constraints are parallel in the projected space. This  condition can be  more general than that in \autoref{them:exactobj} since $\underline{m}\leq \tilde{m}$ implies that $ \F^{\underline{m}}\left(\conv(\X)\right) \subseteq \F^{\tilde{m}}\left(\conv(\X)\right)$;
							\item \autoref{them:bind} provides a fresh geometric angle for illustrating the significance of the binding LMIs when determining the  objective exactness;
							\item It is worth mentioning that \autoref{them:bind} requires a bounded DWR optimal set. The DWR with a unbounded optimal set may not always achieve the objective exactness as shown in
							\autoref{egbd} below. On the other hand, the bounded set $\C_{\rel}$ implies that the optimal set must be bounded; 
							\item 
							The assumption on the non-binding LMIs  in \autoref{them:bind} is to guarantee the objective exactness for any $m$ LMIs. 
							Note that if there is only one non-binding LMI, then the assumption readily holds; and 
							\item Given the domain set $\X:=\{\bm X\in \S_+^n: \rank(\bm X)\le 1\}$, our \autoref{them:bind}  generalizes the classical QCQP result in \cite{ye2003new} with two quadratic constraints, where the authors showed that the DWR objective exactness holds if one of the quadratic constraints is not binding under the Slater conditions of both the DWR and its dual, 
							which essentially implies the boundedness of both the DWR optimal value and optimal set. 
							Applying our \autoref{them:bind}, we arrive at a more general conclusion by relaxing both primal and dual Salter conditions,  as summarized in \autoref{cor:bind}.
							For example, \autoref{cor:bind} can cover  the case that  the DWR feasible set $\C_{\rel}$ is a singleton, while \cite{ye2003new} cannot.
						\end{enumerate}
						
						\begin{example}\label{egbd} \rm
							Using the same domain set $\X=\{\bm X\in \S_+^2: \rank(\bm X)\le 1, X_{12}=0\}$ and its closed convex hull $\clconv(\X)=\{\bm X\in \S_+^2: X_{12}=0\}$ as  \autoref{eg:ray}, we consider $m=\tilde m=2$ LMIs- $X_{11}\le X_{22}$ and $X_{22}\ge1$. The corresponding feasible sets $\C$ and $\C_{\rel}$ of RCOP \eqref{eq_rank} and its DWR \eqref{eq_rel_rank}  are 
							\[ \C:=\{ \bm X\in \X:  X_{11}\le X_{22}, X_{22}\ge 1  \},  \  \ \C_{\rel}:=\{ \bm X \in \clconv(\X):  X_{11}\le X_{22}, X_{22}\ge 1 \},\]
							which are marked as a red line and a shadow area in Figure \ref{udc} and Figure \ref{udcrel}, respectively.
							
							By setting the objective function to be $X_{22}-X_{11} $, the DWR optimal value is $\V_{\rel}=0$ with the unbounded  optimal set $\{(X_{11}, X_{22}): X_{11}=X_{22}, X_{22}\ge 1\}$.
							There is an optimal solution with only one binding LMI, i.e., $\underline{m}=1$. Since only one LMI is not binding, the parallel assumption required by \autoref{them:bind} naturally holds  as indicated in Part (iv) above. Although any no larger than one-dimensional face of $\clconv(\X)$ is contained in domain set $\X$, we have $\V_{\opt}=1> \V_{\rel}=0$, i.e., the objective exactness fails in this example. 
							Thus, our condition in \autoref{them:bind} is not 
							sufficient for objective exactness when  the optimal set is unbounded. 
							\qedA
						\end{example}
						
						\begin{figure}[htb]
							\centering
							\vskip -0.1in
							\subfigure[$\C$ ]{
								{\includegraphics[width=0.21\columnwidth]{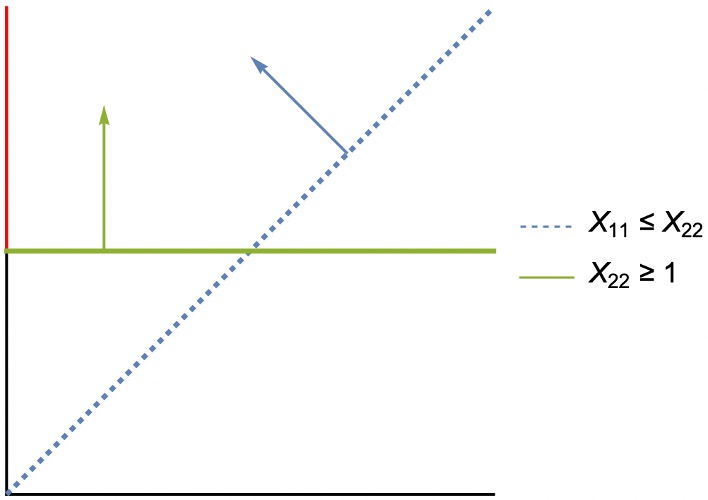}}	\label{udc}}
							\hspace{7em}
							\subfigure[$\C_{\rel}$ ]{
								{\includegraphics[width=0.21\columnwidth]{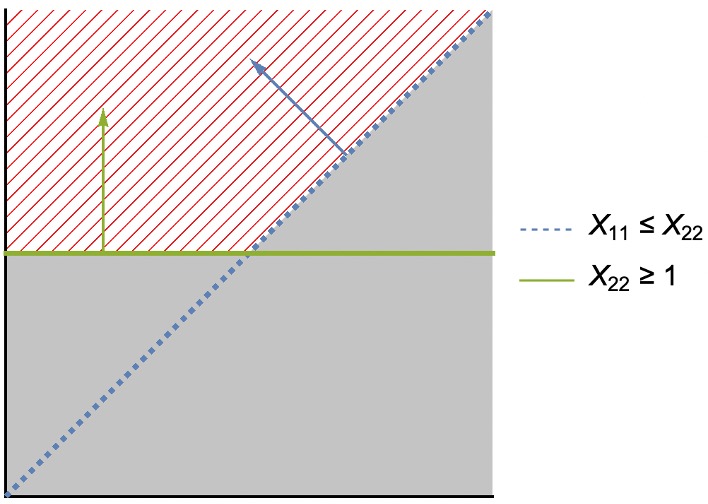}}	\label{udcrel}}
							\caption{Illustration of Sets in \autoref{eg:sdqcqp}  and $\C_{\rel}\neq \clconv(\C)$.  }
							\label{fig:ubd}
						\end{figure}
						\vspace{-1.5em}
						\begin{corollary}\label{cor:bind}
							For QCQP \eqref{qcqp} with two  one-sided quadratic constraints, suppose that its DWR has a finite optimal value and a bounded optimal set. Then the objective exactness holds if one  LMI  constraint of DWR is not binding at optimality.
						\end{corollary}
						\begin{proof}
							This result can be derived by leveraging   \autoref{lem:psd} and \autoref{them:bind}.
							\qed
						\end{proof}
						
						Let us now turn to two application problems in fair unsupervised learning whose DWRs can achieve objective exactness  by using the binding LMI-based condition in \autoref{them:bind}.\par
						\noindent \textbf{Fair PCA (FPCA).} Fair PCA (FPCA) extends the conventional PCA to the dataset concerning different groups. The seminal work  \citep{samadi2018price} presented a real-world example to show that  the conventional PCA can cause gender bias and introduced the notion of FPCA to improve the fairness of the conventional PCA. Their follow-up work \citep{tantipongpipat2019multi} proposed  the following rank-$k$ constrained formulation for FPCA that seeks to optimize the dimensionality reduction  over $m$ different groups in a fair way:
						\begin{align}\label{eq:fpca}
							\text{(FPCA)} \quad \V_{\opt}:=  \max_{(z, \bm X) \in \Re\times\X} \left\{z: z\le \langle  \bm A_i, \bm X \rangle, \forall i \in [m]\right\}, \ \ \X:=\left\{\bm X\in \S_+^n:  \rank(\bm X)\le k, ||\bm X||_2\le 1\right\},
						\end{align}
						where  $\bm A_i \in \S_+^n$ denotes the covariance matrix of the $i$-th group for each  $i\in [m]$. It is important to note that FPCA fits in our RCOP framework \eqref{eq_rank}  by observing that at optimality, 
						$z$ is equal to at least one of $\{ \langle \bm A_i, \bm X \rangle\}_{i\in [m]}$, i.e., FPCA \eqref{eq:fpca} is equivalent to
						\begin{align*}
							\text{(FPCA)} \quad  \V_{\opt}:= \max_{j\in [m]} \max_{\bm X \in\X} \left\{\langle\bm A_j, \bm X \rangle:\langle\bm A_j, \bm X \rangle\le \langle  \bm A_i, \bm X \rangle, \forall i \in [m]\right\}.
						\end{align*}
						For ease of analysis, we focus on FPCA \eqref{eq:fpca} with auxiliary variable $z$. 
						We first derive the properties of the closed convex hull of FPCA domain set $\X$.
						\begin{lemma}\label{lem:fpca}
							Suppose that the domain set $\X:=\left\{\bm X\in \S_+^n: \rank(\bm X)\le k, ||\bm X||_2\le 1 \right\}$, then we have
							\begin{enumerate}[(i)]
								\item $\clconv(\X)=\conv(\X):=\left\{\bm X\in \S_+^n: \tr(\bm X)\le k, ||\bm X||_2\le 1 \right\}$; and
								\item Any no larger than one-dimensional face of  $\clconv(\X)$ is contained in $\X$, i.e.,  $F \subseteq \X$ for all $F\in \F^{1}(\clconv(\X))$.
							\end{enumerate}
						\end{lemma}
						\begin{proof}
							\textbf{Part (i).} Note that $\clconv(\X)=\conv(\X)$ since the domain set $\X$ is compact.  Suppose $\bm X  = \bm Q\Diag(\bm \lambda) \bm Q^{\top}$ denotes the eigen-decomposition of matrix $\bm X$ in $\X$, then the domain set $\X$ is equivalent to
							$$\X:=\{\bm X\in \S_+^n: \bm \lambda \in \Re^n_+, ||\bm \lambda||_0 \le k, ||\bm \lambda||_{\infty}\le 1, \bm Q \bm  Q^{\top} = \bm I_n, \bm X  = \bm Q\Diag(\bm \lambda) \bm Q^{\top} \}.$$
							It is known  that $\conv \left( \left\{\bm \lambda\in \Re_+^n: ||\bm \lambda||_0 \le k, ||\bm \lambda||_{\infty}\le 1 \right\} \right) = \left\{\bm \lambda \in \Re_+^n: ||\bm \lambda||_1 \le k, ||\bm \lambda||_{\infty}\le 1 \right\}$ (see, e.g.,  \cite{argyriou2012sparse}).  Hence, we have
							\begin{align*}
								\conv(\X)&=\conv\left(\{\bm X\in \S_+^n: \bm{\lambda}\in \conv\{\bm \lambda \in \Re^n_+: ||\bm \lambda||_0 \le k, ||\bm \lambda||_{\infty}\le 1\}, \bm Q \bm  Q^{\top} = \bm I_n, \bm X  = \bm Q\Diag(\bm \lambda) \bm Q^{\top} \} \right)\\
								&=\conv\left(\{\bm X\in \S_+^n: \bm \lambda \in \Re^n_+, ||\bm \lambda||_1 \le k, ||\bm \lambda||_{\infty}\le 1 , \bm Q \bm  Q^{\top} = \bm I_n, \bm X  = \bm Q\Diag(\bm \lambda) \bm Q^{\top} \} \right)\\
								&=\left\{\bm \X \in \S_+^n: \tr(\bm X) \le k, ||\bm X||_2\le 1 \right\},
							\end{align*}
							where the last equation stems from projecting out variables $\bm\lambda, \bm Q$ and the identities: 
							$||\bm \lambda||_1 = \tr(\bm X)$ and $||\bm \lambda||_{\infty} = ||\bm X||_2$ for any $\bm X\in \S_+^n$. 
							
							\textbf{Part (ii).}  If there exists a one-dimensional face $F \subseteq \clconv(\X)$  not contained in $\X$, then we can find a matrix $\hat{\bm X}\in F\setminus \X$ of  a rank $r$ greater than $k$.  Let $\hat{\bm X} = \bm Q_1 \bm \Lambda_1 \bm Q_1^{\top} + \bm Q_2 \bm \Lambda_2 \bm Q_2^{\top}$ denote the eigen-decomposition of matrix $\hat{\bm X}$, where  the eigenvalues attaining one compose the diagonal matrix $\bm \Lambda_1 \in \S_{++}^{r_1}$ and the other fractional eigenvalues are in $\bm \Lambda_2 \in \S_{++}^{r_2}$.  Note that since $\tr(\hat{\bm X})\leq k, k<r$ and $\|\hat{\bm X}\|_2\le 1$, we must have
							$r_1+r_2=r$ and $r_2 \ge 2$. Following the proof in \autoref{lem:psd}, as ${r_2}(r_2+1)/2\ge 3$, we can construct a nonzero symmetric matrix $\bm \Delta \in \S^{r_2}$ satisfying $\bm Q_2  \bm \Delta  \bm Q_2^{\top} \perp F$ and  $\hat{\bm X} \pm \delta \bm Q_2  \bm \Delta  \bm Q_2^{\top} \in F$ for significantly small $\delta>0$. A contradiction. \qed
						\end{proof}
						
						According to  Part (i) in \autoref{lem:fpca}, the DWR of FPCA \eqref{eq:fpca} is defined by
						\begin{align}\label{eq:fpca_dwr}
							\V_{\rel}:=\max_{(z, \bm \X) \in \Re\times\clconv(\X)} \left\{z: z\le \langle  \bm A_i, \bm X \rangle, \forall i \in [m]  \right\},  
						\end{align}
						where $\clconv(\X)=\left\{\bm X\in \S_+^n: \tr(\bm X)\le k, ||\bm X||_2\le 1 \right\}$.
						Part (ii) in \autoref{lem:fpca} indicates that the domain set $\X$ of FPCA \eqref{eq:fpca} contains all zero- and one-dimensional faces of $\clconv(\X)$. In the following, we show by using the proposed conditions that the convex hull exactness holds for FPCA whenever $\tilde{m}=2$, i.e., there are two independent covariance matrices from $m$ groups.
						Note that the set $\clconv(\X)$ in DWR \eqref{eq:fpca_dwr} contains no line, and DWR \eqref{eq:fpca_dwr} always has finite objective value.
						%
						%
						
						
						\begin{corollary}\label{cor:fpca}
							For FPCA \eqref{eq:fpca} with $\tilde{m}=2$ linearly independent covariance matrices, its  DWR  \eqref{eq:fpca_dwr}  admits the convex hull exactness.
						\end{corollary}
						\begin{proof}
							Without loss of generality, the  feasible sets of the FPCA \eqref{eq:fpca}  and DWR  \eqref{eq:fpca_dwr} with $\tilde{m}=2$ can be written as
							\begin{align*}
								&\C:= \left\{(z, \bm X) \in \Re\times \S_+^n: z \le \langle \bm A_1, \bm X \rangle, z \le \langle  \bm A_2, \bm X \rangle, \rank(\bm X)\le k, ||\bm X||_2\le 1 \right\},\\	
								&\C_{\rel}:= \left\{(z, \bm X) \in \Re\times \S_+^n: z \le \langle \bm A_1, \bm X \rangle, z \le \langle  \bm A_2, \bm X \rangle, \tr(\bm X)\le k, ||\bm X||_2\le 1 \right\}.
							\end{align*}
							We see that the recession cone of set  $\C_{\rel}$ in DWR  \eqref{eq:fpca_dwr} is $\reccone(\C_{\rel}):=\{(z, \bm 0): z \le 0 \}$, and $\reccone(\C_{\rel}) =\reccone(\clconv(\C))$. Thus, to prove $\C_{\rel} = \clconv(\C)$, we  only need to show the extreme point  exactness of the set $\C_{\rel}$. For any extreme point $(\hat{z}, \hat{\bm X})$ in set $\C_{\rel}$,  two cases are discussed below.
							\begin{enumerate}[1)]
								\item If only one LMI is binding at $(\hat{z}, \hat{\bm X})$, given that any no larger than one-dimensional face of $\clconv(\X)$ is contained in $\X$ shown in \autoref{lem:fpca}, using the extreme point exactness result in \autoref{them:ext}, we have that $\hat{\bm X}$ is rank-$k$ and $(\hat{z}, \hat{\bm X}) \in \C$.
								\item If both LMIs are binding at $(\hat{z}, \hat{\bm X})$, then  $\hat{\bm X}$ is also an extreme point of set
								$\left\{\bm X \in \clconv(\X): \langle \bm A_1 -\bm A_2, \bm X \rangle = 0\right\}$. 
								Similarly, we can show that  $\hat{\bm X}\in \X$. Thus, $(\hat{z}, \hat{\bm X}) \in \C$. 	\qed
							\end{enumerate}
						\end{proof}

						We remark that (i) although set $\C_{\rel}$ in DWR \eqref{eq:fpca_dwr} is unbounded  (as variable $z$ can be $-\infty$), its convex hull exactness still holds with $\tilde{m}=2$ groups since the recession cones of $\C_{\rel}$ and $\clconv(\C)$ coincide; 
						(ii) our result extends the extreme point exactness for FPCA in  \cite{tantipongpipat2019multi}, where the authors proved that the upper bound of the rank of all extreme points in the set $\C_{\rel}$ is linear in $\sqrt{{m}}$ and the rank bound becomes $k$ when ${m}=2$; and 
						(iii) this result can be further generalized to objective exactness when there are $\underline{m}=2$-dimensional binding LMIs of DWR \eqref{eq:fpca_dwr} at optimality according to our  \autoref{them:bind}, which is summarized in \autoref{cor:fpca2} below. Note that DWR \eqref{eq:fpca_dwr}  always yields finite optimal value and bounded optimal set.
						\begin{corollary}\label{cor:fpca2}
							For FPCA \eqref{eq:fpca}, suppose that (i) its  DWR  \eqref{eq:fpca_dwr} has an optimal solution with $\underline{m}\le2$-dimensional LMIs denoted  by $T\subseteq [m]$ and (ii) matrices  $\{\bm A_j-\Proj_{\H}(\bm A_j)\}_{[m]\setminus T}$  are parallel with the same direction, where set $\H:=\spa(\{\bm A_i\}_{i\in T})$.
							Then  its DWR  \eqref{eq:fpca_dwr}   admits the objective exactness.
						\end{corollary}
						
						\noindent\textbf{Fair SVD (FSVD). }Another significant application of our relaxed condition in \autoref{them:bind} is the fair SVD (FSVD), which can be formulated as
						\begin{align}\label{eq:fsvd}
							\text{(FSVD)} \quad  \V_{\opt}:=  \max_{(z, \bm \X) \in \Re\times \X} \left\{z: z\le \langle  \bm A_i, \bm X \rangle, \forall i \in [m]  \right\},  \ \ \X:=\{\bm X\in \Re^{n\times p}: \rank(\bm X)\le k, ||\bm X||_2\le 1 \},
						\end{align}
						where  $\bm A_i \in \Re^{n\times p}$ denotes the data matrix of the $i$-th group for each  $i\in [m]$. Different from FPCA \eqref{eq:fpca}, the FSVD \eqref{eq:fsvd} aims to seek a fair representation learning of $m$ different data matrices that are non-symmetric. Similar to \autoref{lem:fpca} for FPCA \eqref{eq:fpca}, the next lemma characterizes the convex hull  and  facial inclusion property of  the domain set $\X$ in FSVD \eqref{eq:fsvd}. 
						\begin{lemma}\label{lem:fsvd}
							Suppose that domain set $\X:=\left\{\bm X\in \Re^{n\times p}: \rank(\bm X)\le k, ||\bm X||_2\le 1 \right\}$, then we have
							\begin{enumerate}[(i)]
								\item $\clconv(\X)=\conv(\X):=\left\{\bm X\in \Re^{n\times p}: ||\bm X||_* \le k, ||\bm X||_2\le 1 \right\}$, where $ ||\cdot||_* $ is the nuclear norm; and
								\item Any no larger than one-dimensional face of  $\clconv(\X)$ is contained in $\X$, i.e., $F \subseteq \X$ for all $F\in \F^{2}(\clconv(\X))$.
							\end{enumerate}
						\end{lemma}
						\begin{proof}The proof is similar to \autoref{lem:fpca} and thus is omitted. \qed
						\end{proof}
						
						Note that the domain set $\X$ in FPCA \eqref{eq:fpca} contains any at most one-dimensional face in its convex hull, and we show that the convex hull exactness holds  when there are $\tilde m=2$ linearly independent groups in \autoref{cor:fpca}.  According to Part (ii) of \autoref{lem:fsvd}, 
						it is natural to extend the convex hull exactness for FSVD \eqref{eq:fsvd} up  to two linearly independent groups (i.e., $\tilde{m}\le 2$) as below.
						\begin{corollary}\label{cor:fsvd}
							For the FSVD \eqref{eq:fsvd} with $\tilde{m}\le 2$ linearly independent data matrices, its  DWR admits the convex hull  exactness.
						\end{corollary}
						\begin{proof}
							The proof directly follows from  \autoref{lem:fsvd} and \autoref{cor:fpca}. \qed
						\end{proof}
						
						We remark that this is the first-known convex hull exactness result for FSVD  \eqref{eq:fsvd} when there are two groups, and this result can be further extended to objective exactness when its DWR has at most $\underline{m}=2$-dimensional  binding LMIs.
						\begin{corollary}\label{cor:fsvd2}
							For FPCA \eqref{eq:fsvd}, suppose that (i) its  DWR has an optimal solution with  $\underline{m}\le 2$-dimensional   binding LMIs denoted by $T\subseteq [m]$, and (ii) matrices  $\{\bm A_j-\Proj_{\H}(\bm A_j)\}_{[m]\setminus T}$  are parallel with the same direction with $\H$ denoting the space spanned by binding data matrices. Then  its DWR admits the objective exactness.
						\end{corollary}
						
						\subsection{Objective Exactness Under Setting (IV): Relaxed Simultaneously Necessary and Sufficient Condition based on the  Nonzero Optimal Lagrangian Multipliers}\label{sec:dual}
						The previous subsection shows that whether the DWR  \eqref{eq_rel_rank} achieves objective exactness highly depends on its binding LMIs. 
						Motivated by the fact that binding LMIs may also have zero-value Lagrangian multipliers, this subsection  further relaxes the simultaneously necessary and sufficient  condition for objective exactness  by leveraging the Lagrangian multipliers of DWR \eqref{eq_rel_rank}. This result allows us the flexibility to cover and generalize the objective exactness results for more applications present in the literature. Analogous to the previous subsection, we still focus on the  one-sided LMIs in RCOP \eqref{eq_rank}.
						
						We first show an example of the DWR in which the objective exactness holds, and the dimension of  binding constraints is strictly larger than the number of nonzero optimal Lagrangian multipliers, i.e., $\underline m > \underline m^*$. More importantly, in this example, both conditions in \autoref{them:exactobj} and \autoref{them:bind} fail to cover the objective exactness of the DWR. This motivates us to further relax the conditions from the perspective of Lagrangian multipliers.
						\begin{example} \label{eg:sd2qcqp} \rm
							Using the domain set $\X:= \{\bm X\in \S_+^2: \rank(\bm X) \le 1, X_{12}=0\}$ and its closed convex hull  $\clconv(\X):=\{\bm X\in \S_+^2: X_{12}=0\}$ same as \autoref{eg:ray}, we consider  $m=\tilde m=2$ LMIs: $X_{11}\le X_{22}, 2X_{11}\le X_{22}$. Hence, we have sets $\C$ and $\C_{\rel}$ defined as 
							\[ \C=\{ \bm X\in \X:  X_{11}\le X_{22},  2X_{11}\le X_{22}  \},   \ \  \C_{\rel}=\{ \bm X \in \clconv(\X):  X_{11}\le X_{22},  2X_{11}\le X_{22} \},\]
							where set $\C$ in this example is the vertical red solid line as shown in Figure \ref{sd2c}, and set $\C_{\rel}$ is presented in red shadow area in Figure \ref{sd2crel}. Note that both sets are unbounded.
							
							
							If we set the objective function of the DWR to be $X_{22}$, then the objective exactness holds, i.e., $\V_{\rel}=\V_{\opt}=0$ with the same optimal  point $a_1$. 
							It is seen that the point $a_1$ falls on $\underline m:=2$-dimensional   binding LMIs in Figure \ref{sd2crel}. However, \autoref{them:exactobj} and \autoref{them:bind}  cannot be used to show the objective exactness in this example   as the two-dimensional face of set $\clconv(\X)$ is not contained in $\X$. On the other hand, there is an optimal dual  solution (i.e., $\bm \mu^*=(0, 1)^\top $) of the DWR that has only $\underline m^*:=1$  nonzero Lagrangian multiplier. 
							In fact, we show that the smallest number of nonzero optimal Lagrangian multipliers is upper bounded by the dimension of binding LMIs. Therefore, to leverage the number of nonzero optimal Lagrangian multipliers,  we are motivated to derive another simultaneously necessary and sufficient condition for objective exactness, 
							further relaxing the one based on binding LMIs in \autoref{them:bind}.
							\qedA
						\end{example}
						
						\begin{figure}[htbp]
							\centering
							\vskip -0.2in
							\subfigure[$\C=\clconv(\C)$ ]{
								{\includegraphics[width=0.27\columnwidth]{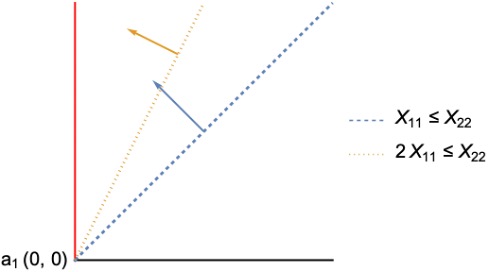}}	\label{sd2c}}
							\hspace{6em}
							\subfigure[$\C_{\rel}$ ]{
								{\includegraphics[width=0.27\columnwidth]{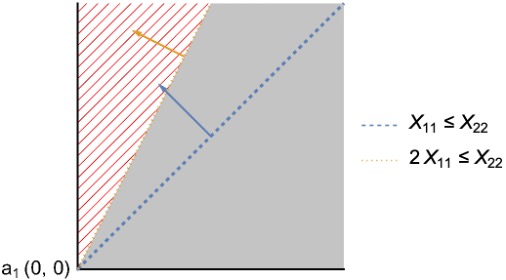}}	\label{sd2crel}}
							\caption{Illustration of Sets in \autoref{eg:sd2qcqp}  and $\C_{\rel} \neq \clconv(\C)$.  }
							\label{fig:sd2}
							\vskip -0.1in
						\end{figure}


						Next we give a geometric interpretation of the Lagrangian multipliers  of DWR  by the normal cone of set $\C_{\rel}$. 
						
						\begin{definition}[Normal Cone]  \label{def:norm}
							Let $D\subseteq \Re^n$ be a closed convex set. The normal cone of set $D$ at point $\bm x$ is defined by
							$\N_{D}(\bm x) = \left\{\bm g \in \Re^n:  \bm g^{\top}(\bm y-\bm x)\le 0, \forall  \bm y \in D \right\}$. 
						\end{definition}
						
						\begin{proposition} \label{them:opt}
							Given a nonempty closed  domain set $\X$ with its closed convex hull $\clconv(\X)$ being line-free and $\V_{\rel} >-\infty$    in DWR \eqref{eq_rel_rank}, then the followings  hold:
							\begin{enumerate}[(i)]
								\item 	A  feasible solution $\bm X^* \in \C_{\rel}$ is  optimal to DWR \eqref{eq_rel_rank} if and only if  $-\bm A_0 \in \N_{\C_{\rel}}(\bm X^*)$;
								
								\item Suppose that   $b_i^l=-\infty$ for all $i\in [m]$ in RCOP \eqref{eq_rank},
								and DWR \eqref{eq_rel_rank} satisfies the relaxed Slater condition, i.e.,  $\C_{\rel}\cap \ri(\clconv(\X)) \neq \emptyset$.  Then for an optimal solution $\bm X^*$ to DWR with binding LMIs from $T^*\subseteq [m]$,  there  exist  Lagrangian multipliers $\bm \mu^* \in \Re_+^{m}$ such that $\supp(\bm \mu^*)\subseteq T^*$ and
								\[ -\bm A_0 - \sum_{i\in T^*} \mu^*_i \bm A_i \in \N_{\clconv(\X)}(\bm X^*), \]
								where  $\N_{\clconv(\X)}(\bm X^*)$ denotes the normal cone of  set $\clconv(\X)$ at an optimal solution $\bm X^*$. 
							\end{enumerate}
						\end{proposition}
						\begin{proof}
							\noindent \textbf{Part (i).}	Since the feasible set $\C_{\rel}$ of the DWR is closed convex, then the optimality condition is  $\langle\bm A_0 , {\bm X^*} \rangle \le \langle\bm A_0 , \bm X \rangle$ for any $\bm X \in \C_{\rel}$, which is equivalent to $-\bm A_0 \in \N_{\C_{\rel}}(\bm X^*)$ according to \autoref{def:norm} of the normal cone.
							
							\noindent\textbf{Part (ii).} Since DWR \eqref{eq_rel_rank} satisfies the relaxed Slater condition, i.e.,  $\C_{\rel}\cap \ri(\clconv(\X)) \neq \emptyset$, the normal cone of the intersection set $\C_{\rel}$ is equal to the intersection of normal cones \citep{burachik2005simple}. Therefore,  there exists $\bm \mu^* \in \Re_+^m$ such that
							\begin{align*}
								-\bm A_0 \in \N_{\C_{\rel}}(\bm X^*) =  \N_{\clconv(\X)\cap \{\bm{X}\in \Q: \langle \bm A_i,\bm{X}\rangle\leq b_i^u,\forall i\in [m]\}}(\bm X^*) =  \sum_{i\in T^*} \mu_i^* \bm A_i + \N_{\clconv(\X)}(\bm X^*).
							\end{align*}
							Note that for each $i\in [m]\setminus T^*$ non-binding LMI, we have  $\mu^*_i=0$. This completes the proof.
							\qed
						\end{proof}
						
						We  remark that the relaxed Slater condition in \autoref{them:opt} is used to provide an explicit description of the normal cone of the intersection set $\C_{\rel}$, and this condition can be further relaxed (see, e.g.,  \cite{burachik2005simple}), which is omitted in this paper due to page limit. Part (ii) of \autoref{them:opt} implies that $\underline m^*\le \underline  m$, which enables us to improve the condition in \autoref{them:bind}. Before that, we use \autoref{them:opt} to derive the objective exactness for several QCQP examples below.

						\noindent{\bf One-sided QCQP. } To exploit \autoref{them:opt}, we study the one-sided QCQP \eqref{eq_qcqp}, where we let $b_i^l =-\infty$ for all $i\in [m]$.
							Recent studies reveal the DWR of one-sided QCQP can achieve objective exactness when the matrix coefficients $\{\bm A_0\}\cup \{\bm A_i\}_{i\in [m]}$ exhibit favorable properties. 
							For instance, \cite{kim2003exact} proved that $\V_{\rel}=\V_{\opt}>-\infty$ when all the off-diagonal elements of  matrices $\bm  A_0$ and $\bm A_1, \cdots, \bm A_m$ are nonpositive. Mapping the matrix coefficients of the one-sided QCQP  into an undirected graph $\G$ that there is an edge in $\G$ if any of $(A_0)_{ij}$ and $(A_1)_{ij}, \cdots, (A_m)_{ij}$ is nonzero, the work \citep{sojoudi2014exactness} generalized  \cite{kim2003exact} and proved that $\V_{\rel}=\V_{\opt}>-\infty$ when all the off-diagonal elements of matrices $\bm  A_0$ and $\bm A_1, \cdots, \bm A_m$ are \textit{sign-definite}, i.e.,
							for any pair $(i,j)$ with $i\neq j$, $(A_0)_{ij}$ and $(A_1)_{ij}, \cdots, (A_m)_{ij}$ are either all nonpositive or all nonnegative, and the signs of  matrices further satisfy
							\[ \prod_{i\in [|S|]} \sigma_{S(i),S(i+1)} = (-1)^{|S|},  \ \ \forall S\text{ is a cycle of graph $\G$},  \]
							where we let $S(i)$ denote the $i$th node
							in the cycle $S$, let $S(|S|+1)=S(1)$, and let $\sigma_{ij}=-1$ if $(A_0)_{ij}, (A_1)_{ij}, \cdots, (A_m)_{ij}$ are all nonpositive and $\sigma_{ij}=1$ if $(A_0)_{ij}, (A_1)_{ij}, \cdots, (A_m)_{ij} $ are all positive.
							In recent follow-up works \citep{burer2020exact, kilincc2021exactness}, the authors reproved the result in
							\cite{sojoudi2014exactness} from different angles when applying to the one-sided  diagonal QCQP. Our following theorem  provides a  unified analysis of these cases that achieve objective exactness  using \autoref{them:opt}. 
							
							\begin{restatable}{corollary}{cordiag}\label{cor:diag}
								Suppose that $b_i^l =-\infty$ for all $i\in [m]$ in	QCQP \eqref{eq_qcqp}. Then its DWR admits the objective exactness for any linear objective function such that $\V_{\rel}>-\infty$ and any $m$ LMIs of dimension $\tilde m$
								when  any of the following conditions holds:
								\begin{enumerate}[(i)]
									\item All the off-diagonal elements of  matrices $\bm  A_0, \bm A_1, \cdots, \bm A_m$ are sign-definite, and for each cycle $S$ in graph $\G$, we have $ \prod_{i\in [|S|]} \sigma_{S(i),S(i+1)} = (-1)^{|S|}$ with $S(|S|+1)=S(1)$;
									\item  Matrices $\bm  Q_0, \bm Q_1, \cdots, \bm Q_m$  are diagonal and vectors $\bm  q_0, \bm q_1, \cdots, \bm q_m$  are sign-definite.
								\end{enumerate}
								Note that  $\{\bm  Q_0\}\cup \{\bm Q_i\}_{i\in [m]}$ are symmetric, and $\bm A_i  = \begin{pmatrix}
									0 & \bm q_i^{\top}/2\\
									\bm q_i/2 & \bm Q_i
								\end{pmatrix}$ for each $i\in \{0\}\cup [m]$. 
							\end{restatable}
							\begin{proof}
								See Appendix \ref{proof:cordiag}. \qed
							\end{proof}

							Finally, we conclude this subsection by showing a relaxed simultaneously necessary and sufficient condition for the DWR objective exactness  by analyzing the number of its  nonzero optimal Lagrangian multipliers.
							\begin{theorem}\label{them:dual}
								Given a nonempty closed  domain set $\X$ with its closed convex hull $\clconv(\X)$ being line-free and $b_i^l =-\infty$ for any $i\in [m]$ in RCOP \eqref{eq_rank}. Then the followings are equivalent.
								\begin{enumerate}[(a)]
									\item \textbf{Inclusive Face:} Any no larger than $\underline{m}^*$-dimensional face of set $\clconv(\X)$ is contained in the domain set $\X$, i.e., $F \subseteq \X$ for all $F\in \F^{\underline{m}^*}(\clconv(\X))$; 
									\item \textbf{Objective Exactness:} The DWR  \eqref{eq_rel_rank} has the same optimal value as problem \eqref{eq_rank} (i.e., $\V_{\opt}=\V_{\rel}$) for any linear objective function and $m$ LMIs such that (i) the relaxed Slater condition holds, i.e.,  $\C_{\rel}\cap \ri(\clconv(\X)) \neq \emptyset$, (ii) $\V_{\rel}>-\infty$, (iii) the optimal set of DWR \eqref{eq_rel_rank} is bounded, (iv) there are $\underline{m}^*$ nonzero optimal Lagrangian multipliers indexed by set $T\subseteq [m]$ corresponding to the DWR, and (v) matrices $\{ \bm A_j - \Proj_{\H}(\bm A_j) \}_{j \in [m]\setminus T}$ are parallel with the same direction with $\H:= \spa( \{\bm A_i\}_{i\in T})$.
								\end{enumerate}
							\end{theorem}
							\begin{proof}
								We split the proof into two parts.
								\begin{enumerate}[(i)]
									\item  
									We first prove (a)$\Longrightarrow$ (b). 	
									Let ${\bm X^*}$ denote an optimal solution of DWR \eqref{eq_rel_rank}, corresponding to the optimal  dual solution $\bm \mu^* \in \Re^m_+$ with only $\underline{m}^*$ nonzero Lagrangian multipliers.
									We let $F \subseteq \clconv(\X)$ denote the smallest-dimensional face  of set $\clconv(\X)$ containing $\bm X^*$. Suppose that the  dimension of face $F$ is $d$.
									Then there are two cases to be discussed depending on the dimension $d$.
									\begin{enumerate}[(1)]
										\item Suppose that $d\le \underline{m}^* $. Then according to Part (a), we have that $\bm X^* \in F \subseteq  \X$ and the objective exactness of the DWR  directly follows, i.e., $\V_{\opt}= \V_{\rel}$.
										
										\item Suppose that $d>  \underline{m}^* $.  
										%
										%
										Given the relaxed Slater condition, using Part (ii) of  \autoref{them:opt}, the  primal and dual optimal solutions $\bm X^*$ and $\bm \mu^*$ satisfy
										\begin{align}\label{eq:kkt}
											-\bm A_0 - \sum_{i \in T}\mu^*_i \bm  A_i  \in \N_{\clconv(\X)}(\bm X^*),
										\end{align}
										where  set $T$ denotes the support of the dual optimal solution $\bm \mu^*$  with $|T|= \underline{m}^*$.
										
										Since $F$ is the smallest-dimension face containing $\bm X^*$, point $\bm X^*$ lies in the relative interior of set $F$.	According to \autoref{def:norm} of normal cone, it follows that for any $\hat{\bm V}\in \N_{\clconv(\X)}(\bm X^*)$,
										$$ \langle\hat{\bm V}, \bm  X^* - \hat{\bm X}\rangle = 0, \forall \hat{\bm  X} \in F.$$ Then let us define  a set 
										$$\hat{C}_{\rel}:=F\cap \{\bm X \in \Q:  \langle\bm  A_i, \bm X \rangle= b_i^u, \forall i \in T\},$$ 
										which is neither empty since $\langle\bm  A_i, \bm X^* \rangle= b_i^u$ nor a singleton as the dimension of face $F$ is greater than $|T|=\underline m^*$.
										Thus, for any $\hat{\bm{X}} \in \hat{C}_{\rel}$ and some $\hat{\bm V}\in \N_{\clconv(\X)}(\bm X^*)$, using the results above, we have 
										\[-\bm A_0= \sum_{i \in [\underline{m}^*]} \mu_i^*\bm  A_i+ \hat{\bm V},   \ \ 
										\langle \bm  A_i, \bm X^* \rangle = \langle \bm  A_i, \hat{\bm{X}} \rangle = b_i^u, \forall i \in T,  \ \
										\langle \hat{\bm V}, \bm X^* - \hat{\bm{X}}\rangle = 0, \]
										where the first equation is from the optimality condition \eqref{eq:kkt}.  Therefore, we can conclude that $\langle\bm A_0, \hat{\bm{X}} \rangle  = -\langle\sum_{i \in [\underline{m}^*]} \mu_i^*\bm  A_i+ \hat{\bm V},  \hat{\bm{X}} \rangle = -\langle\sum_{i \in [\underline{m}^*]} \mu_i^*\bm  A_i+ \hat{\bm V},  \bm{X}^* \rangle = \langle\bm A_0, \bm X^* \rangle=  \V_{\rel}$ for any $\hat{\bm{X}} \in \hat{C}_{\rel}$.
										
										Following the proof of \autoref{them:bind}, there must exist an extreme point $\bm Y^*$ in set $\hat{C}_{\rel}$ that satisfies $m$ LMIs and belongs to  some at most $\underline{m}^*$-dimensional face of face $F$, which is exactly some face in $\F^{\underline{m}^*}(\clconv(\X))$ and thus a subset of $ \X$. Hence, we have $\bm Y^*\in \C$ and  $\V_{\opt} =\V_{\rel}$. 
									\end{enumerate} 
									
									\item We next prove (b)$\Longrightarrow$ (a).  According to the relaxed Slater condition,
									set $\clconv(\X)$ must have the nonempty relative interior. Suppose that $\hat{\bm X}\in \ri(\clconv(\X))$. Then there exists a ball $B(\hat{\bm{X}},\epsilon)$ with radius $\epsilon>0$ such that $B(\hat{\bm{X}},\epsilon)\cap \aff(\clconv(\X))\subseteq \clconv(\X)$. 
									

									Given a $d$-dimensional proper face $F$ of set $\clconv(\X)$ with $d\le \{\dim(\clconv(X))-1, \underline{m}^*\}$, we must have $\hat{\bm{X}}\notin F$, otherwise $F=\clconv(\X)$. Then for any point $\bm X\in F$,  let us can construct $d$-dimensional LMI system $\H$ whose corresponding hyperplanes pass the point $\bm{X}$ and are tangent to the ball $B(\hat{\bm{X}}, \|\bm{X}-\hat{\bm{X}}\|/2)$. In addition, the intersection of $\H$ and $\clconv(\X)$  includes the ball $B(\hat{\bm{X}}, \|\bm{X}-\hat{\bm{X}}\|/2)\cap \aff(\clconv(\X))$, and let the rest $m-d$ number of LMIs have zero-valued technology matrices and zero-valued right-hand sides. Thus, we can construct a DWR set $\C_{\rel}=\H\cap \clconv(\X)$ and it satisfies relaxed slater condition since $B(\hat{\bm{X}}, \|\bm{X}-\hat{\bm{X}}\|/2)\cap \aff(\clconv(\X))\subseteq \C_{\rel}$. Based on the construction of $\H$ and $\C_{\rel}$, 
									the point $\bm{X}$ is an extreme point of set $\C_{\rel}$. Following the proof of \autoref{them:exactobj}, we have that $\bm{X}\in \C\subseteq \X$. Hence, $F\subseteq \X$.
									
									Now suppose that $d=\dim(\clconv(X))$, i.e., $F=\clconv(\X)$. Then for any $\bm{X}\in F$ but $\bm{X}\neq \hat{\bm X}$, we can do the same procedure as the previous  proof and obtain $\clconv(\X)\setminus \{\hat{\bm X}\}\subseteq \X$. Since  set $\X$ is closed, we must have $\clconv(\X)=\X$. This completes the proof.
									\qed
								\end{enumerate}
							\end{proof}
							
							For the simultaneously necessary and sufficient condition in \autoref{them:dual},  we remark that 
							\begin{enumerate}[(i)]
								\item This condition, together with that provided by \autoref{them:bind}, provide a unified primal and dual analysis for the DWR objective exactness.
								\item From Part (ii) in \autoref{them:opt}, we see that the smallest number of nonzero optimal Lagrangian multipliers is upper bounded by the dimension of  binding LMIs, i.e., $\underline{m}^* \le \underline m$. Hence, the condition in \autoref{them:dual} can be more general than that in \autoref{them:bind} based on binding LMIs since $\F^{\underline{m}^*}(\conv(\X)) \subseteq \F^{\underline m}(\conv(\X))$, which aligns with the findings in \autoref{eg:sd2qcqp}. 
								On the other hand,  we need the relaxed Slater condition in \autoref{them:dual}, while \autoref{them:bind} does not; and
								\item Analogous to  \autoref{cor:bind}, it is interesting to apply \autoref{them:dual} to the QCQP \eqref{qcqp} with two quadratic constraints and  generalize the result in \cite{ben2014hidden}. In fact, \autoref{cor:dual} below proves the objective exactness for the QCQP with two quadratic constraints when there is only one nonzero Lagrangian multiplier. However, the authors in \cite{ben2014hidden} proved the objective exactness for  only the SD-QCQP under the same condition. 
								It is worth mentioning that the work in \cite{ben2014hidden}  assumed the Slater condition, finite optimal value, and bounded optimal set, which satisfies the presumption in Part (b) of \autoref{them:dual}. 
							\end{enumerate}
							
							
							\begin{corollary}\label{cor:dual}
								For the QCQP \eqref{qcqp} with two quadratic one-sided inequality constraints, suppose that its DWR satisfies the relaxed Slater condition and yields finite optimal value and bounded optimal set. Then the objective exactness holds if  one of the optimal Lagrangian multipliers is zero.
							\end{corollary}
							\begin{proof}
								Using \autoref{them:dual},	the proof is similar to \autoref{cor:bind} and thus omitted. \qed
							\end{proof}
							
							\section{Conclusion} \label{sec:con}
							In this paper, we study  the rank-constrained optimization problem (RCOP). Replacing its  domain set including a rank constraint by the closed convex hull offers us a convex Dantzig-Wolfe Relaxation (DWR) of the RCOP. We have derived the first-known simultaneously necessary and sufficient conditions for three notions of DWR exactness. Our proposed conditions have identified,  for the first time,  the effect of the domain set on determining  the DWR exactness from a geometric angle. Specifically, the DWR exactness relies on how many faces in the closed convex hull of the domain set and its recession cone are contained in the original domain set. Our results have many potential applications, 
							e.g., we can analyze the sparse-constrained machine learning problems where the zero-norm constraint can be viewed as the rank constraint of a diagonal matrix. We are working on studying the rank bound when the domain set is defined by spectral functions and developing an efficient Dantzig-Wolfe (or equivalently, column generation) solution approach. 

\section*{Acknowledgment}
								This research has been supported in part by \exclude{the National Science Foundation grants 2246414 and 2246417, and} the Georgia Tech ARC-ACO fellowship. The authors would like to thank Prof. Boshi Yang from Clemson University	for his valuable suggestions on the first draft of the paper.

							\section*{Declarations}
							
							\noindent \textbf{Funding and/or Conflicts of interests/Competing interests:} This research has been by the National Science Foundation and the Georgia Tech ARC-ACO fellowship. Authors have no other competing interests to report.
							
							\bibliography{references.bib}
							\newpage
							\begin{appendices}
								\section{Additional Proofs}
								\subsection{Proof of \autoref{cor:gtrs}} \label{proof:corgtrs}
								\corgtrs*
								\begin{proof}
									For the GTRS problem \eqref{eq:gtrs}, its corresponding sets $\C$ and $\C_{\rel}$ are equal to
									\begin{align*}
										\C:=\left\{\bm X\in \S_+^{n+1}: \rank(\bm X)\le 1, X_{11}=1, \langle\bm A_1, \bm X\rangle \le b_1 \right\}, \ \ \C_{\rel}  :=\left\{\bm X\in \S_+^{n+1}:  X_{11}=1, \langle\bm A_1, \bm X\rangle \le b_1 \right\},
									\end{align*}
									where $\bm A_1 = \begin{pmatrix}
										0 & \bm q_1^{\top}/2\\
										\bm q_1/2 & \bm Q_1
									\end{pmatrix}$.
									
									According to representation theorem in \cite{rockafellar2015convex}, it suffices to prove that $\C_{\rel}$ and $\clconv(\C)$ share the same extreme points and recession cones in order to show their equivalence. They always have the same extreme points based on \autoref{them:egqcqp}. 
									Thus, without loss of generality, we can assume both sets are nonempty and unbounded. We will show that these two sets enjoy the same recession cone. To prove this, we observe that the recession cone of  $\C_{\rel} $ is equal to
									\[ \reccone(\C_{\rel})= \left\{\bm X\in \S_+^{n+1}:  X_{11}=0, \langle\bm A, \bm X\rangle \le 0 \right\} = \left\{\bm X\in \S_+^{n+1}:  \exists  \bm Y\in \S_+^{n},\bm X=\begin{pmatrix}
										0 & \bm 0^\top\\
										\bm 0 & \bm Y
									\end{pmatrix}, \langle \bm Q_1, \bm Y\rangle \le 0 \right\},\]
									where  the second equation is due to the fact that $X_{1j}=X_{j1}=0$ for all $j\in [n+1]$. According to Part (ii) in \autoref{them:egqcqp},  
									thus we have
									\begin{align*}
										\reccone(\C_{\rel}) &= \clconv\left( \left\{\bm X\in \S_+^{n+1}:  \exists  \bm Y\in \S_+^{n},\bm X=\begin{pmatrix}
											0 & \bm 0^\top\\
											\bm 0 & \bm Y
										\end{pmatrix},  \rank(\bm Y)\le 1, \langle \bm Q_1, \bm Y\rangle \le 0  \right\} \right) \\
										&= \clconv\left( \left\{\bm X\in \S_+^{n+1}: \rank(\bm X)\le1, \langle \bm A_1, \bm X\rangle \le 0,  X_{11}=0 \right\}\right).
									\end{align*}
									
									Thus, it remains to show that any rank-one direction in  set $\C_{\rel}$ is also a direction in set $\clconv(\C)$, which implies the recession cone equivalence. Let us denote by $\bm D  :=\begin{pmatrix}
										0 & \bm 0^{\top}\\
										\bm 0 & \bm y \bm y^{\top}
									\end{pmatrix}$ a nonzero rank-one direction in $\reccone(\C_{\rel})$. Then we  have $\langle\bm Q_1, \bm y \bm y^{\top} \rangle \leq 0$.
									There are two cases to be discussed depending on whether 
										$
										\langle\bm Q_1, \bm y \bm y^{\top} \rangle - |\bm q_1^{\top} \bm y|
										$ is equal to zero or not.
										\begin{enumerate}[(i)]
											\item Suppose that $\langle\bm Q_1, \bm y \bm y^{\top} \rangle - |\bm q_1^{\top} \bm y|  \neq 0$. Since matrix $\bm D$ is sign-invariant with respect to $\bm y$, we can let $\bm y:=-\bm y $ if $ \bm q_1^{\top} \bm y > 0$.  Thus,  we always have $\langle\bm Q_1, \bm y \bm y^{\top} \rangle + \bm q_1^{\top} \bm y < 0$. Then  there is a scalar $\bar{\gamma}\ge 1$ such that  
											\begin{align}
												\left\langle \bm A_1, \begin{pmatrix}
													1 & \gamma \bm y^{\top}\\
													\gamma \bm y & \gamma^2 \bm y \bm y^{\top}
												\end{pmatrix} \right \rangle = \left\langle \begin{pmatrix}
													0 & \bm q_1^{\top}/2\\
													\bm q_1/2 & \bm Q_1
												\end{pmatrix}, \begin{pmatrix}
													1 & \gamma \bm y^{\top}\\
													\gamma \bm y & \gamma^2 \bm y \bm y^{\top}
												\end{pmatrix} \right \rangle \le \gamma \left(\langle\bm Q_1, \bm y \bm y^{\top} \rangle + \bm q_1^{\top} \bm y\right) \le b_1, \forall \gamma \ge \bar{\gamma},
												\label{eq_feasible}
											\end{align}
											where the first inequality is due to $\langle\bm Q_1, \bm y \bm y^{\top}\rangle \le 0$.
											
											Let us define a rank-one matrix $\hat{\bm X} := \begin{pmatrix}
												1 & 2\bar{\gamma} \bm y^{\top}\\
												2\bar{\gamma} \bm y & 4\bar{\gamma}^2 \bm y \bm y^{\top}
											\end{pmatrix}$. According to the result in \eqref{eq_feasible}, we have $\hat{\bm X}  \in \C$.  For any $\alpha\ge 0$, matrix $\hat{\bm X} +\alpha\bm D$ can be written as the following convex combination 
											\begin{align*}
												\hat{\bm X} +\alpha\bm D  &= \frac{\alpha}{\alpha+\bar{\gamma}^2} \begin{pmatrix}
													1 & \bar{\gamma} \bm y^{\top}\\
													\bar{\gamma} \bm y & \bar{\gamma}^2 \bm y \bm y^{\top}
												\end{pmatrix}  + \frac{\bar{\gamma}^2}{\alpha+\bar{\gamma}^2} \begin{pmatrix}
													1 & (\alpha+2\bar{\gamma}^2)/\bar{\gamma}\bm y^{\top}\\
													(\alpha+2\bar{\gamma}^2)/\bar{\gamma}\bm y & (\alpha+2\bar{\gamma}^2)^2/\bar{\gamma}^2 \bm y \bm y^{\top} 
												\end{pmatrix}  \\
												&=  \frac{\alpha}{\alpha+\bar{\gamma}^2} \bm X_1+  \frac{\bar{\gamma}^2}{\alpha+\bar{\gamma}^2} \bm X_2,
											\end{align*}
											where both rank-one matrices $\bm X_1,\bm X_2$ belong to set $\C$ because $\bar{\gamma}, (\alpha+2\bar{\gamma}^2)/\bar{\gamma} \ge  \bar{\gamma}$, and thus the inequalities \eqref{eq_feasible} hold. It follows that $\hat{\bm X} +\alpha\bm D \in \clconv(\C)$ for any $\alpha\geq 0$, implying that $\bm D$ is a also recession direction of $\clconv(\C)$. 

											\item Suppose that $\langle\bm Q_1, \bm y \bm y^{\top} \rangle - |\bm q_1^{\top} \bm y|  = 0$, then we have
											$\bm q_1^{\top} \bm y =\langle\bm Q_1, \bm y \bm y^{\top} \rangle =0$.  
											Then let us consider two following subcases depending on whether  $\bm Q_1 \bm y=\bm 0$ holds or not.
											\begin{enumerate}[a)]
												\item Suppose that $\bm Q_1 \bm y =\bm 0$. Let $\hat{\bm X} = \begin{pmatrix}
													1 &  \hat{\bm y}^{\top}\\
													\hat{\bm y} & \hat{\bm y} \hat{\bm y}^{\top}
												\end{pmatrix} $ denote a feasible solution in set $\C$, i.e., $\langle\bm A_1, \hat{\bm X} \rangle = \langle\bm Q_1, \hat{\bm y} \hat{\bm y}^{\top} \rangle + \bm q_1^{\top} \hat{\bm y}  \le  b_1$, then for any $\alpha\geq 0$, we have
												\begin{align*}
													\hat{\bm X} +\alpha\bm D  
													&=\frac{1}{2} \begin{pmatrix}
														1 &  (\hat{\bm y}+\sqrt{\alpha}\bm{y})^{\top}\\
														\hat{\bm y}+\sqrt{\alpha}\bm{y}&(\hat{\bm y}+\sqrt{\alpha}\bm{y}) (\hat{\bm y}+\sqrt{\alpha}\bm{y})^{\top}
													\end{pmatrix}  +\frac{1}{2} \begin{pmatrix}
														1 &  (\hat{\bm y}-\sqrt{t}\bm{y})^{\top}\\
														\hat{\bm y}-\sqrt{t}\bm{y}&(\hat{\bm y}-\sqrt{\alpha}\bm{y}) (\hat{\bm y}-\sqrt{\alpha}\bm{y})^{\top}
													\end{pmatrix}\\
													&= \frac{1}{2} \bm X_1 + \frac{1}{2} \bm X_2,
												\end{align*}
												where both rank-one matrices $\bm X_1, \bm X_2$ above lie in set $\C$ since  $\langle\bm A_1, \bm X_1\rangle = \langle\bm A_1, \bm X_2\rangle = \langle\bm A_1, \hat{\bm X}\rangle$. This implies that $\bm D$ is a recession direction of $\clconv(\C)$.
												
												\item Suppose that $\bm Q_1 \bm y \neq \bm 0$. Then we can decompose $\bm{y}=\bm{y}_1+\bm{y}_2$ such that $\bm{y}_1^\top \bm Q_1 \bm{y}_1>0$, $\bm{y}_2^\top \bm Q_1 \bm{y}_2<0$, and $\bm{y}_1^\top \bm Q_1 \bm{y}_2=0$. Then let us construct a new rank-one direction $\bm D_\epsilon \in \reccone(\C_{\rel})$ for any $\epsilon >0$ as below, which satisfies $\langle\bm Q_1, (\bm y+\epsilon \bm{y}_2) (\bm y+\epsilon \bm{y}_2)^{\top}  \rangle - |\bm q_1^{\top} (\bm y+\epsilon \bm{y}_2)|  < 0$ and $\langle\bm Q_1, (\bm y+\epsilon \bm{y}_2) (\bm y+\epsilon \bm{y}_2)^{\top} \rangle  < 0$
												\begin{align*}
													\bm D_\epsilon:= \begin{pmatrix}
														0 & \bm 0^{\top}\\
														\bm 0 & (\bm y+\epsilon \bm{y}_2)(\bm y+\epsilon \bm{y}_2)^{\top}
													\end{pmatrix}.
												\end{align*} 
												Following the similar proof of Part (i), we can show that for any $\epsilon >0$, matrix $\bm D_{\epsilon}$ is a recession direction of $\clconv(\C)$. 
												Since the recession cone of set $\clconv(\C)$ is closed, which indicates that the limit $ \bm D=\lim_{\epsilon \to 0} \bm D_{\epsilon}$  is also a recession direction in set  $\clconv(\C)$.

																\end{enumerate}
																	
																	This completes the proof. \qed
																\end{enumerate}

																%
																%
															\end{proof}
															
															\subsection{Proof of \autoref{cor:diag}}\label{proof:cordiag}
															\cordiag*
															\begin{proof}
																We first prove Part (i) by contradiction and the other one follows a similar analysis. 
																
																\noindent \textbf{Part (i).}  Suppose that there is no rank-1 optimal solution to the DWR corresponding to the QCQP problem with such setting. 
																Consider a DWR optimal solution $\bm X^*\in \S_+^{n+1}$ with $\rank(\bm X^*)>1$. 
																Since $\rank(\bm X^*)>1$, there exists a $2\times 2$ principal submatrix $\bm X_{T,T}^* \in \S_+^2 $ of $\bm X^*$ satisfying $(X_{T(1),T(2)} ^*)^2< X^*_{T(1),T(1)}X^*_{T(2),T(2)}$. Without loss of generality, suppose $T=\{1,2\}$. 
																Let us construct a new solution $\hat{\bm X}$ such that $\hat{X}_{1,2}=- \sigma_{ij}\sqrt{X^*_{1,1}X^*_{2,2}}, \hat{X}_{1,1}=X^*_{1,1}, \hat{X}_{1,1}=X^*_{2,2}$ 
																and $\hat X_{ij} = X_{ij}^*$ for any $(i,j)\in ([n+1]\times [n+1])\setminus (T\times T)$. Then we consider an equivalent truncated DWR problem with a focus on the variables indexed by $T\times T$ as
																\begin{align}\label{eq:trun}
																	\text{(Truncated DWR)}  \ \   &\langle \bm A_0, \bm X^* \rangle := \min_{\bm  X \in \clconv(\mathcal{\hat X})}\left\{\langle  \hat{\bm A}_0, \bm X \rangle + c_0: 
																	\langle \hat{\bm A}_i, \bm X \rangle + c_i \le  b_i, \forall i \in [m] \right\},
																\end{align}
																where $\mathcal{\hat X}:=\{\bm X\in \S_+^{2}:  X_{12}^2 =X_{11}X_{22} \}$ and $\clconv(\mathcal{\hat X}) =\S_+^{2}$, 
																$\hat{\bm A}_0 \in  \S^{2}= (\bm A_0)_{T, T}$ and $\hat{\bm A}_i \in  \S^{2}= (\bm A_i)_{T, T}$  for each $i\in [m]$, and $c_0= \langle \bm A_0, \bm X^* \rangle - \langle \hat{\bm A}_0, \bm X^*_{T,T}\rangle $ and $c_i= \langle \bm A_i, \bm X^* \rangle - \langle \hat{\bm A}_i, \bm X^*_{T,T} \rangle $ for each $i\in [m]$.
																
																For the truncated DWR problem \eqref{eq:trun}, without loss of generality, the submatrix $\bm X^*_{T,T}$ is optimal and leads to the objective value $\langle \bm A_0, \bm X^* \rangle $. We also see that $\bm X^*_{T,T}$  is in the interior of $\clconv(\hat{\X})$ and satisfied the relaxed Slater condition. According to \autoref{them:opt}, there must exist an $\bm \mu^* \in \Re_+^m$  such that
																\[ -\hat{\bm A}_0 \in \sum_{ i \in [m]}\mu^*_i  \hat{\bm A}_i + \N_{\clconv(\hat{\X})} (\bm X^*_{T,T}) = \sum_{ i \in [m]}\mu^*_i  \hat{\bm A}_i,\]
																where the equation is because $\bm X^*_{T,T}$ is  in the interior of $\clconv(\hat{\X}))$. Next we discuss two cases depending on whether $(\hat{ A}_0)_{12}$ attains zero or not to show that $\hat{X}_{T,T}$ is also optimal to the truncated DWR \eqref{eq:trun}.
																\begin{enumerate}[(a)]
																	\item Suppose that $(\hat{ A}_0)_{12} \neq 0$. Since $(\hat{ A}_0)_{12}$ has the same sign with $(\hat{ A}_1)_{12}, \cdots, (\hat{ A}_m)_{12}$, given $\bm \mu^* \in \Re_+^m$, the optimality condition $-\hat{\bm A}_0  = \sum_{ i \in [m]}\mu^*_i  \hat{\bm A}_i$ cannot hold, a contradiction.
																	
																	\item Suppose that $(\hat{ A}_0)_{12} = 0$. In this case, changing $X^*_{12}$ does not affect the objective value. As $(X_{1,2} ^*)^2< X^*_{1,1}X^*_{2,2}$, submatrix $\hat{\bm X}_{T,T}$ is  feasible and attains the same optimal value for the truncated DWR \eqref{eq:trun}  and thus is optimal.
																	
																	Following this scheme to adjust the solution $\hat{\bm  X}$, we can either find an alternative DWR feasible rank-one matrix $\hat{\bm  X}$  such that for any $i\neq j$, $\hat{X}_{ij} = - \sigma_{ij} \sqrt{\hat{X}_{ii}\hat{X}_{jj}}$ or arrive at a contradiction as part (a). In the first case, given the condition that for each cycle $S$ of graph $\G$, we have $ \prod_{i\in [|S|]} \sigma_{S(i),S(i+1)} = (-1)^{|S|}$, following the similar argument in \cite{sojoudi2014exactness}, there exists a vector $\bm x^*$  such that $\hat{\bm  X}= \begin{bmatrix}
																		1 & \bm x^*\\
																		\bm x^* &  \bm x^* (\bm x^*)^{\top}
																	\end{bmatrix}$, implying the superiority
																	of a rank-one matrix, a contradiction.
																\end{enumerate}
																
																
																\noindent \textbf{Part (ii).} For any feasible solution $\bm X^*:= \begin{bmatrix}
																	1 & \bm x^*\\
																	\bm x^* &  \bm Y^*
																\end{bmatrix}$ to the DWR of the one-sided QCQP, 
																suppose that for some  $i\in [n]$, we have $Y^*_{ii} > (x_i^*)^2 $. Following the similar proof as \textbf{Part (i)}, we can construct a truncated problem of the DWR that only involves with variables $(x_i, Y_{ii})$ with the feasible set $\hat{\X}:=\{(x, y) \in \Re^2: x^2=y \}$. 
																Since $\bm q_0, \bm q_1, \cdots, \bm q_n$ are sign-definite, we can either obtain a contradiction or adjust $x_i^*$ and construct a pair of new optimal solution $(\sigma_i \sqrt{Y^*_{ii}},  Y^*_{ii} )$, where $\sigma_i=-1$ if $(q_0)_i,  (q_1)_i, \cdots , (q_m)_i \ge 0$ and $\sigma_i=1$, otherwise.
																It follows that there exists a rank-one optimal solution, i.e., $Y^*_{ii} = (x_i^*)^2 $ for all $i\in [n]$, a contradiction.
																\qed
															\end{proof}

\end{appendices}

\end{document}